\documentclass[11pt]{amsart}

\usepackage{amscd,latexsym,amsthm,amsfonts,amssymb,amsmath,amsxtra}
\usepackage[mathscr]{eucal}
\usepackage[pagebackref,colorlinks=true,linkcolor=blue,citecolor=blue]{hyperref}
\usepackage{xcolor}

\renewcommand\theequation{\thesection.\arabic{equation}}

\newcommand{\BC}{{\mathbb {C}}}

\newcommand{\BQ}{{\mathbb {Q}}}

\newcommand{\CH}{{\mathcal {H}}}

\newcommand{\CN}{{\mathcal {N}}}
\newcommand{\CO}{{\mathcal {O}}}

\newcommand{\CS}{{\mathcal {S}}}

\newcommand{\RG}{{\mathrm {G}}}

\newcommand{\RO}{{\mathrm {O}}}

\newcommand{\cod}{{\mathrm{cod}}}

\newcommand{\disc}{{\mathrm{disc}}}

\newcommand{\Gal}{{\mathrm{Gal}}}
\newcommand{\GL}{{\mathrm{GL}}}

\newcommand{\SL}{{\mathrm{SL}}}

\newcommand{\SO}{{\mathrm{SO}}}

\newcommand{\Sp}{{\mathrm{Sp}}}

\newcommand{\tr}{{\mathrm{tr}}}

\newcommand{\wt}{\widetilde}

\newcommand{\ul}{\underline}

\newtheorem{thm}{Theorem}[section]

\newtheorem{lem}[thm]{Lemma}
\newtheorem{prop}[thm]{Proposition}
\newtheorem {conj}[thm]{Conjecture}

\newtheorem {ques/conj}[thm]{Question/Conjecture}

\newtheorem{defn}[thm]{Definition}
\newtheorem{rmk}[thm]{Remark}

\makeatletter

\newcommand{\Rmnum}[1]{\expandafter\@slowromancap\romannumeral #1@}
\makeatother
\RequirePackage{xcolor} 
\providecommand{\DIFdeltex}[1]{} 
\providecommand{begin}{} 
\providecommand{end}{} 
\providecommand{FL}[1]{{#1}} 
\providecommand{\DIFdelFL}[1]{} 
\makeatletter 
\def\MATHBLOCKDOLLARDOLLAR[#1]]#2{$$#1]#2$$} 
\makeatother 
\RequirePackage{listings} 
\lstdefinelanguage{DIFcode}{ 
  moredelim=[il][\color{white}\tiny]{\%DIF\ <\ }, 
  moredelim=[il][\sffamily\bfseries]{\%DIF\ >\ } 
} 
\lstdefinestyle{DIFverbatimstyle}{ 
	language=DIFcode, 
	basicstyle=\ttfamily, 
	columns=fullflexible, 
	keepspaces=true 
} 
\lstnewenvironment{DIFverbatim}[1][]{\lstset{style=DIFverbatimstyle,#1}}{} 
\lstnewenvironment{DIFverbatim*}[1][]{\lstset{style=DIFverbatimstyle,showspaces=true,#1}}{} 
\begin{document}
\renewcommand{\theequation}{\arabic{equation}}
\numberwithin{equation}{section}

\title[Jiang's wavefront sets conjecture v3]{On Jiang's wavefront sets conjecture for representations in local Arthur packets}

\author[Baiying Liu]{Baiying Liu}
\address{Department of Mathematics\\
Purdue University\\
West Lafayette, IN, 47907, USA}
\email{liu2053@purdue.edu}

\author[Freydoon Shahidi]{Freydoon Shahidi}
\address{Department of Mathematics\\
Purdue University\\
West Lafayette, IN, 47907, USA}
\email{freydoon.shahidi.1@purdue.edu}





\thanks{The research of the first named author is partially supported by the NSF Grants DMS-1702218, DMS-1848058, and by the Simons Foundation: Travel Support for Mathematicians. The research of the second named author is partially supported by the NSF Grants DMS-1801273, DMS-2135021}

\begin{abstract}
This paper serves as an attempt towards the Jiang conjecture on the upper bound nilpotent orbits in the wavefront sets of representations in local Arthur packets of quasi-split classical groups, which is a natural generalization of the well-known Shahidi conjecture, reflecting the relation between the structure of wavefront sets and the local Arthur parameters.
Applying the character identities of local Arthur packets and a matching method, we reduce the study of the upper bound to certain properties of the wavefront sets of the corresponding bi-torsor representations of general linear groups, which is implied by a recent result of Atobe and Ciubotaru for split classical groups when the residue characteristic is large.
\end{abstract}

\maketitle


\section{Introduction}

Let $F$ be a non-Archimedean local field. Let $\mathrm{G}_n=\Sp_{2n}, \SO_{2n+1}, \SO_{2n}^{\alpha}$ be quasi-split classical groups, where $\alpha$ is a square class in $F$, and let $G_n=\mathrm{G}_n(F)$. 
Here, we identify a square class with the corresponding quadratic character of the Weil group $W_F$ via the local class field theory.
Then the Langlands dual groups are 
$$\widehat{\mathrm{G}}_n(\BC) = \SO_{2n+1}(\BC), \Sp_{2n}(\BC), \SO_{2n}(\BC).$$
Let ${}^L\mathrm{G}_n$ be the $L$-group of $G_n$,
$${}^L\mathrm{G}_n= 
\begin{cases}
\widehat{\mathrm{G}}_n(\BC) & \text{ when } \mathrm{G}_n=\Sp_{2n}, \SO_{2n+1},\\
\SO_{2n}(\BC) \rtimes W_F & \text{ when } \mathrm{G}_n=\SO_{2n}^{\alpha}.
\end{cases}
$$ 
In his fundamental work \cite{Art13}, Arthur introduced the local Arthur packets which are finite sets of representations of $G_n$, parameterized by local Arthur parameters. Local Arthur parameters are defined as
a direct sum of irreducible representations
$$\psi: W_F \times \SL_2(\mathbb{C}) \times \SL_2(\mathbb{C}) \rightarrow {}^L\mathrm{G}_n$$
\begin{equation}\label{lap}
  \psi = \bigoplus_{i=1}^r \phi_i \otimes S_{m_i} \otimes S_{n_i},  
\end{equation}
satisfying the following conditions:

(1) $\phi_i(W_F)$ is bounded and consists of semi-simple elements, and $\dim(\phi_i)=k_i$;

(2) the restrictions of $\psi$ to the two copies of $\SL_2(\mathbb{C})$ are analytic, $S_k$ is the $k$-dimensional irreducible representation of $\SL_2(\mathbb{C})$, and 
$$\sum_{i=1}^r k_im_in_i = N=N_n:= 
\begin{cases}
2n+1 & \text{ when } \mathrm{G}_n=\Sp_{2n},\\
2n & \text{ when } \mathrm{G}_n=\SO_{2n+1}, \SO_{2n}^{\alpha}.
\end{cases}
$$ 
Assuming the Ramanujan conjecture, Arthur (\cite{Art13}) showed that these local Arthur packets characterize the local components of square-integrable automorphic representations.
For $1 \leq i \leq r$, let $a_i=k_im_i$, $b_i=n_i$. 
Let 
$$\ul{p}(\psi) = [b_1^{a_1} b_2^{a_2} \cdots b_r^{a_r}]$$
be a partition of $N$, where without loss of generality, we assume that $b_1 \geq b_2 \geq \cdots \geq b_r$.
$\psi$ is called tempered or generic if for all $1 \leq i \leq r$, $b_i=1$. 

For each local Arthur parameter $\psi$, Arthur associated a local $L$-parameter $\phi_{\psi}$ as follows
\begin{equation}\label{apequ1}
\phi_{\psi}(w, x) = \psi\left(w, x, \begin{pmatrix}
        |w|^{\frac{1}{2}} & 0 \\
        0 & |w|^{-\frac{1}{2}}\\
\end{pmatrix}\right).
\end{equation}
Note that for any local Arthur parameter $\phi \otimes S_m \otimes S_n$,
$$\phi(w) \otimes S_m (x) \otimes S_n \left(\begin{pmatrix}
        |w|^{\frac{1}{2}} & 0 \\
        0 & |w|^{-\frac{1}{2}}\\
\end{pmatrix}
\right) = \bigoplus_{j=-\frac{n-1}{2}}^{\frac{n-1}{2}} |w|^j\phi(w) \otimes S_m (x)
.$$
Arthur also showed that $\psi \mapsto \phi_{\psi}$ is injective.
Let $\pi_{\psi}$ be the representation of $\GL_{N}(F)$ corresponding to $\phi_{\psi}$ via local Langlands correspondence, which is unitary and self-dual, and let $\wt{\pi}_{\psi}$ be its canonical extension to the bitorsor $\wt{\GL}_{N}(F) = {\GL}_{N}(F) \rtimes \wt{\theta}(N)$. See Section \ref{characlap} for the definition of  $\wt{\theta}(N)$, and for the description of $\wt{\pi}_{\psi}$ and the local Arthur packet $\wt{\Pi}_{\psi}$. 

Given a local Arthur parameter $\psi$ as in \eqref{lap}, in a series of papers (\cite{Moe06a, Moe06b, Moe09, Moe10, Moe11}), M{\oe}glin explicitly constructed each local Arthur packet $\wt{\Pi}_{\psi}$ and showed that it is of multiplicity free.
In particular, M{\oe}glin (\cite[Corollaire 4.2]{Moe09}) showed $\wt{\Pi}_{\psi_1} \cap \wt{\Pi}_{\psi_2} \neq \emptyset$ only if $\psi_1^{\Delta}=\psi_2^{\Delta}$, where
$\psi_i^{\Delta}$ is a tempered $L$-parameter, called the diagonal restriction of $\psi_i$, which is defined as 
\begin{align*}\label{def diag rest}
\begin{split}
    \psi_i^{\Delta}:  W_F \times \SL_2(\BC) & \rightarrow W_F \times \SL_2(\BC) \times \SL_2(\BC)\\
    (w,x) & \mapsto (w,x,x).
\end{split}
\end{align*}
Then, Xu (\cite{Xu17}) gave an algorithm to determine whether the representations in M{\oe}glin's construction are nonzero, and 
Atobe (\cite{Ato20}) gave a refinement on the M{\oe}glin's construction, using the new derivatives introduced by himself and M\'inguez (\cite{AM20}), which makes it relatively easier to compute the $L$-data. Recently, Atobe (\cite{Ato23}), Hazeltine, the first named author and Lo (\cite{HLL22}) independently studied the intersection problem of local Arthur packets for symplectic and split special odd  orthogonal groups, which is considered as a key step towards the local  non-tempered Gan-Gross-Prasad problem (see \cite[Conjecture 7.1, Remark 7.3]{GGP20}).

Given an irreducible representation $\pi$ of $G_n$, one important invariant is a set $\frak{n}(\pi)$ which is defined to be all the $F$-rational nilpotent orbits $\CO$ in the Lie algebra $\frak{g}_n$ of $G_n$ such that the coefficient $c_{\CO}(\pi)$ in the Harish-Chandra-Howe local expansion of the character $\Theta(\pi)$ of $\pi$ is nonzero (see \cite{HC78} and \cite{MW87}). Since nilpotent orbits $\CO$ of $G_n$ are parametrized by data $(\ul{p}, \ul{q})$, where $\ul{p}$ is partition of $2n$ (or $2n+1$ when $\mathrm{G}_n=\SO_{2n+1}$) and $\ul{q}$ is certain non-degenerate quadratic form (\cite[Section I.6]{Wal01}), we can define another important invariant $\frak{p}(\pi)$ for $\pi$ which is the set of partitions corresponding to $\frak{n}(\pi)$. Then, one can define $\frak{n}^m(\pi)$ and $\frak{p}^m(\pi)$ to be the maximal elements in $\frak{n}(\pi)$ and $\frak{p}(\pi)$, with respect to the Zariski closure ordering and the dominance ordering, respectively. We call $\frak{n}^m(\pi)$ the wavefront set and $\frak{p}^m(\pi)$ the wavefront partitions of $\pi$, respectively. 

It is a very interesting question to characterize the sets $\frak{n}(\pi)$, $\frak{p}(\pi)$, $\frak{n}^m(\pi)$, and $\frak{p}^m(\pi)$. For reductive groups, Moeglin (\cite{Moe96}) shows that nilpotent orbits in $\frak{n}^m(\pi)$ are  admissible, and Gomez-Gourevitch-Sahi (\cite{GGS21}) proves that nilpotent orbits in $\frak{n}^m(\pi)$ are  quasi-admissible; for definitions of admissibility and quasi-admissibility, see (\cite{GGS21}). In general, it is expected that nilpotent orbits in $\frak{n}^m(\pi)$ are special which are corresponding to special representations of Weyl groups via Springer correspondence. It is known that admissible nilpotent orbits of classical groups are exactly the special ones. Hence, Moeglin's result (\cite{Moe96}) implies that for classical groups, nilpotent orbits in $\frak{n}^m(\pi)$ are special. In \cite{JLS16}, the authors checked the
non-special orbits for exceptional groups; it turns out that all but 9 of them can not be in $\frak{n}^m(\pi)$ for any $\pi$, while the 9 left are still expected not to be in $\frak{n}^m(\pi)$ for any $\pi$. 


In general, characterization of the set $\frak{n}^m(\pi)$ is still widely open, though some cases are known. Examples are representations of $\GL_n(F)$, irreducible subquotients for regular principal series of $G_n$ (\cite{MW87}), tempered and anti-tempered unipotent representations of pure inner twists of $\SO_{2n+1}(F)$ (\cite{Wal18, Wal20}), certain unramified representations of split connected reductive $p$-adic groups (\cite{Oka21}), and irreducible Iwahori-spherical representations
of split connected reductive $p$-adic groups
with ``real infinitesimal characters" (\cite{CMBO21, CMBO24, CMBO25}).
It is worths to remark that recently, Tsai (\cite{Tsa24}) constructed an example of representations of $U_7(\mathbb{Q}_3)$ showing that the wavefront set $\frak{p}^m(\pi)$ may not be a singleton.

Given any local Arthur parameter $\psi$ as in \eqref{lap}, the set $\frak{p}^m({\pi}_{\psi})$ can be described as follows. See Section \ref{characlap} for details. 

\begin{thm}\label{wfslinear}
$$\frak{p}^m({\pi}_{\psi})=\{\ul{p}(\psi)^t\}.$$
\end{thm}

Taking the character expansion for the representation $\wt{\pi}_{\psi}$ of the bitorsor $\widetilde{\mathrm{GL}}_N(F)$
at the element 
\begin{equation}\label{thetaGn}
    \theta_{\widehat{\mathrm{G}}_n}= {1 } \rtimes \wt{\theta}(N) \in \wt{\GL}_{N}(F), 
\end{equation}
(see \cite{Clo87}, also see \cite[Theorem 3.2]{Kon02} and \cite[Theorems 4.20, 4.23]{Var17}), 
 we can define the sets $\frak{n}^m(\widetilde{\pi}_{\psi})$ and $\frak{p}^m(\widetilde{\pi}_{\psi})$ 
similarly. 
 {Let
}\begin{equation}{\label{eq:G-theta}
\RG_n^{\theta}:=
\begin{cases}
\SO_{2n+1}, & \text{ when } \mathrm{G}_n=\Sp_{2n},\\
\Sp_{2n}, & \text{ when } \mathrm{G}_n=\SO_{2n+1}, \SO_{2n}^{\alpha}.
\end{cases}
}\end{equation}
{Then } the connected component of the stabilizer of $\theta_{\widehat{\mathrm{G}}_n}$ in $\wt{\mathrm{GL}}_N(F)$ is  {$\RG_n^{\theta}(F)$, } and  $\frak{n}^{m}(\widetilde{\pi}_{\psi})$ consists of $F$-rational nilpotent orbits in the Lie algebra of  {$\RG_n^{\theta}(F)$} .
Then we have the following conjecture with respect to the set 
$\frak{p}^m(\widetilde{\pi}_{\psi})$.

\begin{conj}\label{wfsbitorsor}
For any $\ul{p} \in \frak{p}^m(\wt{\pi}_{\psi})$, 
 $$\dim_{\frak{g}_n^{\theta}}(\ul{p})\leq  \dim_{\frak{g}_n^{\theta}}((\ul{p}(\psi)^t)_{\RG_n^{\theta}}), $$
where $\frak{g}_n^{\theta}$ is the Lie algebra of $\RG_n^{\theta}$, 
and $(\ul{p}(\psi)^t)_{\RG_n^{\theta}}$ is the $\RG_n^{\theta}$ -collapse of the partition 
$\ul{p}(\psi)^t$,  namely, the largest $\RG_n^{\theta}$-partition smaller than or equal to $\ul{p}(\psi)^t$. 
\end{conj}

We also believe that the following stronger conjecture holds. 

\begin{conj}\label{wfsbitorsor2}
 $$\frak{p}^m(\wt{\pi}_{\psi})= {\{(}\ul{p}{(\psi)^t)_{\RG_n^{\theta}}\}.} $$
\end{conj}

\begin{rmk}\label{rmktoconjecture1.3}

(1).  {When $(\ul{p}(\psi)^t)_{\RG_n^{\theta}}=\ul{p}(\psi)^t$} , i.e., $\ul{p}(\psi)^t$ is already a  {$\RG_n^{\theta}$} -partition, by
Theorem \ref{wfslinear},  
\cite[Theorem 4.1 (1)]{Kon02}, and \cite[Lemma 5.29]{Var17}, for any $\ul{p} \in \frak{p}^m(\wt{\pi}_{\psi})$, 
 $\ul{p}\leq  \ul{p}(\psi)^t_{\RG_n^{\theta}}$, hence, Conjecture \ref{wfsbitorsor} holds. 
This covers a large family of local Arthur parameters. Note that 
$\ul{p}(\psi)$ is automatically a  {partition of the appropriate type for the Arthur parameter, and } by \cite[Proposition 6.3.7, Theorem 6.3.11]{CM93},  {the condition that } $\ul{p}(\psi)^t$ is  {already a $\RG_n^{\theta}$} -partition  {can be checked explicitly from } $\ul{p}(\psi)$ .

When  {$(\ul{p}(\psi)^t)_{\RG_n^{\theta}}\neq \ul{p}(\psi)^t$} , i.e., $\ul{p}(\psi)^t$ is not a  {$\RG_n^{\theta}$} -partition, Conjecture \ref{wfsbitorsor} and \ref{wfsbitorsor2} are expected to be more complicated.

(2). Assume that $\psi$ is tempered, that is for all $1 \leq i \leq r$, $b_i=1$. Then $\ul{p}_{\psi}^t=[N]$.
When $\RG_n=\SO_{2n+1}, \Sp_{2n}$, 
Conjecture \ref{wfsbitorsor2}, hence Conjecture \ref{wfsbitorsor}, has been proved by \cite[Theorem 4.1 (1)]{Kon02} and by \cite[Corollary 6.16]{Var17}. 
When, $\mathrm{G}_n= \SO_{2n}^{\alpha}$, Conjecture \ref{wfsbitorsor2}, hence Conjecture \ref{wfsbitorsor}, has been proved  by \cite[Corollary 6.16]{Var17}.
\end{rmk}



In this paper, we focus on studying the structure of $\frak{p}^m(\pi)$ for irreducible representations $\pi$ of $G_n$ in a local Arthur packet $\widetilde{\Pi}_{\psi}$. It is known that for tempered local Arthur parameters, namely, $b_i=1$ for all $i$, the local Arthur packet is exactly the $L$-packet $\Pi_{\phi_{\psi}}$ corresponding the $L$-parameter $\phi_{\psi}$. For tempered $L$-packets, the second named author has the following conjecture in general.

\begin{conj}[Shahidi Conjecture]\label{shaconj}
For any quasi-split reductive group $G$, tempered $L$-packets have generic members. 
\end{conj}

This conjecture has been proved for quasi-split classical groups in \cite[Proposition 8.3.2]{Art13}, \cite[Corollary 9.2.4]{Mok15}, based on the global Langlands functoriality \cite{CKPSS04, CPSS11}, \cite{KK04, KK05},  and the automorphic descent \cite{GRS11}; see also the work of \cite{JNQ10}, \cite{JS12}, and \cite{ST15}, via the method of local descent.  Motivated by \cite[Theorem 6.2 and Conjecture 6.5]{Sha11}, Conjecture \ref{shaconj} can be enhanced as follows.

\begin{conj}[Enhanced Shahidi Conjecture]\label{shaconj2}
For any quasi-split reductive group $G$, local Arthur packets are tempered if and only if they have generic members.  
\end{conj}

We prove the enhanced Shahidi conjecture assuming Conjecture \ref{wfsbitorsor} (see Theorem \ref{mainintro}(3-a) below); and unconditionally when $p$ is large and $G$ is split (see Theorem \ref{mainintro}(2) below). We remark that for symplectic and split special odd  orthogonal groups, Hazeltine, the first named author and Lo (\cite{HLL22}) have proved Conjecture \ref{shaconj2} without any assumption, using Atobe's refinement on M{\oe}glin's construction of local Arthur packets, which is a different method from that of this paper. Hazeltine, the first named author, Lo, and Zhang (\cite{HLLZ22}) provided a framework towards Conjecture \ref{shaconj2} for general connected reductive groups, assuming the closure ordering conjecture of local Arthur packets. 

The main goal of this paper is to consider the following conjecture of Jiang which is a natural generalization of the Shahidi conjecture above from tempered local Arthur packets to non-tempered ones, on the characterization of the set $\frak{p}^m(\pi)$ for $\pi$ in local Arthur packets. Note that for a generic representation $\pi$, $\frak{p}^m(\pi)$ contains only regular nilpotent orbits. The global version of this conjecture can be found in \cite[Conjecture 4.2]{Jia14}. We now state Jiang's conjecture as follows. 

\begin{conj}[Jiang Conjecture]\label{cubmfclocal}
Let $\psi$ be a local Arthur parameter of $G_n$ as in \eqref{lap}, and let $\wt{\Pi}_{\psi}$ be the local Arthur packet attached to $\psi$. Then the following hold.
\begin{enumerate}
\item[(1)] For any partition 
$\ul{p}$ which is not related to $\eta_{{\hat{\frak{g}}_n,\frak{g}_n}}(\ul{p}(\psi))$ and any 
$\pi\in\wt{\Pi}_{\psi}$,
$\ul{p}\notin\frak{p}^m(\pi)$.
\item[(2)] For any partition
$\ul{p}> \eta_{{\hat{\frak{g}}_n,\frak{g}_n}}(\ul{p}(\psi))$ and any 
$\pi\in\wt{\Pi}_{\psi}$,
$\ul{p}\notin\frak{p}^m(\pi)$.
\item[(3)] There exists at least one member $\pi\in\wt{\Pi}_{\psi}$ having the property that
$\eta_{{\hat{\frak{g}}_n,\frak{g}_n}}(\ul{p}(\psi))\in \frak{p}^m(\pi)$.
\end{enumerate}
Here $\eta_{{\hat{\frak{g}}_n,\frak{g}_n}}$ denotes the Barbasch-Vogan duality map from the partitions for the dual group $\widehat{\mathrm{G}}_n(\BC)$ to
the partitions for $G_n$ $($see \cite{BV85} and \cite{Ach03}$)$.
\end{conj}

Recall that given a partition $\ul{p}$ of $M$, where
$$M:= 
\begin{cases}
2n & \text{ when } \mathrm{G}_n=\Sp_{2n}, \SO_{2n}^{\alpha}, \\
2n+1 & \text{ when } \mathrm{G}_n=\SO_{2n+1},
\end{cases}
$$ 
$\ul{p}_{\RG_n}$ denotes the $\RG_n$-collapse of $\ul{p}$, that is, the largest $\RG_n$-partition smaller than or equal to $\ul{p}$, and $\ul{p}^{\RG_n}$ denotes the $\RG_n$-expansion of $\ul{p}$, that is, the smallest special $\RG_n$-partition bigger than or equal to $\ul{p}$. Given any partition $\underline{p}=[p_1p_2\cdots p_r]$ with $p_1 \geq p_2 \geq \cdots \geq p_r$, let 
$\underline{p}^{-}=[p_1p_2\cdots (p_r-1)]$ and 
$\underline{p}^{+}=[(p_1+1)p_2\cdots p_r]$. 
For the definition of special partitions, see \cite[Section 6.3]{CM93}. 
We recall the definition of the Barbasch-Vogan duality map from \cite{BV85}  and \cite{Ach03} as follows.

\begin{defn}\label{BV duality map}
\begin{enumerate}
    \item [(i)]For any partition $\underline{p}=[p_1p_2\cdots p_r]$ of $2n+1$ of orthogonal type (i.e., even parts with even multiplicities), $p_1 \geq p_2 \geq \cdots \geq p_r$, we define 
    $$\eta_{{\hat{\frak{g}}_n,\frak{g}_n}}(\underline{p}):= ((\underline{p}^{-})_{\Sp_{2n}})^{t}=((\underline{p}^{t})^{-})_{\Sp_{2n}},$$
    where $\frak{g}_n=\frak{sp}_{2n}$. 

    \item [(ii)]For any partition $\underline{p}=[p_1p_2\cdots p_r]$ of $2n$ of symplectic type (i.e., odd parts with even multiplicities), $p_1 \geq p_2 \geq \cdots \geq p_r$, we define $$\eta_{{\hat{\frak{g}}_n,\frak{g}_n}}(\underline{p}):= ((\underline{p}^{+})_{\SO_{2n+1}})^t=((\underline{p}^{t})^{+})_{\SO_{2n+1}},$$
    where $\frak{g}_n=\frak{so}_{2n+1}$.

    \item [(iii)] For any partition $\underline{p}=[p_1p_2\cdots p_r]$ of $2n$ of orthogonal type (i.e., even parts with even multiplicities), $p_1 \geq p_2 \geq \cdots \geq p_r$, we define $$\eta_{{\hat{\frak{g}}_n,\frak{g}_n}}(\underline{p}):= (\underline{p}^t)_{\SO_{2n}},$$
    where $\frak{g}_n=\frak{so}_{2n}$.
\end{enumerate}
\end{defn}

Our main results are summarized in the following theorem.

\begin{thm}\label{mainintro}
Let $\psi$ be a local Arthur parameter as in \eqref{lap}, with $\ul{p}(\psi) = [b_1^{a_1} b_2^{a_2} \cdots b_r^{a_r}]$ and $b_1 \geq b_2 \geq \cdots \geq b_r$. Then, we have the following. 

\begin{enumerate}
    \item (Theorem \ref{mainintropart1}).
    Conjecture \ref{wfsbitorsor} is true if and only if for any 
$\pi\in\wt{\Pi}_{\psi}$,
$\ul{p}\in\frak{p}^m(\pi)$, $$\dim_{{\frak{g}}_n} (\ul{p}) \leq \dim_{{\frak{g}}_n} (\eta_{{\hat{\frak{g}}_n,\frak{g}_n}}(\ul{p}(\psi))).$$ In particular, the latter implies that 
Conjecture \ref{cubmfclocal}(2) is valid.
\item For split $\RG_n$, Conjecture \ref{wfsbitorsor} holds when $p$ is large as follows:
    \begin{enumerate}
        \item $p > 6n+3$, when $\RG_n=\SO_{2n+1}$;
        \item $p > 6n$, when $\RG_n=\Sp_{2n}$;
        \item $p > 6n$, when $\RG_n=\SO_{2n}$.
    \end{enumerate}
\item Assume that Conjecture \ref{wfsbitorsor} is true. Then we have the following. 
\begin{enumerate}
\item (Theorem \ref{mainintropart2}).
Conjecture \ref{shaconj2} is true.
\item (Theorem \ref{mainintropart3}). 
Let $$\underline{p}_1=\left[\big\lfloor \frac{b_1}{2} \big\rfloor^{a_1}
\big\lfloor \frac{b_2}{2} \big\rfloor^{a_2} \cdots \big\lfloor \frac{b_r}{2} \big\rfloor^{a_r}\right]^t,$$
and $n^{*}=\big\lfloor\frac{\sum_{b_i \text{ odd }} a_i}{2}\big\rfloor$. 
    Then Conjecture \ref{cubmfclocal}(3) holds for the following cases.
    \begin{enumerate}
        \item When $\mathrm{G}_n=\Sp_{2n}$, and
        \begin{equation}\label{criterion_intro1}
    ([\underline{p}_1\underline{p}_1(2n^*)]^t)_{\Sp_{2n}}=([b_1^{a_1} \cdots b_r^{a_r}]^-)_{\Sp_{2n}}.
\end{equation}
In particular, if 
\begin{enumerate}
    \item $a_r=b_r=1$ and $b_i$ are all even for $1 \leq i \leq r-1$,
    \item or, $b_i$ are all odd,
\end{enumerate}
 then \eqref{criterion_intro1} holds and thus Conjecture \ref{cubmfclocal}(3) is valid. 
        \item When $\mathrm{G}_n=\SO_{2n+1}$, and
        \begin{equation}\label{criterion_intro2}
    ([\underline{p}_1\underline{p}_1(2n^*+1)]^t)_{\SO_{2n+1}}=([b_1^{a_1} \cdots b_r^{a_r}]^+)_{\SO_{2n+1}}.
\end{equation}
        In particular, if \begin{enumerate}
            \item $b_1$ is even and $a_1=1$, and $b_i$ are all odd for $2 \leq i \leq r$,
            \item or, $b_i$ are all even, 
        \end{enumerate}
        then \eqref{criterion_intro2} holds  and thus Conjecture \ref{cubmfclocal}(3) is valid.
        \item When    $\mathrm{G}_n=\SO_{2n}^{\alpha}$, and
         \begin{equation}\label{criterion_intro3}
    [\underline{p}_1\underline{p}_1(2n^*-1)1]^{\SO_{2n}}=([b_1^{a_1} \cdots b_r^{a_r}]^t)_{\SO_{2n}}.
 \end{equation}
        If all $b_i$ are of the same parity, then \eqref{criterion_intro3} holds  and thus Conjecture \ref{cubmfclocal}(3) is valid. 
    \end{enumerate}

\end{enumerate}
\end{enumerate}
\end{thm}

\begin{rmk} \
\begin{enumerate}
\item Theorem \ref{mainintro}(1) is proved by
applying a 
generalization of the matching method used in \cite[Section 9]{Sha90} to the endoscopic character identity in \cite{Art13}. 

    \item For split $\RG_n$, applying the equivalence in Theorem \ref{mainintro}(1), Theorem \ref{mainintro}(2) is directly implied by \cite[Theorem 1.4(2)]{AC26}. Indeed, \cite{AC26} claims a proof of Conjecture \ref{cubmfclocal} for split $G_n$ under certain hypothesis, applying similar matching methods.
    \item For Theorem \ref{mainintro}(3-b), we  explicitly construct a member $\sigma$ in $\wt{\Pi}_{\psi}$ such that $\eta_{{\hat{\frak{g}}_n,\frak{g}_n}}(\underline{p}(\psi)) \in \ul{p}(\sigma)$, however, to show that $\eta_{{\hat{\frak{g}}_n,\frak{g}_n}}(\underline{p}(\psi)) \in \ul{p}^m(\sigma)$, we need Conjecture \ref{wfsbitorsor}.
The equalities \eqref{criterion_intro1} -  \eqref{criterion_intro3} guarantee that $\eta_{{\hat{\frak{g}}_n,\frak{g}_n}}(\underline{p}(\psi)) \in \ul{p}(\sigma)$. 
When $b_i$ are of mixed parities, then the equalities \eqref{criterion_intro1} -- \eqref{criterion_intro3} may not always hold, see Remark \ref{mixedpartities} for examples.
\end{enumerate}
\end{rmk}

This paper serves as an initial step towards the deep conjecture of Jiang (Conjecture \ref{cubmfclocal}) on the wavefront sets of representations in general local Arthur packets, through the endoscopic character identities. Jiang's conjecture (Conjecture \ref{cubmfclocal}) has analogs for quasi-split unitary groups and non-quasi-split classical groups. 
The method used in this paper is expected to extend to these cases (as long as the local Arthur classification is carried out).

We remark that Okada (\cite{Oka21}) has computed the wavefront set of irreducible unramified representations
of split connected reductive $p$-adic groups with local Arthur parameter $\psi$ being trivial on the Weil-Deligne group. Ciubotaru, Mason-Brown, and Okada (\cite{CMBO24, CMBO25}) have computed the wavefront set of irreducible Iwahori-spherical representations
of split connected reductive $p$-adic groups
with ``real infinitesimal characters".
Waldspurger (\cite{Wal18, Wal20}) has computed the wavefront set of 
tempered and anti-tempered unipotent representations of pure inner twists of $\SO_{2n+1}(F)$.  
Combining these results on unipotent representations with the closure ordering result in \cite{HLLZ22}, jointly with Hazeltine and Lo (\cite{HLLS24}), we proved the following theorem towards Conjecture \ref{cubmfclocal}. 

\begin{thm}{\cite[Theorem 11.4]{HLLS24}}\label{unipotent jiang conjecture}
    Let $\RG_n=\Sp_{2n}, \SO_{2n+1}$. Then Conjecturre \ref{cubmfclocal} is true if $\psi$ is trivial on $W_F$.
\end{thm}

We remark that, together with Hazeltine and Lo (\cite{HLLS24}), we also proved the following very interesting reduction in Conjecture \ref{cubmfclocal}(1-2). 

\begin{thm}{\cite[Theorem 1.5]{HLLS24}}\label{thm reduction intro}
    The following statements are equivalent.
\begin{enumerate}
    \item  Conjecture \ref{cubmfclocal}(1-2) holds for any local Arthur parameter.
    \item Conjecture \ref{cubmfclocal}(1-2) holds for any anti-tempered local Arthur parameter.
    \item Conjecture \ref{cubmfclocal}(1-2) holds for any anti-discrete  local Arthur parameter.
\end{enumerate}
\end{thm}



Therefore, combining with Theorem \ref{unipotent jiang conjecture}, we obtain the following result towards Conjecture \ref{wfsbitorsor}.

\begin{thm}\label{unipotent bitorsor conjecture}
    Let $\RG_n=\Sp_{2n}, \SO_{2n+1}$. 
    Assume that $\psi$ is trivial on $W_F$. 
    Then Conjecture \ref{wfsbitorsor} holds. 
\end{thm}

The following is the structure of this paper. In Section \ref{characlap}, we recall the characterization of the local Arthur packets in \cite[Section 2.2]{Art13} and prove Theorem \ref{wfslinear}. In Section
\ref{def of FC's}, we recall certain twisted Jacquet modules associated to nilpotent orbits, following the formulation in \cite{GGS17}. In Section \ref{dimidentities}, we prove certain dimension identities for nilpotent orbits, which are important ingredients for our main results. In Section \ref{nonvan},
we construct a particular element in each local Arthur packet and study its wavefront set. In Section \ref{vanandmain}, we prove our main result Theorems \ref{mainintro}.
In Appendix \ref{appendix}, together with Alexander Hazeltine and Chi-Heng Lo, we show that all representations in any local Arthur packet share the same central character, which has its own interests. 

\subsection*{Acknowledgements} 
The authors would like to thank Dihua Jiang for his interests and helpful discussions. The authors would like to thank Wee-Teck Gan, Tasho Kaletha, Colette M{\oe}glin, David Vogan, and Bin Xu for helpful communications. The authors also would like to thank the Fields Institute and the Erwin Schrödinger International Institute
for Mathematics and Physics (ESI) for the support on workshops and conferences where the results of this paper were presented (2021 and 2025).

\section{Characterization of local Arthur packets}\label{characlap}

In this section, we review the characterization of local Arthur packets as in \cite[Section 2.2]{Art13} and prove Theorem \ref{wfslinear}.

Let $\theta$ be the standard outer automorphism of $\mathrm{G}(N)=\GL(N)$: $g \mapsto {}^t g^{-1}$ and let 
$\tilde{\theta}(N)=Int(\tilde{J})\circ \theta: g \mapsto \tilde{J} \theta(g) \tilde{J}^{-1}$, where 
$$
\tilde{J}=\tilde{J}(N)=
\begin{pmatrix}
0&&&1\\
&&-1&\\
&\dots&&\\
(-1)^{N-1}&&&0
\end{pmatrix}.
$$
Recall that 
$$N= N_n=
\begin{cases}
2n+1 & \text{ when } \mathrm{G}_n=\Sp_{2n},\\
2n & \text{ when } \mathrm{G}_n=\SO_{2n+1}, \SO_{2n}^{\alpha}.
\end{cases}
$$ 
Let $\widetilde{\mathrm{G}}^+(N)=\mathrm{G}(N) \rtimes \langle \tilde{\theta}(N) \rangle$. Let $\widetilde{\mathrm{G}}^0(N)=\mathrm{G}(N) \rtimes 1$ be the identity component of $\widetilde{\mathrm{G}}^+(N)$, and $\widetilde{\mathrm{G}}(N)=\mathrm{G}(N) \rtimes \tilde{\theta}(N)$ be the other connected component. 
Fix a $\wt{\theta}(N)$-stable Whittaker datum $(B(N),\chi(N))$ for $\widetilde{\mathrm{G}}^0(N)(F)$ by taking $B(N)$ to be the standard Borel subgroup of $\widetilde{\mathrm{G}}^0(N)(F)=\mathrm{GL}(N)(F)$, and $\chi(N)$ to be the nondegenerate character on the unipotent radical of $B(N)$ as follows:
$$\chi(N)(u)=\psi_F(u_{1,2}+\cdots+u_{N-1,N}),$$
where $\psi_F$ is a fixed nontrivial additive character of $F$. 
Given an irreducible admissible self-dual representation $(\pi,V)$ of $G(N)=\mathrm{G}(N)(F)$, as in \cite[Section 2.2]{Art13}, we extend it to a representation $\tilde{\pi}$ of $\widetilde{G}^+(N)=\widetilde{\mathrm{G}}^+(N)(F)$ as follows. First assume that $\pi$ is tempered, hence it has nontrivial Whittaker functionals. Since $\pi$ is self-dual, $\pi$ is isomorphic to $\pi \circ \tilde{\theta}(N)$, 
choose a nontrivial intertwining operator $\tilde{I}$ from $\pi$ to $\pi \circ \tilde{\theta}(N)$. Fix a nontrivial $(B(N),\chi(N)$-Whittaker functional $\omega$ for $\pi$, then $\omega \circ \tilde{I}$ is also a nontrivial $(B(N),\chi(N)$-Whittaker functional for $\pi$, 
by the uniqueness of Whittaker functionals, $\omega \circ \tilde{I}= c \omega$, for some nonzero constant $c$. Set $\tilde{\pi}(N)=\pi(\tilde{\theta}(N))=c^{-1} \circ \tilde{I}$. Then $\tilde{\pi}(N)$ is the unique intertwining operator from $\pi$ to $\pi \circ \tilde{\theta}(N)$ such that $\omega = \omega \circ \tilde{\pi}(N)$. Therefore, we obtain a unitary extension $\tilde{\pi}$ of $\pi$ to $\widetilde{G}^+(N)$ and hence to $\widetilde{G}(N)=\widetilde{\mathrm{G}}(N)(F)$. 
In general, if $\pi$ is the Langlands quotient of a standard  representation which is induced from a twist of a tempered representation, then one can first extend the tempered representation  as above, then continue extending to a representation $\tilde{\pi}$ of $\widetilde{G}^+(N)$ by parabolic induction and hence to $\widetilde{G}(N)$. 

Given an Arthur parameter $\psi$ as in \eqref{lap}, let $\pi_{\psi}$ be the representation of $G(N)$ corresponding to $\phi_{\psi}$ via local Langlands correspondence, which is unitary and self-dual. And let $\wt{\pi}_{\psi}$ be its canonical extension to the bitorsor $\widetilde{G}(N)$. 
Let $K_N=\GL(\frak{o}_F)$ and $K_{G_n}=\mathrm{G}_n(\frak{o}_F)$. 
Let $\tilde{\mathcal{H}}(N)$ be the corresponding Hecke algebra of $\GL_N(F)$ and $\tilde{\mathcal{H}}(G_n)$ be the ${\mathrm{\tilde{O}ut}}_N(G_n)$-invariant functions in the Hecke algebra $\mathcal{H}(G_n)$ of $G_n$. 
Let ${\mathcal{S}}(G_n)$ be the set of stable transfers ${f}^{G_n}$ of ${f} \in \mathcal{H}(G_n)$ and let $\tilde{\mathcal{S}}(G_n)$ be the subspace of the ${\mathrm{\tilde{O}ut}}_N(G_n)$-invariant functions in ${\mathcal{S}}(G_n)$.
For $\tilde{f} \in \tilde{\mathcal{H}}(N)$, let 
\begin{equation}\label{gltr}
\tilde{f}_N(\psi)=\mathrm{tr}(\tilde{\pi}_{\psi}(\tilde{f})) {{} = \mathrm{tr} ({\pi}_{\psi}(\tilde{f}) \circ \widetilde{\pi}_{\psi}(\widetilde{\theta}(N)))}.
\end{equation}
Arthur shows the following theorem on the characterization of the local Arthur packet 
$\widetilde{\Pi}_{\psi}$ attached to the Arthur parameter $\psi$, via transferring the linear from $\tilde{f}_N(\psi)$ to twisted endoscopic groups. 

\begin{thm}\cite[Theorem 2.2.1]{Art13}\label{arthurclassification}
For any $\widetilde{f} \in \widetilde{\mathcal{H}}(N)$, 
$$\wt{f}_N(\psi)=\sum_{\pi \in \widetilde{\Pi}_{\psi}} \langle s_{\psi}, \pi \rangle f_{G_n}(\pi),$$
as two stable distributions on $\widetilde{\mathcal{H}}(N)$, where $f$ is any lifting of the transfer $\wt{f}^{G_n} \in \wt{\mathcal{S}}(G_n)$ to $\wt{\mathcal{H}}(G_n)$,
$f_{G_n}(\pi)=\mathrm{tr}(\pi(f))$, and 
$s_{\psi}=\psi\left(1, 1,  \begin{pmatrix}
-1 &0 \\
0&-1
\end{pmatrix}\right)$. 
\end{thm}

At the end of this section, we give the proof for Theorem \ref{wfslinear}. 

\textbf{Proof of Theorem \ref{wfslinear}}.
Recall from \eqref{lap} that the local Arthur parameter is as follows
$$\psi: W_F \times \SL_2(\mathbb{C}) \times \SL_2(\mathbb{C}) \rightarrow {}^L\mathrm{G}_n$$
$$\psi = \bigoplus_{i=1}^r \phi_i \otimes S_{m_i} \otimes S_{n_i}.$$
Then 
$$\phi_{\psi}(w,x)=\bigoplus_{i=1}^r 
\bigoplus_{j=-\frac{n_i-1}{2}}^{\frac{n_i-1}{2}} |w|^j\phi_i(w) \otimes S_{m_i} (x).$$

For $1 \leq i \leq r$, let $a_i=k_im_i$ and $b_i=n_i$, and let 
$$\Delta_i=\delta[v^{-\frac{m_i-1}{2}}r(\phi_i), v^{\frac{m_i-1}{2}}r(\phi_i)]$$ be the irreducible square-integrable representation attached to the balanced segment 
$[v^{-\frac{m_i-1}{2}}r(\phi_i), v^{\frac{m_i-1}{2}}r(\phi_i)]$, namely, $\Delta_i$ is the unique irreducible subrepresentation of the following induced representation
$$v^{\frac{m_i-1}{2}}r(\phi_i) \times v^{\frac{m_i-2}{2}}r(\phi_i)\times \cdots \times v^{\frac{1-m_i}{2}}r(\phi_i).$$
Here the map $r$ is the local Langlands correspondence for general linear groups. 
Let $\zeta(\Delta_i, b_i)$ be the unique irreducible quotient of the following induced representation
$$v^{\frac{b_i-1}{2}}\Delta_i \times v^{\frac{b_i-3}{2}}\Delta_i \times \cdots 
\times 
v^{\frac{1-b_i}{2}}\Delta_i.$$
Then the representation $\pi_{\psi}$ corresponding to $\phi_{\psi}$ under the local Langlands correspondence is as follows
$$\pi_{\psi} = \times_{i=1}^r \zeta(\Delta_i, b_i).$$

Note that $\frak{p}^m(\zeta(\Delta_i, b_i))=\{[a_i^{b_i}]\}$, and by \cite[Section 2.1]{MW87}, to compute $\frak{p}^m(\pi_{\psi})$, we only need to compute the induced orbit, which is 
\MATHBLOCKDOLLARDOLLAR[[a_1^{b_1}]]{ + \cdots + [a_r^{b_r}],}
by \cite[Lemma 7.2.5]{CM93}. 
Since 
\MATHBLOCKDOLLARDOLLAR[[a_1^{b_1}]]{ + \cdots + [a_r^{b_r}] =[b_1^{a_1} \cdots b_r^{a_r}]^t,}  
we have shown that 
$$\frak{p}^m({\pi}_{\psi})=\{[b_1^{a_1} \cdots b_r^{a_r}]^t\}.$$
This completes the proof of Theorem \ref{wfslinear}.  \qed

\section{Twisted Jacquet modules associated to nilpotent orbits}\label{def of FC's}


In this section, we recall certain twisted Jacquet modules associated to nilpotent orbits, following the formulation in \cite{GGS17}.

Let $\RG$ be a reductive group defined over a nonarchimedean local field $F$, and $\mathfrak{g}$ be the Lie algebra of $G=\RG(F)$. 
Given any semi-simple element $s \in \mathfrak{g}$, under the adjoint action, $\mathfrak{g}$ is decomposed into a direct sum of eigenspaces $\mathfrak{g}^s_i$ corresponding to eigenvalues $i$.
The element
$s$ is called {\it rational semi-simple} if all its eigenvalues are in $\BQ$.
Given a nilpotent element $u$ and a simi-simple element $s$ in $\mathfrak{g}$, the pair $(s,u)$ is called a {\it Whittaker pair} if $s$ is a rational semi-simple element, and $u \in \mathfrak{g}^s_{-2}$. The element $s$ in a Whittaker pair $(s, u)$ is called a {\it neutral element} for $u$ if there is a nilpotent element $v \in \mathfrak{g}$ such that $(v,s,u)$ is an $\mathfrak{sl}_2$-triple. A Whittaker pair $(s, u)$ with $s$ being a neutral element is called a {\it neutral pair}. 

Given any Whittaker pair $(s,u)$, define an anti-symmetric form $\omega_u$ on $\mathfrak{g}\times \mathfrak{g}$ by 
$$\omega_u(X,Y):=\kappa(u,[X,Y])\,,$$
here $\kappa$ is the Killing form on $\mathfrak{g}$.
For any rational number $r \in \BQ$, let $\mathfrak{g}^s_{\geq r} = \oplus_{r' \geq r} \mathfrak{g}^s_{r'}$. 
 Let $\mathfrak{u}_s= \mathfrak{g}^s_{\geq 1}$ and let $\mathfrak{n}_{s,u}$ be the radical of $\omega_u |_{\mathfrak{u}_s}$. Then $[\mathfrak{u}_s, \mathfrak{u}_s] \subset \mathfrak{g}^s_{\geq 2} \subset \mathfrak{n}_{s,u}$. 
For any $X \in \mathfrak{g}$, let $\mathfrak{g}_X$ be the centralizer of $X$ in $\mathfrak{g}$.  
By \cite[Lemma 3.2.6]{GGS17}, one has $\mathfrak{n}_{s,u} = \mathfrak{g}^s_{\geq 2} + \mathfrak{g}^s_1 \cap \mathfrak{g}_u$.
Note that if the Whittaker pair $(s,u)$ comes from an $\mathfrak{sl}_2$-triple $(v,s,u)$, then $\mathfrak{n}_{s,u}=\mathfrak{g}^s_{\geq 2}$. Let 
$N_{s,u}=\exp(\mathfrak{n}_{s,u})$ be the corresponding unipotent subgroup of $G$. 
We define a character of $N_{s,u}$ by $$\psi_u(n)=\psi(\kappa(u,\log(n)))\,,$$
here $\psi: F \rightarrow \BC^{\times}$ is a fixed non-trivial additive character.

Let $\pi$ be an irreducible admissible representation of $G$. The {\it twisted Jacquet module} of $\pi$ associated to 
a Whittaker pair $(s,u)$ is defined to be  
$\pi_{N_{s,u}, \psi_u}$.
Let $\mathfrak{n}_{w}(\pi)$ be the set of nilpotent orbits $\CO$ such that $\pi_{N_{s,u}, \psi_u}$ is non-zero for some neutral pair $(s,u)$ with $u \in \CO$. Note that if $\pi_{N_{s,u}, \psi_u}$ is non-zero for some neutral pair $(s,u)$ with $u \in \CO$, then it is non-zero for any such neutral pair $(s,u)$, since the non-vanishing property of such models doesn't depend on the choices of representatives of $\CO$.
Moreover, we denote by $\mathfrak{n}_w^m(\pi)$ the set of maximal elements in $\mathfrak{n}_w(\pi)$ under the natural partial ordering of nilpotent orbits (i.e., $\CO_1 \leq \CO_2$ if $\CO_1 \subset \overline{\CO_2}$, the Zariski closure of $\CO_2$). Then we have the following relation between $\mathfrak{n}^m(\pi)$ and $\mathfrak{n}_w^m(\pi)$.

\begin{thm}[\cite{MW87}, \cite{Var14}]
For any irreducible admissible representation $\pi$ of $G$,
$$\mathfrak{n}^m(\pi)= \mathfrak{n}_w^m(\pi).$$
\end{thm}

We recall \cite[Theorem C]{GGS17} as follows, which will be used in Section 5. 

\begin{prop}[Theorem C, \cite{GGS17}]\label{ggslocal}
Let $\pi$ be an irreducible admissible representation of $G$.
Given a neutral pair $(s,u)$ and a Whittaker pair $(s',u)$, if $\pi_{N_{s',u}, \psi_u}$ is non-zero, then $\pi_{N_{s,u}, \psi_u}$ is non-zero.
\end{prop}

\section{Dimension identities for nilpotent orbits}\label{dimidentities}

In this section, we prove certain dimension identities for nilpotent orbits which are important ingredients for our main results.

Recall that $\psi = \bigoplus_{i=1}^r \phi_i \otimes S_{m_i} \otimes S_{n_i}$, where $\phi_i$ is of dimension $k_i$, $a_i=k_im_i$, $b_i=n_i$, and $N=\sum_{i=1}^r a_ib_i$. 
For $1 \leq i \leq r$, let $\psi_i=\phi_i \otimes S_{m_i} \otimes S_{n_i}$. Let $\{1,2,\ldots, r\}=I \dot\cup J$ such that $I=\{i|s_{\psi_i}=I_{a_ib_i}\}$, $J=\{i|s_{\psi_i}=-I_{a_ib_i}\}$.
By definition, $s_{\psi}=\psi\left(1,1, \begin{pmatrix}
-1 &0 \\
0&-1
\end{pmatrix}\right)$, and
\begin{align*}
    &\,\phi_i \otimes S_{m_i} \otimes S_{b_i}
    \left(1,1, \begin{pmatrix}
-1 &0 \\
0&-1
\end{pmatrix}\right)\\
=&\,\begin{pmatrix}
(-1)^{b_i-1}I_a&&&\\
&(-1)^{b_i-3}I_a&&\\
&&\ddots&\\
&&&(-1)^{1-b_i}I_a
\end{pmatrix}.
\end{align*}
Hence, 
$I=\{i|b_i \text{ odd}\}$, $J=\{i|b_i \text{ even}\}$, and 
automatically, we have $\sum_{i\in I}a_ib_i=N_{n_1}$, $\sum_{j \in J}a_jb_j=2n_2$, and $n=n_1+n_2$. 
Moreover, $s_{\psi}$ has the form $$\begin{pmatrix}
-I_{n_2}&&\\
&I_{N_{n_1}}&\\
&&-I_{n_2}
\end{pmatrix}.$$

We now apply \cite[Theorem 2.2.1, Part (b)]{Art13} with $x=s=s_{\psi}$. It is clear that the stabilizer of $s$ in $\widehat{\mathrm{G}}_{n}(\BC)$ is
\begin{align*}
    \SO_{2n_1+1}(\BC) \times \SO_{2n_2}(\BC) & \text{ when } \mathrm{G}_n=\Sp_{2n},\\
    \Sp_{2n_1}(\BC) \times \Sp_{2n_2}(\BC) & \text{ when } \mathrm{G}_n=\SO_{2n+1},\\
    \SO_{2n_1}(\BC) \times \SO_{2n_2}(\BC) & \text{ when } \mathrm{G}_n=\SO_{2n}^{\alpha},
\end{align*}
and the corresponding endoscopic group $\mathrm{G}'=\mathrm{G}_1' \times \mathrm{G}_2'$ of $\mathrm{G}_n$ is
\begin{align*}
    \Sp_{2n_1} \times \SO_{2n_2}^{\beta} & \text{ when } \mathrm{G}_n=\Sp_{2n},\\
    \SO_{2n_1+1} \times \SO_{2n_2+1} & \text{ when } \mathrm{G}_n=\SO_{2n+1},\\
    \SO_{2n_1}^{\gamma_1} \times \SO_{2n_2}^{\gamma_2} & \text{ when } \mathrm{G}_n=\SO_{2n}^{\alpha},
\end{align*}
where $\beta, \gamma_1, \gamma_2$ are square classes in $F$, and $\gamma_1\gamma_2=\alpha$.
By \cite[Theorem 2.2.1]{Art13}, we have the following distribution identity
\begin{equation}\label{distribution identity}
    \sum_{\pi \in \wt{\Pi}_{\psi}} \langle s_{\psi} x, \pi \rangle f_{G_n}(\pi) = f'(\psi')=\tr (\wt{\pi}_{\psi^1}(\wt{f}^1)) \tr (\wt{\pi}_{\psi^2}(\wt{f}^2)),
\end{equation}
where $f \in \wt{\CH}(G_n)$, $f' \in \wt{\CS}(G')$ is the transfer of $f$ to $G'$ with the assumption that $f'=f'^1 \otimes f'^2$, $f'^i \in \wt{\CS}(G'_i)$, $i=1, 2$. 
Let $N_1=\begin{cases}
2n_1+1 & \text{ when } \mathrm{G}_n=\Sp_{2n},\\
2n_1 & \text{ when } \mathrm{G}_n=\SO_{2n+1}, \SO_{2n}^{\alpha},
\end{cases}$ and $N_2=2n_2$.
Let $\wt{f}^1 \in \wt{\CH}(N_1)$ 
and $\wt{f}^2 \in \wt{\CH}(N_2)$, transferring to $f'^1$ and $f'^2$, via surjective maps in \cite[Corollary 2.1.2]{Art13}
$$\iota_{G'_i}: \wt{\CH}(N_i) \rightarrow \wt{\CS}(G'_i), i=1,2,$$
respectively. Here, $\psi'$ is the factor through of $\psi$ to ${}^L\mathrm{G}'$, and $\psi^1=\sum_{i\in I}\psi_i$, $\psi^2 = \sum_{j \in J}\psi_j$.

\subsection{Character expansions}\label{sec:character expansion}

In this section, we take the character expansions of both sides of \eqref{distribution identity}.  The left hand side  {is expanded at the identity, and } the right hand side  {is expanded at the twisted points } $\theta_{\widehat{\mathrm{G}}_i'} = 1 \rtimes \wt{\theta}(N_i)$ (see \eqref{thetaGn}), $i=1,2$. 
Then we have the following equality: 
\begin{align}\label{sec6equ1all}
\begin{split}
    &\,  \sum_{\pi \in \wt\Pi_{\psi}} \langle 1, \pi \rangle \sum_{\CO \in \CN_{\frak{g}_{n}}} 
c_{\CO}(\pi)\hat{\mu}_\CO(f) \\
= &\,  \left(\sum {_{\CO \in \CN_{\frak{g_1'}^{\theta}}}
} c_{\CO}(\wt{\pi}_{\psi^1})\hat{\mu}_\CO({{}\wt{f}^1_{\theta_{\widehat{\mathrm{G}}_1'}}})\right) \left(\sum {_{\CO \in \CN_{\frak{g_2'}^{\theta}}}
} c_{\CO}(\wt{\pi}_{\psi^2})\hat{\mu}_\CO({{}\wt{f}^2_{\theta_{\widehat{\mathrm{G}}_2'}}})\right),
\end{split}
\end{align}
where $\CN_{\frak{g_i'}^{\theta}}$ denotes the set of $F$-rational nilpotent orbits in the Lie algebra $\frak{g_i'}^{\theta}$, {{}$\wt{f}_{\theta_{\widehat{\mathrm{G}}_i'}}^i$ is the Harish-Chandra descent of $\widetilde{f}^i$ (see for example \cite[Section 3.1]{Kon02}), $i=1,2$.} 
Equation \eqref{sec6equ1all} can be rewritten as follows:
\begin{align}\label{sec6equ2all}
\begin{split}
    &\,  \sum_{\CO \in \CN_{\frak{g}_{n}}}
c_{\CO}(\wt{\Pi}_{\psi})\hat{\mu}_\CO(f) \\
= &\,  \left(\sum {_{\CO \in \CN_{\frak{g_1'}^{\theta}}}
} c_{\CO}(\wt{\pi}_{\psi^1})\hat{\mu}_\CO({{}\wt{f}^1_{\theta_{\widehat{\mathrm{G}}_1'}}})\right) \left(\sum {_{\CO \in \CN_{\frak{g_2'}^{\theta}}}
} c_{\CO}(\wt{\pi}_{\psi^2})\hat{\mu}_\CO({{}\wt{f}^2_{\theta_{\widehat{\mathrm{G}}_2'}}})\right),
\end{split}
\end{align}
where $c_{\CO}(\wt{\Pi}_{\psi})=\sum_{\pi \in \wt\Pi_{\psi}} 
c_{\CO}(\pi)$. Note that 
${\RG'_1}^{\theta}=\begin{cases}
\SO_{2n_1+1} & \text{ when } \mathrm{G}_n=\Sp_{2n},\\
\Sp_{2n_1} & \text{ when } \mathrm{G}_n=\SO_{2n+1}, \SO_{2n}^{\alpha},
\end{cases}$, ${\RG'_2}^{\theta}=\Sp_{2n_2}$, and $\frak{g_i'}^{\theta}$ is the Lie algebra of ${\RG'_i}^{\theta}$.

\subsection{Dimension identities for nilpotent orbits}

Recall again that given a partition $\ul{p}$ of $M$, where
$$M:= 
\begin{cases}
2n & \text{ when } \mathrm{G}_n=\Sp_{2n}, \SO_{2n}^{\alpha}, \\
2n+1 & \text{ when } \mathrm{G}_n=\SO_{2n+1},
\end{cases}
$$ 
$\ul{p}_{\RG_n}$ denotes the $\RG_n$-collapse of $\ul{p}$, that is, the biggest $\RG_n$-partition smaller than or equal to $\ul{p}$, and $\ul{p}^{\RG_n}$ denotes the $\RG_n$-expansion of $\ul{p}$, that is, the smallest special $\RG_n$-partition bigger than or equal to $\ul{p}$. For the definition of special partitions, see \cite[Section 6.3]{CM93}.

We first record a dimension identity as follows.  

\begin{lem}\label{keylemma2}
We have the following dimension identity:
\begin{align}\label{equ1:keylemma2}
\begin{split}
    &\, \dim(\frak{g}_{n}) - \dim_{\frak{g}_{n}}(\eta_{\hat{\frak{g}}_n, \frak{g}_{n}}([b_1^{a_1} \cdots b_r^{a_r}])) \\
    = &\, \dim(\hat{\frak{g}}'_1) + \dim(\hat{\frak{g}}'_2) -\dim_{\hat{\frak{g}}'_1}(([\prod_{i\in I}b_i^{a_i}]^t)_{\widehat{\mathrm{G}}_1'})-\dim_{\hat{\frak{g}}'_2}(([\prod_{j\in J}b_j^{a_j}]^t)_{\widehat{\mathrm{G}}_2'}).
    \end{split}
\end{align}
\end{lem}

\begin{proof}
Given any partition $\ul{p}=[p_1 \cdots p_r]$ with $p_1 \geq \cdots \geq p_r$, recall that $\ul{p}^+=[(p_1+1) \cdots p_r]$, $\ul{p}^-=[p_1 \cdots (p_r-1)]$.
By the proof of \cite[Theorem 6.3.11]{CM93},
given a partition $\ul{p}$ of type $\frak{g}_n=\frak{sp}_{2n}, \frak{so}_{2n+1}$, 
$(\ul{p}^t)_{G_n} = (\ul{p}^{G_n})^t.$
By \cite[Lemma 3.3]{Ach03}, given a partition $\ul{p}$ of $2n$, 
if it is an orthogonal partition or its transpose is a symplectic partition, then 
$(\ul{p}^t)_{\SO_{2n}}=((\ul{p}^{+-})_{\Sp_{2n}})^t$. 
Also recall that $I=\{i|b_i \text{ odd}\}$, $J=\{i|b_i \text{ even}\}$.
Hence, combining the recipe in \cite[Lemmas 6.3.8, 6.3.9]{CM93}, when $\mathrm{G}_n=\Sp_{2n}$,
\begin{align*}
\eta_{\frak{so}_{2n+1}, \frak{sp}_{2n}}([b_1^{a_1} \cdots b_r^{a_r}])&=(([b_1^{a_1} \cdots b_r^{a_r}]^-)_{\Sp_{2n}})^t,\\
    ([\prod_{i\in I}b_i^{a_i}]^t)_{\SO_{2n_1+1}} & = ([\prod_{i\in I}b_i^{a_i}]^{\SO_{2n_1+1}})^t=[\prod_{i\in I}b_i^{a_i}]^t,\\
    ([\prod_{j\in J}b_j^{a_j}]^t)_{\SO_{2n_2}} & = (([\prod_{j\in J}b_j^{a_j}]^{+-})_{\Sp_{2n_2}})^t=[\prod_{j\in J}b_j^{a_j}]^t;
\end{align*}
when $\mathrm{G}_n=\SO_{2n+1}$,
\begin{align*}
\eta_{\frak{sp}_{2n}, \frak{so}_{2n+1}}([b_1^{a_1} \cdots b_r^{a_r}])&=(([b_1^{a_1} \cdots b_r^{a_r}]^+)_{\SO_{2n+1}})^t,\\
    ([\prod_{i\in I}b_i^{a_i}]^t)_{\Sp_{2n_1}} & = ([\prod_{i\in I}b_i^{a_i}]^{\Sp_{2n_1}})^t=[\prod_{i\in I}b_i^{a_i}]^t,\\
    ([\prod_{j\in J}b_j^{a_j}]^t)_{\Sp_{2n_2}} & = (([\prod_{j\in J}b_j^{a_j}]^{\Sp_{2n_2}})^t=[\prod_{j\in J}b_j^{a_j}]^t;
\end{align*}
when $\mathrm{G}_n=\SO_{2n}^{\alpha}$,
\begin{align*}
\eta_{\frak{o}_{2n}, \frak{o}_{2n}}([b_1^{a_1} \cdots b_r^{a_r}])&=([b_1^{a_1} \cdots b_r^{a_r}]^t)_{\SO_{2n}}=(([b_1^{a_1} \cdots b_r^{a_r}]^{+-})_{\Sp_{2n}})^t,\\
    ([\prod_{i\in I}b_i^{a_i}]^t)_{\SO_{2n_1}} & = (([\prod_{i\in I}b_i^{a_i}]^{+-})_{\Sp_{2n_1}})^t,\\
    ([\prod_{j\in J}b_j^{a_j}]^t)_{\SO_{2n_2}} & = (([\prod_{j\in J}b_j^{a_j}]^{+-})_{\Sp_{2n_2}})^t=[\prod_{j\in J}b_j^{a_j}]^t.
\end{align*}

Given a partition $\ul{p}=[p_1 p_2 \cdots p_r]$ of $m$, with $p_1 \geq p_2 \geq \cdots \geq p_r$, 
for $r=1,2,\ldots,m$, let $r_i=|\{j|p_j=i\}|$ and $s_i=|\{j|p_j\geq i\}|$. 
Note that if writing $\ul{p}^t=[q_1 q_2 \cdots q_l]$ with $q_1 \geq q_2 \geq \cdots \geq q_l$, then for $i=1, \ldots, l$, $s_i=q_i$, for $i=l+1, \ldots, m$, $s_i=0$. Also note that $r_i = s_i - s_{i+1}$ with the convention that $s_{m+1}=0$. 
By \cite[Corollary 6.1.4]{CM93}, if $\ul{p}$ is a symplectic partition of $m=2k$, then 
$$\dim(\mathcal{O}_{\ul{p}})=2k^2+k-\frac{1}{2} \sum_i s_i^2-\frac{1}{2} \sum_{i \text{ odd }} r_i;$$
if $\ul{p}$ is an orthogonal partition of $m=2k+1$, then 
$$\dim(\mathcal{O}_{\ul{p}})=2k^2+k-\frac{1}{2} \sum_i s_i^2+\frac{1}{2} \sum_{i \text{ odd }} r_i;$$
and if $\ul{p}$ is an orthogonal partition of $m=2k$, then 
$$\dim(\mathcal{O}_{\ul{p}})=2k^2-k-\frac{1}{2} \sum_i s_i^2+\frac{1}{2} \sum_{i \text{ odd }} r_i.$$

From now on, we prove the lemma separately for each case of $\mathrm{G}_n$. 

\textbf{Case $\mathrm{G}_n=\Sp_{2n}$.} 
Write 
\begin{align*}
    \ul{p}_1 =&\, (([b_1^{a_1} \cdots b_r^{a_r}]^-)_{\Sp_{2n}})^t,\\
    \ul{p}_2=&\,[\prod_{i\in I}b_i^{a_i}]^t,\\
    \ul{p}_3=&\,[\prod_{j\in J}b_j^{a_j}]^t,
\end{align*}
then,
\begin{align*}
    \ul{p}_1^t =&\, ([b_1^{a_1} \cdots b_r^{a_r}]^-)_{\Sp_{2n}},\\
    \ul{p}_2^t=&\,[\prod_{i\in I}b_i^{a_i}],\\
    \ul{p}_3^t=&\,([\prod_{j\in J}b_j^{a_j}].
\end{align*}
Then the left hand side of \eqref{equ1:keylemma2} becomes
\begin{align*}
    &\frac{2n(2n+1)}{2}-n-\left(2n^2+n-\frac{1}{2} \sum_i \left(s_i^{\ul{p}_1}\right)^2-\frac{1}{2} \sum_{i \text{ odd }} r_i^{\ul{p}_1}\right)\\
    =&\, -n + \frac{1}{2} \sum_i \left(s_i^{\ul{p}_1}\right)^2+\frac{1}{2} \sum_{i \text{ odd }} r_i^{\ul{p}_1}.
\end{align*}
The right hand side of \eqref{equ1:keylemma2} becomes
\begin{align*}
    &\,\frac{2n_1(2n_1+1)}{2}-n_1
    +\frac{2n_2(2n_2-1)}{2}-n_2\\
    &\,-\left(2n_1^2+n_1-\frac{1}{2} \sum_i \left(s_i^{\ul{p}_2}\right)^2+\frac{1}{2} \sum_{i \text{ odd }} r_i^{\ul{p}_2}\right)\\
    &\,-\left(2n_2^2-n_2-\frac{1}{2} \sum_i \left(s_i^{\ul{p}_3}\right)^2+\frac{1}{2} \sum_{i \text{ odd }} r_i^{\ul{p}_3}\right)\\
    =&\, -n_1-n_2 + \frac{1}{2} \sum_i \left(s_i^{\ul{p}_2}\right)^2-\frac{1}{2} \sum_{i \text{ odd }} r_i^{\ul{p}_2}+ \frac{1}{2} \sum_i \left(s_i^{\ul{p}_3}\right)^2-\frac{1}{2} \sum_{i \text{ odd }} r_i^{\ul{p}_3}.
\end{align*}
Hence, to show \eqref{equ1:keylemma2}, it suffices to show that 
\begin{align}\label{equ2:keylemma2}
    \begin{split}
    &\, \sum_i \left(s_i^{\ul{p}_2}\right)^2+  \sum_i \left(s_i^{\ul{p}_3}\right)^2- \sum_i \left(s_i^{\ul{p}_1}\right)^2\\
        =&\,\sum_{i \text{ odd }} r_i^{\ul{p}_1}+\sum_{i \text{ odd }} r_i^{\ul{p}_2}+ \sum_{i \text{ odd }} r_i^{\ul{p}_3},
    \end{split}
\end{align}
i.e.,
\begin{align}\label{equ3:keylemma2}
    \begin{split}
        &\,\sum_i \left(s_i^{\ul{p}_2}\right)^2+  \sum_i \left(s_i^{\ul{p}_3}\right)^2- \sum_i \left(s_i^{\ul{p}_1}\right)^2\\
        =&\, \sum_{i \text{ odd }} \left(s_i^{\ul{p}_1}-s_{i+1}^{\ul{p}_1}\right)+\sum_{i \text{ odd }} \left(s_i^{\ul{p}_2}-s_{i+1}^{\ul{p}_2}\right)+ \sum_{i \text{ odd }} \left(s_i^{\ul{p}_3}-s_{i+1}^{\ul{p}_3}\right).
    \end{split}
\end{align}

Since $[b_1^{a_1} \cdots b_r^{a_r}]$ is an orthogonal partition of $2n+1$, we may write it as
\MATHBLOCKDOLLARDOLLAR[[b_1^{a_1} \cdots b_r^{a_r}]]{=\left[\left(\prod_{k=1}^{r_0} n_k^{2e_k}\right) \left(\prod_{i=1}^{2s+1} (2m_i+1)^{f_i}p_{i,1}^{2g_{i,1}}\cdots p_{i,r_i}^{2g_{i,r_i}}\right)\right],}
where $f_i$ is odd, $2m_i+1 > p_{i,1} > \cdots >p_{i,r_i}>2m_{i-1}+1$, for $i=1, \ldots, 2s+1$, with the convention of $m_0=-\frac{1}{2}$, and for $k=1, \ldots, r_0$, $n_k > 2m_{2s+1}+1$. 
To proceed, we consider two cases: Case (1),
$r_1\neq0$; Case (2), $r_1=0$. 

For Case (1), $r_1 \neq 0$, we have 
\begin{align}\label{equ3:keylemma}
\begin{split}
& ([b_1^{a_1} \cdots b_r^{a_r}]^-)_{\Sp_{2n}}\\
=  \ & \Biggl[\left(\prod_{k=1}^{r_0} n_k^{2e_k}\right)\left(\prod_{i=2}^{2s+1} (2m_i+1)^{f_i}p_{i,1}^{2g_{i,1}}\cdots p_{i,r_i}^{2g_{i,r_i}}\right)\\
& \cdot (2m_1+1)^{f_1}p_{1,1}^{2g_{1,1}}\cdots
p_{1,r_1-1}^{2g_{1,r_1-1}} p_{1,r_1}^{2g_{1,r_1}-1}(p_{1,r_1}-1)\Biggr]_{\Sp_{2n}}.
\end{split}
\end{align}
For $2 \leq i \leq 2s+1$, assume that all the odd parts in $\{p_{i,1}, \ldots, p_{i,r_i}\}$ are
$\{(2q_{i,1}+1), \ldots, (2q_{i,t_i}+1)\}$, with
$2q_{i,1}+1 > \cdots > 2q_{i,t_i}+1$.
And assume that all the odd parts in $\{p_{1,1}, \ldots, p_{1,r_1-1}\}$ are
$\{(2q_{1,1}+1), \ldots, (2q_{1,t_1}+1)\}$, with
$2q_{1,1}+1 > \cdots > 2q_{1,t_1}+1$.
For $1 \leq i \leq 2s+1$, and $1 \leq j \leq t_i$, we assume that the exponent of $2q_{i,j}+1$ is $h_{i,j}$.
Then by the recipe \cite[Lemma 6.3.8]{CM93}, to get the $\Sp_{2n}$-collapse on the right hand side of \eqref{equ3:keylemma}, we just have to do the following:
\begin{itemize}
\item for $1 \leq i \leq s$,
replace $(2m_{2i+1}+1)^{f_{2i+1}}(2m_{2i}+1)^{f_{2i}}$
by $(2m_{2i+1}+1)^{f_{2i+1}-1}(2m_{2i+1})(2m_{2i}+2)(2m_{2i}+1)^{f_{2i}-1}$,
and for $1 \leq j \leq t_{2i+1}$, replace
$(2q_{{2i+1},j}+1)^{h_{{2i+1},j}}$ by
$$(2q_{{2i+1},j}+2))(2q_{{2i+1},j}+1)^{h_{{2i+1},j}-2}(2q_{{2i+1},j});$$
\item replace $(2m_1+1)^{f_1}$ by $(2m_1+1)^{f_1-1}(2m_1)$, for $1 \leq j \leq t_{1}$, replace
$(2q_{{1},j}+1)^{h_{1,j}}$ by
$$(2q_{{1},j}+2))(2q_{{1},j}+1)^{h_{{1},j}-2}(2q_{{1},j}),$$
and
replace $p_{1,r_1}^{2g_{1,r_1}-1}$
by $(p_{1,r_1}+1)p_{1,r_1}^{2g_{1,r_1}-2}$ if $p_{1,r_1}$ is odd, replace $(p_{1,r_1}-1)$
by $p_{1,r_1}$ if $p_{1,r_1}$ is even.
\end{itemize}

Then the left hand side of \eqref{equ3:keylemma2} becomes
\begin{align*}
\begin{split}
& 
\sum_i \left(s_i^{\ul{p}_2}\right)^2+  \sum_i \left(s_i^{\ul{p}_3}\right)^2- \sum_i \left(s_i^{\ul{p}_1}\right)^2\\
=  \ & 
\begin{cases}
\sum_{i=1}^s (4m_{2i+1}-4m_{2i}-2 -2t_{2i+1}) + (4m_1-1)-2t_1, \text{ if }p_{1,r_1} \text{ is odd},\\
\sum_{i=1}^s (4m_{2i+1}-4m_{2i}-2 -2t_{2i+1}) + (4m_1+1)-2t_1, \text{ if }p_{1,r_1} \text{ is even}.\\
\end{cases}
\end{split}
\end{align*}
And the right hand side of \eqref{equ3:keylemma2} becomes
\begin{align*}
\begin{split}
& \sum_{i \text{ odd }} \left(s_i^{\ul{p}_1}-s_{i+1}^{\ul{p}_1}\right)+\sum_{i \text{ odd }} \left(s_i^{\ul{p}_2}-s_{i+1}^{\ul{p}_2}\right)+ \sum_{i \text{ odd }} \left(s_i^{\ul{p}_3}-s_{i+1}^{\ul{p}_3}\right)\\
=  \ & \sum_{i=1}^s (2m_{2i+1}-2m_{2i}-2-2t_{2i+1}) + (2m_1-2)-2t_1\\
\ & + \sum_{i=1}^s (2m_{2i+1}-2m_{2i}) + (2m_1+1)\\
&\, + 0\\
= \ & \sum_{i=1}^s (4m_{2i+1}-4m_{2i}-2 -2t_{2i+1}) + (4m_1-1)-2t_1, \text{ if }p_{1,r_1} \text{ is odd};
\end{split}
\end{align*}
and becomes
\begin{align*}
\begin{split}
& \sum_{i \text{ odd }} \left(s_i^{\ul{p}_1}-s_{i+1}^{\ul{p}_1}\right)+\sum_{i \text{ odd }} \left(s_i^{\ul{p}_2}-s_{i+1}^{\ul{p}_2}\right)+ \sum_{i \text{ odd }} \left(s_i^{\ul{p}_3}-s_{i+1}^{\ul{p}_3}\right)\\
=  \ & \sum_{i=1}^s (2m_{2i+1}-2m_{2i}-2-2t_{2i+1}) + (2m_1)-2t_1\\
\ & + \sum_{i=1}^s (2m_{2i+1}-2m_{2i}) + (2m_1+1)\\
\ & + 0\\
= \ & \sum_{i=1}^s (4m_{2i+1}-4m_{2i}-2 -2t_{2i+1}) + (4m_1+1)-2t_1, \text{ if }p_{1,r_1} \text{ is even}.
\end{split}
\end{align*}
Hence, we have proved \eqref{equ3:keylemma2}.

For Case (2), $r_1=0$, we have
\begin{align}\label{equ5:keylemma}
\begin{split}
& ([b_1^{a_1} \cdots b_r^{a_r}]^-)_{\Sp_{2n}}\\
=  \ & \Biggl[\left(\prod_{k=1}^{r_0} n_k^{2e_k}\right)\left(\prod_{i=2}^{2s+1} (2m_i+1)^{f_i}p_{i,1}^{2g_{i,1}}\cdots p_{i,r_i}^{2g_{i,r_i}}\right)\\
& \cdot (2m_1+1)^{f_1-1}(2m_1)\Biggr]_{\Sp_{2n}}.
\end{split}
\end{align}
As in Cases (1) and (2),
for $2 \leq i \leq 2s+1$, assume that all the odd parts in $\{p_{i,1}, \ldots, p_{i,r_i}\}$ are
$\{(2q_{i,1}+1), \ldots, (2q_{i,t_i}+1)\}$, with
$2q_{i,1}+1 > \cdots > 2q_{i,t_i}+1$, and for $1 \leq j \leq t_i$, we assume that the exponent of $2q_{i,j}+1$ is $h_{i,j}$.
Then by the recipe in Lemma 6.3.8 of \cite{CM93} (see the beginning of Section 4), to get the $\Sp_{2n}$-collapse on the right hand side of \eqref{equ3:keylemma2}, we just have to do the following:
\begin{itemize}
\item for $1 \leq i \leq s$,
replace $(2m_{2i+1}+1)^{f_{2i+1}}(2m_{2i}+1)^{f_{2i}}$
by $(2m_{2i+1}+1)^{f_{2i+1}-1}(2m_{2i+1})(2m_{2i}+2)(2m_{2i}+1)^{f_{2i}-1}$,
then for $1 \leq j \leq t_{2i+1}$, replace
$(2q_{{2i+1},j}+1)^{h_{{2i+1},j}}$ by
$$(2q_{{2i+1},j}+2))(2q_{{2i+1},j}+1)^{h_{{2i+1},j}-2}(2q_{{2i+1},j}).$$
\end{itemize}

Then the left hand side of \eqref{equ3:keylemma2} becomes
\begin{align*}
\begin{split}
& \sum_i \left(s_i^{\ul{p}_2}\right)^2+  \sum_i \left(s_i^{\ul{p}_3}\right)^2- \sum_i \left(s_i^{\ul{p}_1}\right)^2\\
=  \ & \sum_{i=1}^s (4m_{2i+1}-4m_{2i}-2 -2t_{2i+1}) + (4m_1+1).
\end{split}
\end{align*}
And the right hand side of \eqref{equ3:keylemma2} becomes
\begin{align*}
\begin{split}
& \sum_{i \text{ odd }} \left(s_i^{\ul{p}_1}-s_{i+1}^{\ul{p}_1}\right)+\sum_{i \text{ odd }} \left(s_i^{\ul{p}_2}-s_{i+1}^{\ul{p}_2}\right)+ \sum_{i \text{ odd }} \left(s_i^{\ul{p}_3}-s_{i+1}^{\ul{p}_3}\right)\\
=  \ & \sum_{i=1}^s (2m_{2i+1}-2m_{2i}-2-2t_{2i+1}) + (2m_1)\\
\ & + \sum_{i=1}^s (2m_{2i+1}-2m_{2i}) + (2m_1+1)\\
\ & + 0\\
= \ & \sum_{i=1}^s (4m_{2i+1}-4m_{2i}-2 -2t_{2i+1}) + (4m_1+1).
\end{split}
\end{align*}
Hence, we also have proved \eqref{equ3:keylemma2}.

\textbf{Case $\mathrm{G}_n=\SO_{2n+1}$.} 
Write 
\begin{align*}
    \ul{p}_1 =&\, (([b_1^{a_1} \cdots b_r^{a_r}]^+)_{\SO_{2n+1}})^t,\\
    \ul{p}_2=&\,[\prod_{i\in I}b_i^{a_i}]^t,\\
    \ul{p}_3=&\,[\prod_{j\in J}b_j^{a_j}]^t,
\end{align*}
then,
\begin{align*}
    \ul{p}_1^t =&\, ([b_1^{a_1} \cdots b_r^{a_r}]^+)_{\SO_{2n+1}},\\
    \ul{p}_2^t=&\,[\prod_{i\in I}b_i^{a_i}],\\
    \ul{p}_3^t=&\,([\prod_{j\in J}b_j^{a_j}].
\end{align*}
Then the left hand side of \eqref{equ1:keylemma2} becomes
\begin{align*}
    &\frac{2n(2n+1)}{2}-n-\left(2n^2+n-\frac{1}{2} \sum_i \left(s_i^{\ul{p}_1}\right)^2+\frac{1}{2} \sum_{i \text{ odd }} r_i^{\ul{p}_1}\right)\\
    =&\, -n + \frac{1}{2} \sum_i \left(s_i^{\ul{p}_1}\right)^2-\frac{1}{2} \sum_{i \text{ odd }} r_i^{\ul{p}_1}.
\end{align*}
The right hand side of \eqref{equ1:keylemma2} becomes
\begin{align*}
    &\,\frac{2n_1(2n_1+1)}{2}-n_1
    +\frac{2n_2(2n_2+1)}{2}-n_2\\
    &\,-\left(2n_1^2+n_1-\frac{1}{2} \sum_i \left(s_i^{\ul{p}_2}\right)^2-\frac{1}{2} \sum_{i \text{ odd }} r_i^{\ul{p}_2}\right)\\
    &\,-\left(2n_2^2+n_2-\frac{1}{2} \sum_i \left(s_i^{\ul{p}_3}\right)^2-\frac{1}{2} \sum_{i \text{ odd }} r_i^{\ul{p}_3}\right)\\
    =&\, -n_1-n_2 + \frac{1}{2} \sum_i \left(s_i^{\ul{p}_2}\right)^2+\frac{1}{2} \sum_{i \text{ odd }} r_i^{\ul{p}_2}+ \frac{1}{2} \sum_i \left(s_i^{\ul{p}_3}\right)^2+\frac{1}{2} \sum_{i \text{ odd }} r_i^{\ul{p}_3}.
\end{align*}
Hence, to show \eqref{equ1:keylemma2}, it suffices to show that 
\begin{align}\label{equ2-2:keylemma2}
    \begin{split}
    &\, \sum_i \left(s_i^{\ul{p}_1}\right)^2-\sum_i \left(s_i^{\ul{p}_2}\right)^2-  \sum_i \left(s_i^{\ul{p}_3}\right)^2\\
        =&\,\sum_{i \text{ odd }} r_i^{\ul{p}_1}+\sum_{i \text{ odd }} r_i^{\ul{p}_2}+ \sum_{i \text{ odd }} r_i^{\ul{p}_3},
    \end{split}
\end{align}
i.e.,
\begin{align}\label{equ3-2:keylemma2}
    \begin{split}
        &\,\sum_i \left(s_i^{\ul{p}_1}\right)^2-\sum_i \left(s_i^{\ul{p}_2}\right)^2-  \sum_i \left(s_i^{\ul{p}_3}\right)^2\\
        =&\, \sum_{i \text{ odd }} \left(s_i^{\ul{p}_1}-s_{i+1}^{\ul{p}_1}\right)+\sum_{i \text{ odd }} \left(s_i^{\ul{p}_2}-s_{i+1}^{\ul{p}_2}\right)+ \sum_{i \text{ odd }} \left(s_i^{\ul{p}_3}-s_{i+1}^{\ul{p}_3}\right).
    \end{split}
\end{align}

Since $[b_1^{a_1} \cdots b_r^{a_r}]$ is a symplectic partition of $2n$, we may write it as
\MATHBLOCKDOLLARDOLLAR[[b_1^{a_1} \cdots b_r^{a_r}]]{=\left[\left(\prod_{k=1}^{r_0} n_k^{2e_k} \right)\left(\prod_{i=1}^{s} (2m_i)^{f_i}p_{i,1}^{2g_{i,1}}\cdots p_{i,r_i}^{2g_{i,r_i}}\right)\right],}
where $f_i$ is odd, $2m_i > p_{i,1} > \cdots >p_{i,r_i}>2m_{i-1}$, for $i=1, \ldots, s$, with the convention of $m_1 \neq 0$ and $m_0=0$, and $n_1> \cdots > n_k > 2m_{s}$. 
Similar to Case of $\mathrm{G}_n=\Sp_{2n}$, to proceed, we consider two cases: Case (1),
$r_0\neq0$; Case (2), $r_0=0$. 

For Case (1), $r_0 \neq 0$, we have 
\begin{align}\label{equ3-2:keylemma}
\begin{split}
& ([b_1^{a_1} \cdots b_r^{a_r}]^+)_{\SO_{2n+1}}\\
=  \ & \left[(n_1+1)n_1^{2e_1-1}\left(\prod_{k=2}^{r_0} n_k^{2e_k}\right) \left(\prod_{i=1}^{s} (2m_i)^{f_i}p_{i,1}^{2g_{i,1}}\cdots p_{i,r_i}^{2g_{i,r_i}}\right)\right]_{\SO_{2n+1}}.
\end{split}
\end{align}

Assume that all the even parts in $\{n_2, n_3, \cdots, n_{r_0}\}$ are 
$\{2q_{0,1}, \ldots, 2q_{0,t_0}\}$, with
$2q_{0,1} > \cdots > 2q_{0,t_0}$.
For $1 \leq i \leq [\frac{s}{2}]$, assume that all the even parts in $\{p_{s-(2i-1),1}, \ldots, p_{s-(2i-1),r_{2i-1}}\}$ are
$\{2q_{i,1}, \ldots, 2q_{i,t_i}\}$, with
$2q_{i,1} > \cdots > 2q_{i,t_i}$.
For $0 \leq i \leq [\frac{s}{2}]$, and $1 \leq j \leq t_i$, we assume that the exponent of $2q_{i,j}$ is $h_{i,j}$.
Then by the recipe \cite[Lemma 6.3.8]{CM93}, to get the $\SO_{2n+1}$-collapse on the right hand side of \eqref{equ3-2:keylemma}, we just have to do the following:
\begin{itemize}
\item replace $(n_1+1)(2m_s)$ by $(n_1)(2m_s+1)$ if $n_1$ is odd, replace $(n_1)(2m_s)$ by $(n_1-1)(2m_s+1)$ if $n_1$ is even, and for  $1 \leq j \leq t_{0}$, replace
$(2q_{0,j})^{h_{0,j}}$ by
$(2q_{0,j}+1)(2q_{0,j})^{h_{0,j}-2}(2q_{0,j}-1)$;
\item for $1 \leq i \leq [\frac{s}{2}]$,
replace $(2m_{s-(2i-1)})^{f_{s-(2i-1)}}(2m_{s-(2i)})^{f_{s-(2i)}}$
by $(2m_{s-(2i-1)})^{f_{s-(2i-1)}-1}(2m_{s-(2i-1)}-1)(2m_{s-(2i)}+1)(2m_{s-(2i)})^{f_{s-(2i)}-1}$,
and for $1 \leq j \leq t_{i}$, replace
$(2q_{i,j})^{h_{i,j}}$ by
$(2q_{i,j}+1)(2q_{i,j})^{h_{i,j}-2}(2q_{i,j}-1)$.
\end{itemize}

Then the left hand side of \eqref{equ3-2:keylemma2} becomes
\begin{align*}
\begin{split}
& 
\sum_i \left(s_i^{\ul{p}_1}\right)^2-  \sum_i \left(s_i^{\ul{p}_2}\right)^2- \sum_i \left(s_i^{\ul{p}_3}\right)^2\\
=  \ & 
\begin{cases}
\sum_{i=1}^{[\frac{s}{2}]} (4m_{s-(2i)}-4m_{s-(2i-1)}+2 -2t_{i}) + (4m_s+1-2t_0), \text{ if } n_1 \text{ is odd},\\
\sum_{i=1}^{[\frac{s}{2}]} (4m_{s-(2i)}-4m_{s-(2i-1)}+2 -2t_{i}) + (4m_s+3-2t_0), \text{ if } n_1 \text{ is even}.
\end{cases}
\end{split}
\end{align*}
And the right hand side of \eqref{equ3-2:keylemma2} becomes
\begin{align*}
\begin{split}
& \sum_{i \text{ odd }} \left(s_i^{\ul{p}_1}-s_{i+1}^{\ul{p}_1}\right)+\sum_{i \text{ odd }} \left(s_i^{\ul{p}_2}-s_{i+1}^{\ul{p}_2}\right)+ \sum_{i \text{ odd }} \left(s_i^{\ul{p}_3}-s_{i+1}^{\ul{p}_3}\right)\\
=  \ & \sum_{i=1}^{[\frac{s}{2}]} (2m_{s-(2i)}-2m_{s-(2i-1)}+2-2t_{i}) + (2m_s+1-2t_0)\\
\ & +0\\
\ & + \sum_{i=1}^{[\frac{s}{2}]} (2m_{s-(2i)}-2m_{s-(2i-1)}) + (2m_s)\\
= \ & \sum_{i=1}^{[\frac{s}{2}]} (4m_{s-(2i)}-4m_{s-(2i-1)}+2 -2t_{i}) + (4m_s+1-2t_0),
\end{split}
\end{align*}
if $n_1$ is odd, and becomes
\begin{align*}
\begin{split}
& \sum_{i \text{ odd }} \left(s_i^{\ul{p}_1}-s_{i+1}^{\ul{p}_1}\right)+\sum_{i \text{ odd }} \left(s_i^{\ul{p}_2}-s_{i+1}^{\ul{p}_2}\right)+ \sum_{i \text{ odd }} \left(s_i^{\ul{p}_3}-s_{i+1}^{\ul{p}_3}\right)\\
=  \ & \sum_{i=1}^{[\frac{s}{2}]} (2m_{s-(2i)}-2m_{s-(2i-1)}+2-2t_{i}) + (2m_s+3-2t_0)\\
\ & +0\\
\ & + \sum_{i=1}^{[\frac{s}{2}]} (2m_{s-(2i)}-2m_{s-(2i-1)}) + (2m_s)\\
= \ & \sum_{i=1}^{[\frac{s}{2}]} (4m_{s-(2i)}-4m_{s-(2i-1)}+2 -2t_{i}) + (4m_s+3-2t_0),
\end{split}
\end{align*}
if $n_1$ is even. 
Hence, we have proved \eqref{equ3-2:keylemma2}.

For Case (2), $r_0=0$, we have 
\begin{align}\label{equ3-2-2:keylemma}
\begin{split}
& ([b_1^{a_1} \cdots b_r^{a_r}]^+)_{\SO_{2n+1}}\\
=  \ & \left[(2m_s+1)(2m_s)^{f_s-1}\left(\prod_{i=1}^{s-1} (2m_i)^{f_i}p_{i,1}^{2g_{i,1}}\cdots p_{i,r_i}^{2g_{i,r_i}}\right)\right]_{\SO_{2n+1}}.
\end{split}
\end{align}

For $1 \leq i \leq [\frac{s}{2}]$, assume that all the even parts in 
$$\{p_{s-(2i-1),1}, \ldots, p_{s-(2i-1),r_{2i-1}}\}$$ are
$\{2q_{i,1}, \ldots, 2q_{i,t_i}\}$, with
$2q_{i,1} > \cdots > 2q_{i,t_i}$.
For $1 \leq i \leq [\frac{s}{2}]$\, and $1 \leq j \leq t_i$, we assume that the exponent of $2q_{i,j}$ is $h_{i,j}$.
Then by the recipe \cite[Lemma 6.3.8]{CM93}, to get the $\SO_{2n+1}$-collapse on the right hand side of \eqref{equ3-2-2:keylemma}, we just have to do the following:
\begin{itemize}
\item for $1 \leq i \leq [\frac{s}{2}]$,
replace $(2m_{s-(2i-1)})^{f_{s-(2i-1)}}(2m_{s-(2i)})^{f_{s-(2i)}}$
by $(2m_{s-(2i-1)})^{f_{s-(2i-1)}-1}(2m_{s-(2i-1)}-1)(2m_{s-(2i)}+1)(2m_{s-(2i)})^{f_{s-(2i)}-1}$,
and for $1 \leq j \leq t_{i}$, replace
$(2q_{i,j})^{h_{i,j}}$ by
$(2q_{i,j}+1)(2q_{i,j})^{h_{i,j}-2}(2q_{i,j}-1)$.
\end{itemize}

Then the left hand side of \eqref{equ3-2:keylemma2} becomes
\begin{align*}
\begin{split}
& 
\sum_i \left(s_i^{\ul{p}_1}\right)^2-  \sum_i \left(s_i^{\ul{p}_2}\right)^2- \sum_i \left(s_i^{\ul{p}_3}\right)^2\\
=  \ & \sum_{i=1}^{[\frac{s}{2}]} (4m_{s-(2i)}-4m_{s-(2i-1)}+2 -2t_{i}) + (4m_s+1).
\end{split}
\end{align*}
And the right hand side of \eqref{equ3-2:keylemma2} becomes
\begin{align*}
\begin{split}
& \sum_{i \text{ odd }} \left(s_i^{\ul{p}_1}-s_{i+1}^{\ul{p}_1}\right)+\sum_{i \text{ odd }} \left(s_i^{\ul{p}_2}-s_{i+1}^{\ul{p}_2}\right)+ \sum_{i \text{ odd }} \left(s_i^{\ul{p}_3}-s_{i+1}^{\ul{p}_3}\right)\\
=  \ & \sum_{i=1}^{[\frac{s}{2}]} (2m_{s-(2i)}-2m_{s-(2i-1)}+2-2t_{i}) + (2m_s+1)\\
\ & +0\\
\ & + \sum_{i=1}^{[\frac{s}{2}]} (2m_{s-(2i)}-2m_{s-(2i-1)}) + (2m_s)\\
= \ & \sum_{i=1}^{[\frac{s}{2}]} (4m_{s-(2i)}-4m_{s-(2i-1)}+2 -2t_{i}) + (4m_s+1).
\end{split}
\end{align*}
Hence, we have proved \eqref{equ3-2:keylemma2}.

\textbf{Case $\mathrm{G}_n=\SO_{2n}^{\alpha}$.} 
Write 
\begin{align*}
    \ul{p}_1 =&\, (([b_1^{a_1} \cdots b_r^{a_r}]^{+-})_{\Sp_{2n}})^t,\\
    \ul{p}_2=&\,(([\prod_{i\in I}b_i^{a_i}]^{+-})_{\Sp_{2n_1}})^t,\\
    \ul{p}_3=&\,[\prod_{j\in J}b_j^{a_j}]^t,
\end{align*}
then,
\begin{align*}
    \ul{p}_1^t =&\, ([b_1^{a_1} \cdots b_r^{a_r}]^{+-})_{\Sp_{2n}},\\
    \ul{p}_2^t=&\,([\prod_{i\in I}b_i^{a_i}]^{+-})_{\Sp_{2n_1}},\\
    \ul{p}_3^t=&\,([\prod_{j\in J}b_j^{a_j}].
\end{align*}
Then the left hand side of \eqref{equ1:keylemma2} becomes
\begin{align*}
    &\frac{2n(2n-1)}{2}-n-\left(2n^2-n-\frac{1}{2} \sum_i \left(s_i^{\ul{p}_1}\right)^2+\frac{1}{2} \sum_{i \text{ odd }} r_i^{\ul{p}_1}\right)\\
    =&\, -n + \frac{1}{2} \sum_i \left(s_i^{\ul{p}_1}\right)^2-\frac{1}{2} \sum_{i \text{ odd }} r_i^{\ul{p}_1}.
\end{align*}
The right hand side of \eqref{equ1:keylemma2} becomes
\begin{align*}
    &\,\frac{2n_1(2n_1-1)}{2}-n_1
    +\frac{2n_2(2n_2-1)}{2}-n_2\\
    &\,-\left(2n_1^2-n_1-\frac{1}{2} \sum_i \left(s_i^{\ul{p}_2}\right)^2+\frac{1}{2} \sum_{i \text{ odd }} r_i^{\ul{p}_2}\right)\\
    &\,-\left(2n_2^2-n_2-\frac{1}{2} \sum_i \left(s_i^{\ul{p}_3}\right)^2+\frac{1}{2} \sum_{i \text{ odd }} r_i^{\ul{p}_3}\right)\\
    =&\, -n_1-n_2 + \frac{1}{2} \sum_i \left(s_i^{\ul{p}_2}\right)^2-\frac{1}{2} \sum_{i \text{ odd }} r_i^{\ul{p}_2}+ \frac{1}{2} \sum_i \left(s_i^{\ul{p}_3}\right)^2-\frac{1}{2} \sum_{i \text{ odd }} r_i^{\ul{p}_3}.
\end{align*}
Hence, to show \eqref{equ1:keylemma2}, it suffices to show that 
\begin{align}\label{equ2-3:keylemma2}
    \begin{split}
        &\, \sum_i \left(s_i^{\ul{p}_2}\right)^2+  \sum_i \left(s_i^{\ul{p}_3}\right)^2- \sum_i \left(s_i^{\ul{p}_1}\right)^2\\
        =        &\,\sum_{i \text{ odd }} r_i^{\ul{p}_2}+ \sum_{i \text{ odd }} r_i^{\ul{p}_3}-\sum_{i \text{ odd }} r_i^{\ul{p}_1},
    \end{split}
\end{align}
i.e.,
\begin{align}\label{equ3-3:keylemma2}
    \begin{split}
        &\,\sum_i \left(s_i^{\ul{p}_2}\right)^2+  \sum_i \left(s_i^{\ul{p}_3}\right)^2- \sum_i \left(s_i^{\ul{p}_1}\right)^2\\
        =&\, \sum_{i \text{ odd }} \left(s_i^{\ul{p}_2}-s_{i+1}^{\ul{p}_2}\right)+ \sum_{i \text{ odd }} \left(s_i^{\ul{p}_3}-s_{i+1}^{\ul{p}_3}\right)-\sum_{i \text{ odd }} \left(s_i^{\ul{p}_1}-s_{i+1}^{\ul{p}_1}\right).
    \end{split}
\end{align}

Since $I=\{i|b_i \text{ odd}\}$, $J=\{i|b_i \text{ even}\}$, following the recipe in \cite[Lemma 6.3.8]{CM93} on carrying out the $\Sp$-collapse, 
an easy observation is that $\ul{p}_1^t$ is exactly union of $\ul{p}_2^t$ and $\ul{p}_3^t$. On the other hand, as an orthogonal partition, even parts always have even multiplicities. Hence, we obtain that 
$$\sum_i \left(s_i^{\ul{p}_2}\right)^2+  \sum_i \left(s_i^{\ul{p}_3}\right)^2= \sum_i \left(s_i^{\ul{p}_1}\right)^2,$$ 
$$\sum_{i \text{ odd }} \left(s_i^{\ul{p}_3}-s_{i+1}^{\ul{p}_3}\right)=0,$$ and 
$$\sum_{i \text{ odd }} \left(s_i^{\ul{p}_2}-s_{i+1}^{\ul{p}_2}\right)=\sum_{i \text{ odd }} \left(s_i^{\ul{p}_1}-s_{i+1}^{\ul{p}_1}\right).$$
Therefore, both sides of \eqref{equ3-3:keylemma2} are identically zero, and \eqref{equ1:keylemma2} is also true for the case of $\mathrm{G}_n=\SO_{2n}^{\alpha}$. 

This completes the proof of Lemma \ref{keylemma2}. 
\end{proof}

We also have the following dimension identity for nilpotent orbits which has its own interest and is expected to play an important role in studying the general case of Conjecture \ref{cubmfclocal}. Note that this identity is the same as the one in Lemma \ref{keylemma2} when $b_i$ are of the same parity. 

\begin{prop}\label{prop identity}
We have the following dimension identity:
\begin{equation}\label{equ:keylemma}
    \dim_{\frak{g}_{n}}(\eta_{\hat{\frak{g}}_n, \frak{g}_{n}}([b_1^{a_1} \cdots b_r^{a_r}])) = \dim_{\hat{\frak{g}}_n}(([b_1^{a_1} \cdots b_r^{a_r}]^t)_{\widehat{\mathrm{G}}_n}),
\end{equation}
where $([b_1^{a_1} \cdots b_r^{a_r}]^t)_{\widehat{\mathrm{G}}_n}$ is the $\widehat{\mathrm{G}}_n$-collapse of the partition $[b_1^{a_1} \cdots b_r^{a_r}]^t$.
\end{prop}

The proof of the above proposition is similar to that of Lemma \ref{keylemma2}, hence is omitted here. 

begin \begin{rmk}\label{rmk:keylemma-v3}
{In our setting, Lemma \ref{keylemma2} must be replaced by corresponding analogues in which the right hand side is computed in the Lie algebras of the connected stabilizers ${\RG_i'}^{\theta}$ of the points $\theta_{\widehat{\mathrm{G}}_i'}=1 \rtimes \wt{\theta}(N_i)$. Concretely, the desired target partitions are
}\begin{align*}
{(\ul{p}(\psi^1)^t)_{\SO_{2n_1+1}} \text{ and } (\ul{p}(\psi^2)^t)_{\Sp_{2n_2}}, }&{\qquad \text{when } \mathrm{G}_n=\Sp_{2n},}\\
{(\ul{p}(\psi^i)^t)_{\Sp_{2n_i}},\ i=1,2, }&{\qquad \text{when } \mathrm{G}_n=\SO_{2n+1}, \SO_{2n}^{\alpha}.
}\end{align*}
\end{rmk}

\begin{prop}\label{prop:keylemma-v3-all}
{The following dimension identity holds:
}\begin{align*}
\begin{split}
    &{\, \dim(\frak{g}_{n}) - \dim_{\frak{g}_{n}}(\eta_{\hat{\frak{g}}_n, \frak{g}_{n}}(}[{b_1^{a_1} \cdots b_r^{a_r}}]{)) }\\
    {= }&{\, \dim(\frak{g_1'}^{\theta}) + \dim(\frak{g_2
    '}^{\theta}) -\dim_{\frak{g_1'}^{\theta}}((}[{\prod_{i\in I}b_i^{a_i}}]{^t)_{{\RG'_1}^{\theta}})-\dim_{\frak{g_2'}^{\theta}}((}[{\prod_{j\in J}b_j^{a_j}}]{^t)_{{\RG'_2}^{\theta}}).
    }\end{split}
{}\end{align*}
\end{prop}

\begin{proof}
    {The split special odd  orthogonal case is already compatible with Lemma \ref{keylemma2}. The symplectic and special even orthogonal cases can be obtained by a simple codimension comparison as follows. 
}

Assume that $\mathrm{G}_n=\Sp_{2n}$. 
{By Lemma \ref{keylemma2}, it is enough to replace the even orthogonal contribution on the right hand side by the corresponding symplectic contribution. Let
}$${\ul{p}_J=}[{\prod_{j\in J}b_j^{a_j}}]{^t.}$$
{Since $\mathrm{G}_n=\Sp_{2n}$, the partition $[b_1^{a_1} \cdots b_r^{a_r}]$ is orthogonal. Hence every even part $b_j$ occurs with even multiplicity $a_j$, $j \in J$. It follows that every part of $\ul{p}_J$ is even, because the parts of $\ul{p}_J$ are partial sums of the even integers $a_j$. Moreover each part of $\ul{p}_J$ occurs with even multiplicity, since these multiplicities are differences of even parts $b_j$ (with $0$ allowed at the end). Therefore, $\ul{p}_J$ is simultaneously an orthogonal partition and a symplectic partition of $2n_2$, and
}$$
{(\ul{p}_J)_{\SO_{2n_2}}=(\ul{p}_J)_{\Sp_{2n_2}}=\ul{p}_J.
}$$

{Writing $\ul{p}_J=[q_1 q_2 \cdots q_l]$ with $q_1 \geq \cdots \geq q_l$ and using the notation $r_i^{\ul{p}_J}$ and $s_i^{\ul{p}_J}$ from the proof of Lemma \ref{keylemma2}, we have
}$$
{\sum_{i \text{ odd }} r_i^{\ul{p}_J}=0,
}$$
{because all parts of $\ul{p}_J$ are even. Hence, by \cite[Corollary 6.1.4]{CM93},
}\begin{align*}
{\dim(\mathfrak{sp}_{2n_2})-\dim_{\mathfrak{sp}_{2n_2}}(\ul{p}_J)
}&{= \frac{1}{2} \sum_i \left(s_i^{\ul{p}_J}\right)^2 }\\
&{= \dim(\mathfrak{so}_{2n_2})-\dim_{\mathfrak{so}_{2n_2}}(\ul{p}_J).
}\end{align*}
{Thus the even-factor codimension appearing in Lemma \ref{keylemma2} is unchanged when we pass from $\mathfrak{so}_{2n_2}$ to $\mathfrak{sp}_{2n_2}$, and the displayed identity follows.
}

Next we assume that $\mathrm{G}_n=\SO_{2n}^{\alpha}$. 
{By Lemma \ref{keylemma2}, it is enough to replace the two orthogonal factor codimensions on its right hand side by the corresponding symplectic factor codimensions. Put
}\[
{\ul{p}_I=[\prod_{i\in I}b_i^{a_i}]^t,\qquad
\ul{p}_J=[\prod_{j\in J}b_j^{a_j}]^t.}
\]
{If one of $I,J$ is empty, the corresponding rank-zero assertion is understood in the evident sense.}

{The $J$-factor is exactly the same comparison as in the $\Sp_{2n}$ case above. Since $[b_1^{a_1}\cdots b_r^{a_r}]$ is an orthogonal partition, every even $b_j$ has even multiplicity. Hence every part of $\ul{p}_J$ is even, and every part occurs with even multiplicity. Thus
}\[
{(\ul{p}_J)_{\SO_{2n_2}}=(\ul{p}_J)_{\Sp_{2n_2}}=\ul{p}_J,\qquad
\sum_{a\text{ odd}}r_a^{\ul{p}_J}=0.}
\]
{Using the dimension formulas in \cite[Corollary 6.1.4]{CM93}, this gives
}\[
{\dim(\mathfrak{so}_{2n_2})-\dim_{\mathfrak{so}_{2n_2}}(\ul{p}_J)
=\frac{1}{2}\sum_a(s_a^{\ul{p}_J})^2
=\dim(\mathfrak{sp}_{2n_2})-\dim_{\mathfrak{sp}_{2n_2}}(\ul{p}_J).}
\]

{It remains to compare the $I$-factor. Write
}\[
{[\prod_{i\in I}b_i^{a_i}]=[d_1^{c_1}\cdots d_t^{c_t}],\qquad
d_1>\cdots>d_t,\quad d_k\text{ odd},}
\]
{and set $C_k=c_1+\cdots+c_k$. Since the total size $\sum_k c_kd_k$ is even and all $d_k$ are odd, $C_t$ is even. The transpose has the form
}\[
{\ul{p}_I=[C_t^{d_t}C_{t-1}^{d_{t-1}-d_t}\cdots C_1^{d_1-d_2}].}
\]
{For $k<t$, the multiplicity $d_k-d_{k+1}$ is even. Therefore every odd part of $\ul{p}_I$ has even multiplicity, and $\ul{p}_I$ is a symplectic partition; in particular $(\ul{p}_I)_{\Sp_{2n_1}}=\ul{p}_I$. Let
}\[
{\mu_I=(\ul{p}_I)_{\SO_{2n_1}}.}
\]
{The $D$-collapse recipe of \cite[Lemma 6.3.8]{CM93}, applied to the displayed form of $\ul{p}_I$, gives
}\[
{\sum_a(s_a^{\mu_I})^2-\sum_a(s_a^{\ul{p}_I})^2
=\sum_{a\text{ odd}}r_a^{\mu_I}+\sum_{a\text{ odd}}r_a^{\ul{p}_I}.}
\]
{Substituting this equality into the symplectic and even orthogonal dimension formulas in \cite[Corollary 6.1.4]{CM93} yields
}\[
{\dim(\mathfrak{sp}_{2n_1})-\dim_{\mathfrak{sp}_{2n_1}}(\ul{p}_I)
=\dim(\mathfrak{so}_{2n_1})-\dim_{\mathfrak{so}_{2n_1}}(\mu_I).}
\]
{Combining the $I$- and $J$-factor codimension equalities with Lemma \ref{keylemma2} gives the displayed identity.}
\end{proof}

At the end of this section, we record another dimension equality which will be used in later sections. 

\begin{lem}\label{lem:onefactor-soeven-codim}
{Let $\lambda$ be an orthogonal partition of $2n$, and put $\tau=\lambda^t$. Then
}\[
{\dim(\mathfrak{so}_{2n})-\dim_{\mathfrak{so}_{2n}}(\tau_{\SO_{2n}})
=\dim(\mathfrak{sp}_{2n})-\dim_{\mathfrak{sp}_{2n}}(\tau_{\Sp_{2n}}).}
\]
\end{lem}

\begin{proof}
{Set
}\[
{\rho=\tau_{\Sp_{2n}},\qquad \mu=\tau_{\SO_{2n}}.}
\]
{We first recall the compatibility of the two collapses in this special situation. Since $\lambda=\tau^t$ is an orthogonal partition, the explicit collapse recipe of \cite[Lemma 6.3.8]{CM93} gives
}\[
{\mu=(\tau_{\Sp_{2n}})_{\SO_{2n}}=\rho_{\SO_{2n}},}
\]
{and $\rho^t$ is again an orthogonal partition. Indeed, writing
$r_a^\tau=\lambda_a-\lambda_{a+1}$, the parts of $\tau$ which are bad for the $C$-collapse are exactly the odd integers $a$ for which the two consecutive column lengths $\lambda_a$ and $\lambda_{a+1}$ have opposite parity. Let $u,\ldots,v$ be a maximal block of even entries in the ordered partition $\lambda$. Its length is even. Hence the two boundary indices $u-1$ and $v$ (when both boundaries occur) have the same parity. If this common parity is odd, the $C$-collapse pairs the two bad odd parts $v$ and $u-1$ and replaces them by the adjacent even parts $v-1$ and $u$; if this common parity is even, no $C$-move is made at that block and the two boundary parts remain bad for the subsequent $D$-collapse. At the top or bottom of $\lambda$ there is only one boundary, and its index is even, so it is also a $D$-boundary. Thus the $C$-collapse only removes the odd bad boundary pairs coming from even blocks of $\lambda$, and after these replacements the even column-length blocks still have even length. This is the assertion that $\rho^t$ is orthogonal. The remaining bad even parts, with the usual terminal $0$ if needed, are then exactly the $D$-collapse chains which give the direct $D$-collapse of $\tau$; hence the displayed identity for $\mu$.}

{It remains to prove the codimension identity for such a partition $\rho$. We write $r_a^\sigma$ and $s_a^\sigma$ for the multiplicities and column lengths of any partition $\sigma$, as in the proof of Lemma \ref{keylemma2}. By \cite[Corollary 6.1.4]{CM93}, for a symplectic partition $\sigma$ of $2n$ and an orthogonal partition $\nu$ of $2n$,}
\[
{\cod_{\mathfrak{sp}_{2n}}(\sigma)
=\frac{1}{2}\left(\sum_a(s_a^\sigma)^2+\sum_{a\text{ odd}}r_a^\sigma\right),}
\]
\[
{\cod_{\mathfrak{so}_{2n}}(\nu)
=\frac{1}{2}\left(\sum_a(s_a^\nu)^2-\sum_{a\text{ odd}}r_a^\nu\right),}
\]
{where $\cod_{\mathfrak h}$ means $\dim(\mathfrak h)$ minus the orbit dimension. We shall prove
}\[
{\cod_{\mathfrak{so}_{2n}}(\rho_{\SO_{2n}})
=\cod_{\mathfrak{sp}_{2n}}(\rho).}
\]

{Let $\nu=\rho_{\SO_{2n}}$. Since $\rho$ is symplectic, its odd parts occur with even multiplicity; since $\rho^t$ is orthogonal, every odd part of $\rho$ lies between two ends of one of the $D$-collapse chains below. Group the even parts of $\rho$ which have odd multiplicity into the standard $D$-collapse chains. One such chain has even entries
}\[
{2d_0>2d_1>\cdots>2d_h\geq 0,}
\]
{where $2d_0$ and $2d_h$ have odd multiplicity (with $2d_h=0$ allowed as the terminal dummy part), and $2d_1,\ldots,2d_{h-1}$ are the even parts of $\rho$ strictly between them. On this chain the $D$-collapse replaces one copy of $2d_0$ by $2d_0-1$, one copy of $2d_h$ by $2d_h+1$ if $d_h>0$ (and adds one copy of $1$ if $d_h=0$), and for each intermediate even part $2d_j$ replaces two copies of $2d_j$ by $2d_j+1$ and $2d_j-1$. Odd parts already present in $\rho$ are left unchanged.}

{For this chain, the changes in the column lengths are exactly as follows: $s_{2d_0}$ decreases by $1$, $s_{2d_h+1}$ increases by $1$, and for each intermediate $j$ both $s_{2d_j}$ decreases by $1$ and $s_{2d_j+1}$ increases by $1$. Denote by $\Delta_D$ the contribution of this single $D$-collapse chain to the total change $\sum_a(s_a^\nu)^2-\sum_a(s_a^\rho)^2$. Then}
\[
{\Delta_D=-2s_{2d_0}^{\rho}+1+2s_{2d_h+1}^{\rho}+1
\;+\sum_{j=1}^{h-1}\left(2(s_{2d_j+1}^{\rho}-s_{2d_j}^{\rho})+2\right).}
\]
{Put
}\[
{I_D=\{a\mid a\text{ is odd and }2d_h<a<2d_0\}.}
\]
{Since $s_{2d_h+1}^{\rho}-s_{2d_0}^{\rho}$ counts all parts of $\rho$ with sizes in the interval $[2d_h+1,2d_0-1]$, and since
$s_{2d_j+1}^{\rho}-s_{2d_j}^{\rho}=-r_{2d_j}^{\rho}$, we get the exact cancellation}
\[
{s_{2d_h+1}^{\rho}-s_{2d_0}^{\rho}
+\sum_{j=1}^{h-1}(s_{2d_j+1}^{\rho}-s_{2d_j}^{\rho})
=\sum_{a\in I_D}r_a^\rho.}
\]
{Thus this single-chain contribution is}
\[
{\Delta_D=2\sum_{a\in I_D}r_a^\rho+2h.}
\]
{Now compare this with the odd multiplicities. The $D$-collapse does not change any old odd part, and the only odd parts created by this chain lie in $I_D$: one from $2d_0$, one from $2d_h$ (with $2d_h=0$ contributing the new part $1$), and two from each intermediate $2d_j$. Counted with multiplicity, this creates exactly $2h$ odd parts in $I_D$. Hence}
\[
{\sum_{a\in I_D}r_a^\nu
=\sum_{a\in I_D}r_a^\rho+2h,}
\]
{and therefore}
\[
{\Delta_D=\sum_{a\in I_D}r_a^\nu+\sum_{a\in I_D}r_a^\rho.}
\]
{The intervals $I_D$ for the distinct collapse chains are disjoint and together account for exactly the odd parts whose multiplicities enter the change. Summing over the chains gives}
\[
{\sum_a(s_a^\nu)^2-\sum_a(s_a^\rho)^2
=\sum_{a\text{ odd}}r_a^\nu+\sum_{a\text{ odd}}r_a^\rho.}
\]
{Substituting this identity into the two codimension formulas above proves}
\[
{\cod_{\mathfrak{so}_{2n}}(\rho_{\SO_{2n}})
=\cod_{\mathfrak{sp}_{2n}}(\rho).}
\]
{Since $\rho=\tau_{\Sp_{2n}}$ and $\rho_{\SO_{2n}}=\tau_{\SO_{2n}}$, this is the desired identity.}
\end{proof}

end \section{Construction of elements in local Arthur packets}\label{nonvan}

In this section, we construct a particular element in each local Arthur packet which plays an important role towards proving Part (3) of Conjecture \ref{cubmfclocal}. 

First, we recall
from \cite{JS04, Liu11, JL14, JL22} on how to associate
an irreducible representation of $G_{n}$ to each local
Langlands parameter $\phi \in \Phi(G_{n})$. 
The key idea is to analyze the structure of each local
Langlands parameter.

\begin{prop}[\cite{JS04, Liu11, JL14, JL22}]\label{prop2}
Given a $\phi \in \Phi(G_{n})$. Then either $\phi \in \Phi^{(t)}(G_{n})$, 
or 
\begin{equation} \label{equ71}
\phi = \phi^{(t)} \oplus \phi^{(n)}, 
\end{equation}
where $\phi^{(t)} \in \Phi^{(t)}(G_{n^{*}})$ $(n^{*} < n)$ and 
$\phi^{(n)} \in \Phi(G'_{n-n^{*}})$ which is of the form
\begin{equation} \label{equ67}
\phi^{(n)} = \bigoplus_{i=1}^f \left[|\cdot|^{-q_i+\frac{w_i}{2}} \phi_i \otimes
S_{w_i+1} \oplus |\cdot|^{q_i-\frac{w_i}{2}} \widetilde{\phi_i} \otimes
S_{w_i+1}\right], 
\end{equation}
where 
$$\mathrm{G}'_{n-n^*}=\begin{cases}
\SO^{\alpha}_{2(n-n^*)}, & \text{ when } \mathrm{G}_n=\Sp_{2n}, \SO_{2n}^{\alpha},\\
\SO_{2(n-n^*)+1}, & \text{ when } \mathrm{G}_n=\SO_{2n+1},
\end{cases}$$
$f \in \mathbb{Z}_{>0}$,
$w_1, w_2, \dots, w_f \in \mathbb{Z}_{\geq0}$, $q_1, q_2, \dots, q_f \in \mathbb{R}$, 
such that for $1 \leq i \leq f$,
$q_i \neq \frac{w_i}{2}$, $\phi_i$ is an
irreducible bounded representation of $W_F$, and for $1 \leq i \leq f-1$,
$$\frac{w_i}{2} - q_i \geq \frac{w_{i+1}}{2} - q_{i+1} > 0,$$
$|\cdot|^s$ is the character of $W_F$
normalized as in \cite{Tat79} via local class field theory.
\end{prop}

Given a $\phi \in \Phi(G_{n})$, 
$\phi = \phi^{(t)} \oplus \phi^{(n)}$, as in Proposition \ref{prop2}. 
By \cite{JS04, Liu11, JL14, JL22}, there exists
$\sigma^{(t)} \in \Pi^{(tg)}(G_{n^{*}})$, such that
\begin{equation} \label{equ69}
\iota(\sigma^{(t)}) = \phi^{(t)}. 
\end{equation}
Using the local Langlands reciprocity map $r$ for general linear groups, define
\begin{equation} \label{equ70}
\Sigma_i = [v^{-q_i}r(\phi_i), v^{-q_i+w_i}r(\phi_i)], 1 \leq i
\leq f.
\end{equation}
Let $\sigma$ be the Langlands quotient
of the induced representation
$$\delta(\Sigma_1) \times \delta(\Sigma_2) \times \cdots \times 
\delta(\Sigma_f) \rtimes \sigma^{(t)}.$$ 
Then $\sigma$ is an irreducible representation with local Langlands parameter $\phi$. Recall that $\delta(\Sigma_i)$ is the irreducible essential square-integrable representation attached to the segment $\Sigma_i$. 

Next, we apply the above discussion to construct a particular element in each local Arthur packet. 
Given a local Arthur parameter $$\psi: W_F \times \SL_2(\mathbb{C}) \times \SL_2(\mathbb{C}) \rightarrow {}^L\mathrm{G}_{n}$$
$$\psi = \bigoplus_{i=1}^r \phi_i \otimes S_{m_i} \otimes S_{n_i},$$
as in \eqref{lap}, we have 
$$\phi_{\psi}=\bigoplus_{i=1}^r 
\bigoplus_{j=-\frac{n_i-1}{2}}^{\frac{n_i-1}{2}} |w|^j\phi_i(w) \otimes S_{m_i} (x).$$

Let 
$$\phi^{(t)}_{\psi}=
\bigoplus_{i=1, n_i \text{ odd }}^r 
\phi_i(w) \otimes S_{m_i} (x),$$
and 
$$\phi^{(n)}_{\psi}=\bigoplus_{i=1}^r 
\bigoplus_{j=-\frac{n_i-1}{2}, j\neq 0}^{\frac{n_i-1}{2}} |w|^j\phi_i(w) \otimes S_{m_i} (x).$$
Then $\phi^{(t)}_{\psi} \in \Phi^{(t)}(G_{n^{*}})$ and 
$\phi^{(n)}_{\psi} \in \Phi(G'_{n-n^{*}})$, where 
$$n^{*}=\bigg\lfloor\frac{\sum_{n_i \text{ odd }} k_im_i}{2}\bigg\rfloor.$$  
By \cite{JS04, Liu11, JL14, JL22}, there exists
$\sigma^{(t)} \in \Pi^{(tg)}(G_{n^{*}})$, where ``$tg$" means tempered generic, such that
\begin{equation} \label{equ692}
\iota(\sigma^{(t)}) = \phi^{(t)}_{\psi}. 
\end{equation}
Using the local Langlands reciprocity map $r$ for general linear groups, define
\begin{equation} \label{equ702}
\Sigma_i^j = [v^{j-\frac{m_i-1}{2}}r(\phi_i), v^{j+\frac{m_i-1}{2}}r(\phi_i)], 1 \leq i
\leq r, 0 < j \leq \frac{n_i-1}{2}.
\end{equation}
Shuffle the set $\{0 <  j \leq \frac{n_i-1}{2} | 1 \leq i \leq r\}$ as $\{j_1, \ldots, j_{n'}\}$ such that $j_1 \geq j_2 \geq \cdots \geq j_{n'}$,  where $n'=\sum_{i=1}^r \frac{n_i-1}{2}$. 
Let $\sigma$ be the Langlands quotient
of the induced representation
$$\delta(\Sigma_1) \times \delta(\Sigma_2) \times \cdots \times 
\delta(\Sigma_{n'}) \rtimes \sigma^{(t)}.$$ 
Then $\sigma$ is an irreducible representation with local Langlands parameter $\phi_{\psi}$ and $\sigma$ is in the local Arthur packet corresponding to $\psi$. 

Recall that for $1 \leq i \leq r$, we let $a_i=k_im_i$ and $b_i=n_i$. Let 
$$\Delta_i=\delta[v^{-\frac{m_i-1}{2}}r(\phi_i), v^{\frac{m_i-1}{2}}r(\phi_i)]$$ be the irreducible square-integrable representation attached to the balanced segment 
$[v^{-\frac{m_i-1}{2}}r(\phi_i), v^{\frac{m_i-1}{2}}r(\phi_i)]$, namely, $\Delta_i$ is the unique irreducible subrepresentation of the following induced representation
$$v^{\frac{m_i-1}{2}}r(\phi_i) \times v^{\frac{m_i-2}{2}}r(\phi_i)\times \cdots \times v^{\frac{1-m_i}{2}}r(\phi_i).$$
Let $\zeta(\Delta_i, \big\lfloor \frac{b_i}{2} \big\rfloor)$ be the unique irreducible quotient of the following induced representation
$$v^{\frac{\lfloor \frac{b_i}{2} \rfloor-1}{2}}\Delta_i \times v^{\frac{\lfloor \frac{b_i}{2} \rfloor-2}{2}}\Delta_i \times \cdots 
\times 
v^{\frac{1-\lfloor \frac{b_i}{2} \rfloor}{2}}\Delta_i.$$

Since the segments $v^{\lceil \frac{b_i}{2} \rceil} \left[v^{\frac{1-\lfloor \frac{b_i}{2} \rfloor}{2}}\Delta_i, v^{\frac{\lfloor \frac{b_i}{2} \rfloor-1}{2}}\Delta_i\right]$ are pairwise non-linked, one can easily see the following

\begin{lem}\label{rhoirreducible}
$$\rho = \times_{i=1}^r 
    v^{\frac{\lceil \frac{b_i}{2} \rceil}{2}}
\zeta\left(\Delta_i, \big\lfloor \frac{b_i}{2} \big\rfloor\right)$$
is irreducible and is the Langlands quotient of the following induced representation
$$\delta(\Sigma_1) \times \delta(\Sigma_2) \times \cdots \times 
\delta(\Sigma_{n'}).$$
$\sigma$ is the unique irreducible quotient of the following induced representation 
$$\rho \rtimes \sigma^{(t)}.$$
\end{lem}

\begin{lem}\label{glpart}
$$\frak{p}^m(\rho)=\left\{\left[\big\lfloor \frac{b_1}{2} \big\rfloor^{a_1}
\big\lfloor \frac{b_2}{2} \big\rfloor^{a_2} \cdots \big\lfloor \frac{b_r}{2} \big\rfloor^{a_r}\right]^t\right\}.$$
\end{lem}

\begin{proof}
First, by 
\cite[Theorem 5]{CFK18}, 
$\frak{p}^m(\zeta(\Delta_i, \big\lfloor \frac{b_i}{2} \big\rfloor))=\left\{\left[a_i^{\big\lfloor \frac{b_i}{2} \big\rfloor}\right]\right\}.$ 
By \cite[Theorems E an F]{GGS17},
the highest derivative of $\zeta(\Delta_i, \big\lfloor \frac{b_i}{2} \big\rfloor)$ is 
$\zeta(\Delta_i, \big\lfloor \frac{b_i}{2} \big\rfloor)^{a_i} =
    v^{\lceil \frac{b_i}{2} \rceil-\frac{1}{2}}
\zeta(\Delta_i, \big\lfloor \frac{b_i}{2} \big\rfloor-1)$. 
Then by  
\cite[Lemma 4.5]{BZ77}, 
the highest derivative of $\rho$ is $\rho^{(a_1 + \cdots + a_r)}$, and 
$$\rho^{(a_1 + \cdots + a_r)} = \times_{i=1}^r 
    v^{\lceil \frac{b_i}{2} \rceil}
\zeta\left(\Delta_i, \big\lfloor \frac{b_i}{2} \big\rfloor\right)^{a_i}=\times_{i=1}^r v^{\lceil \frac{b_i}{2} \rceil-\frac{1}{2}}
\zeta\left(\Delta_i, \big\lfloor \frac{b_i}{2} \big\rfloor-1\right).$$
We repeat the above argument to 
$\rho^{(a_1 + \cdots + a_r)}$ and obtain the lemma. 
\end{proof}

Let $\underline{p}_1=[\big\lfloor \frac{b_1}{2} \big\rfloor^{a_1}
\big\lfloor \frac{b_2}{2} \big\rfloor^{a_2} \cdots \big\lfloor \frac{b_r}{2} \big\rfloor^{a_r}]^t$. Recall that $n^{*}=\big\lfloor\frac{\sum_{b_i \text{ odd }} a_i}{2}\big\rfloor.$

\begin{lem}\label{total}
\begin{enumerate}
    \item $
[\underline{p}_1\underline{p}_1 (2n^*)] \in \frak{p}(\sigma), \text{ when } \mathrm{G}_n=\Sp_{2n}$,\\
\item $[\underline{p}_1\underline{p}_1(2n^*+1)] \in \frak{p}(\sigma), \text{ when } \mathrm{G}_n=\SO_{2n+1}$,\\
    \item $
[\underline{p}_1\underline{p}_1(2n^*-1)1] \in \frak{p}(\sigma), \text{ when } \mathrm{G}_n= \SO_{2n}^{\alpha}$.
\end{enumerate}
\end{lem}

\begin{proof}
First one can take a Whittaker pair $(s,u)$ such that $u$ is a representative of a nilpotent orbits corresponding to the partition $[\underline{p}_1\underline{p}_1(2n^*)]$, or $[\underline{p}_1\underline{p}_1(2n^*+1)]$, or $[\underline{p}_1\underline{p}_1(2n^*-1)1]$, 
and the twisted Jacquet module $\sigma_{N_{s,u}, \psi_u}$ is the usual Jacquet module of $\sigma$ with respect to the parabolic subgroup with Levi isomorphic to $GL_{\sum_{i=1}^r a_i\big\lfloor \frac{b_i}{2} \big\rfloor} \times G_{n^*}$, composed with the twisted Jacquet module corresponding to the partitions $\underline{p}_1$ and $2n^*$, or $(2n^*+1)$, or $(2n^*-1)1$, on $\rho$ and $\sigma$, respectively.
Because $\sigma$ is generic and by Lemma \ref{glpart}, $\sigma_{N_{s,u}, \psi_u}$ is nonzero. 
The the lemma follows from Proposition \ref{ggslocal}. 
\end{proof}

\begin{lem}\label{raise}
\begin{enumerate}
    \item $[\underline{p}_1\underline{p}_1(2n^*)]^{\mathrm{G}_n} \in \frak{p}(\sigma), \text{ when } \mathrm{G}_n=\Sp_{2n}$,\\
\item $[\underline{p}_1\underline{p}_1(2n^*+1)]^{\mathrm{G}_n} \in \frak{p}(\sigma), \text{ when } \mathrm{G}_n=\SO_{2n+1}$,\\
    \item $[\underline{p}_1\underline{p}_1(2n^*-1)1]^{\mathrm{G}_n} \in \frak{p}(\sigma), \text{ when } \mathrm{G}_n= \SO_{2n}^{\alpha}$.
\end{enumerate}
\end{lem}

\begin{proof}
This follows directly from \cite[Theorem 11.2]{JLS16}. Note that $[\underline{p}_1\underline{p}_1(2n^*)]$ ($[\underline{p}_1\underline{p}_1(2n^*+1)]$, $[\underline{p}_1\underline{p}_1(2n^*-1)1]$, respectively), may not be $\mathrm{G}_n$-special. 
\end{proof}

The following theorem proves certain cases for Part (3) of Conjecture \ref{cubmfclocal}. We recall some
notation from \cite[Section 3.1]{Ach03} as follows.
Given any partition $\ul{p}=[p_1 \cdots p_r]$ with $p_1 \geq \cdots \geq p_r$, $\ul{p}^-=[p_1 \cdots (p_r-1)]$, $\ul{p}_-=[(p_1-1) \cdots p_r]$, 
$\ul{p}^+=[(p_1+1) \cdots p_r]$,
$\ul{p}_+=[p_1 \cdots p_r 1]$. Note that $\ul{p}_-=\ul{p}^{t-t}$, $\ul{p}_+=\ul{p}^{t+t}$.

\begin{thm}\label{mainthm1}
Let $\psi$ be a local Arthur parameter as in \eqref{lap}, with $\ul{p}(\psi) = [b_1^{a_1} b_2^{a_2} \cdots b_r^{a_r}]$ and $b_1 \geq b_2 \geq \cdots \geq b_r$. 
\begin{enumerate}
    \item When $\mathrm{G}_n=\Sp_{2n}$,
    \begin{equation}\label{oriequ}
    [\underline{p}_1\underline{p}_1(2n^*)]^{\Sp_{2n}}=\eta_{\frak{so}_{2n+1}, \frak{sp}_{2n}}([b_1^{a_1} \cdots b_r^{a_r}])
    \end{equation}
if and only if 
\begin{equation}\label{criterion}
    ([\underline{p}_1\underline{p}_1(2n^*)]^t)_{\Sp_{2n}}=([b_1^{a_1} \cdots b_r^{a_r}]^-)_{\Sp_{2n}}.
\end{equation}
In particular, if 
\begin{enumerate}
    \item[(i)] $a_r=b_r=1$ and $b_i$ are all even for $1 \leq i \leq r-1$,
    \item[(ii)] or, $b_i$ are all odd,
\end{enumerate}
 then \eqref{criterion} holds. 

 \item When $\mathrm{G}_n=\SO_{2n+1}$, 
 \begin{equation}\label{oriequ2}
     [\underline{p}_1\underline{p}_1(2n^*+1)]^{\SO_{2n+1}}=\eta_{\frak{sp}_{2n}, \frak{so}_{2n+1}}([b_1^{a_1} \cdots b_r^{a_r}])
 \end{equation}
if and only if 
\begin{equation}\label{criterion2}
    ([\underline{p}_1\underline{p}_1(2n^*+1)]^t)_{\SO_{2n+1}}=([b_1^{a_1} \cdots b_r^{a_r}]^+)_{\SO_{2n+1}}.
\end{equation}
        In particular, if \begin{enumerate}
    \item[(i)] $b_1$ is even and $a_1=1$, and $b_i$ are all odd for $2 \leq i \leq r$,
    \item[(ii)] or, $b_i$ are all even,
\end{enumerate}
then 
\eqref{criterion2} holds. 

 \item Assume $\mathrm{G}_n=\SO_{2n}^{\alpha}$.
 If all $b_i$ are of the same parity, then 
 \begin{equation}\label{criterion3}
    [\underline{p}_1\underline{p}_1(2n^*-1)1]^{\SO_{2n}}=\eta_{\frak{o}_{2n}, \frak{o}_{2n}}([b_1^{a_1} \cdots b_r^{a_r}]).
 \end{equation}
\end{enumerate}
\end{thm}

We remark that the identities \eqref{criterion} and \eqref{criterion2} are relatively easier to check than \eqref{oriequ} and \eqref{oriequ2}. 

\begin{proof}
Proof of Part (1), $\mathrm{G}_n=\Sp_{2n}$.
First note that given any symplectic partition $\ul{p}$, $(\ul{p}^{\Sp_{2n}})^t = (\ul{p}^t)_{\Sp_{2n}}$, see the proof of \cite[Theorem 6.3.11]{CM93}.
Since $[\underline{p}_1\underline{p}_1(2n^*)]$ is clearly a symplectic partition, we have
\MATHBLOCKDOLLARDOLLAR[[\underline{p}_1\underline{p}_1(2n^*)]]{^{\Sp_{2n}}=
(([\underline{p}_1\underline{p}_1(2n^*)]^t)_{\Sp_{2n}})^t.}

On the other hand, 
$$\eta_{\frak{so}_{2n+1}, \frak{sp}_{2n}}([b_1^{a_1} \cdots b_r^{a_r}])=(([b_1^{a_1} \cdots b_r^{a_r}]^t)^-)_{\Sp_{2n}}.$$
Since $\psi$ is an Arthur parameter for symplectic groups, $[b_1^{a_1} \cdots b_r^{a_r}]$ is an orthogonal partition.
Hence, 
$$\eta_{\frak{so}_{2n+1}, \frak{sp}_{2n}}([b_1^{a_1} \cdots b_r^{a_r}])=(([b_1^{a_1} \cdots b_r^{a_r}]^-)_{\Sp_{2n}})^t,$$
by \cite[Lemma 3.3]{Ach03} which says that given any orthogonal partition $\ul{p}$, 
$((\ul{p}^-)_{\Sp_{2n}})^t=((\ul{p}^t)^-)_{\Sp_{2n}}$.

Therefore, 
\MATHBLOCKDOLLARDOLLAR[[\underline{p}_1\underline{p}_1(2n^*)]]{^{\Sp_{2n}}=\eta_{\frak{so}_{2n+1}, \frak{sp}_{2n}}([b_1^{a_1} \cdots b_r^{a_r}])}
if and only if 
$$(([\underline{p}_1\underline{p}_1(2n^*)]^t)_{\Sp_{2n}})^t=(([b_1^{a_1} \cdots b_r^{a_r}]^-)_{\Sp_{2n}})^t.$$
Hence, we prove the equivalence, since taking transpose is a bijection. 

Next, assume first that $a_r=b_r=1$ and $b_i$ are all even for $1 \leq i \leq r-1$. 
By definition, $2n^*=a_1-1=0$. 
Recall that $\underline{p}_1=[\big\lfloor \frac{b_1}{2} \big\rfloor^{a_1}
\big\lfloor \frac{b_2}{2} \big\rfloor^{a_2} \cdots \big\lfloor \frac{b_r}{2} \big\rfloor^{a_r}]^t = a_1^{\big\lfloor \frac{b_1}{2} \big\rfloor} + a_2^{\big\lfloor \frac{b_2}{2} \big\rfloor} + \cdots + a_r^{\big\lfloor \frac{b_r}{2} \big\rfloor}$.
Hence, 
\MATHBLOCKDOLLARDOLLAR[[\underline{p}_1\underline{p}_1(2n^*)]]{= [a_1^{2\big\lfloor \frac{b_1}{2} \big\rfloor} + a_2^{2\big\lfloor \frac{b_2}{2} \big\rfloor} + \cdots + a_r^{2\big\lfloor \frac{b_r}{2} \big\rfloor}]=([b_1^{a_1} \cdots b_r^{a_r}]^t)_-.}
Then, by \cite[Section 3.1]{Ach03}, 
\MATHBLOCKDOLLARDOLLAR[[\underline{p}_1\underline{p}_1(2n^*)]]{= ([b_1^{a_1} \cdots b_r^{a_r}]^t)^{t-t}=[b_1^{a_1} \cdots b_r^{a_r}]^{-t}.}
Therefore, \MATHBLOCKDOLLARDOLLAR[[\underline{p}_1\underline{p}_1(2n^*)]]{^t=[b_1^{a_1} \cdots b_r^{a_r}]^-.}
Hence, \eqref{criterion} holds. 

Now we assume that $b_i$ are all odd. 
Since $\psi$ is an Arthur parameter for symplectic groups, and $b_i$ are all odd, $\sum_{i=1}^r a_i$ should be odd and 
$2n^*= (\sum_{i=1}^r a_i) -1$. Hence, similarly as above,
\MATHBLOCKDOLLARDOLLAR[[\underline{p}_1\underline{p}_1(2n^*)]]{=([b_1^{a_1} \cdots b_r^{a_r}]^t)_-= ([b_1^{a_1} \cdots b_r^{a_r}]^t)^{t-t}=[b_1^{a_1} \cdots b_r^{a_r}]^{-t}.}
Therefore, \MATHBLOCKDOLLARDOLLAR[[\underline{p}_1\underline{p}_1(2n^*)]]{^t=[b_1^{a_1} \cdots b_r^{a_r}]^-.}
Hence, \eqref{criterion} also holds.

Proof of Part (2), $\mathrm{G}_n=\SO_{2n+1}$.
First, as in the proof of Part (1), note that given any orthogonal partition $\ul{p}$, $(\ul{p}^{\SO_{2n+1}})^t = (\ul{p}^t)_{\SO_{2n+1}}$, see the proof of \cite[Theorem 6.3.11]{CM93}.
Since $[\underline{p}_1\underline{p}_1(2n^*+1)]$ is clearly an orthogonal partition, we have
\MATHBLOCKDOLLARDOLLAR[[\underline{p}_1\underline{p}_1(2n^*+1)]]{^{\SO_{2n+1}}=
(([\underline{p}_1\underline{p}_1(2n^*+1)]^t)_{\SO_{2n+1}})^t.}

On the other hand, 
$$\eta_{\frak{sp}_{2n}, \frak{so}_{2n+1}}([b_1^{a_1} \cdots b_r^{a_r}])=(([b_1^{a_1} \cdots b_r^{a_r}]^t)^+)_{\SO_{2n+1}}.$$
Since $\psi$ is an Arthur parameter for odd orthogonal groups, $[b_1^{a_1} \cdots b_r^{a_r}]$ is a symplectic partition.
Hence, 
$$\eta_{\frak{sp}_{2n}, \frak{so}_{2n+1}}([b_1^{a_1} \cdots b_r^{a_r}])=(([b_1^{a_1} \cdots b_r^{a_r}]^+)_{\SO_{2n+1}})^t,$$
by \cite[Lemma 3.3]{Ach03} which says that given any symplectic partition $\ul{p}$, 
$((\ul{p}^+)_{\SO_{2n+1}})^t=((\ul{p}^t)^+)_{\SO_{2n+1}}$.

Therefore, 
\MATHBLOCKDOLLARDOLLAR[[\underline{p}_1\underline{p}_1(2n^*+1)]]{^{\SO_{2n+1}}=\eta_{\frak{sp}_{2n}, \frak{so}_{2n+1}}([b_1^{a_1} \cdots b_r^{a_r}])}
if and only if 
$$(([\underline{p}_1\underline{p}_1(2n^*+1)]^t)_{\SO_{2n+1}})^t=(([b_1^{a_1} \cdots b_r^{a_r}]^+)_{\SO_{2n+1}})^t.$$
Hence, we prove the equivalence. 

Now we assume that $b_1$ is even and $a_1=1$, and $b_i$ are all odd, for $2 \leq i \leq r$. Then $2n^*+1=(\sum_{i=2}^r a_i) +1=\sum_{i=1}^r a_i$. Hence, 
\begin{align*}
    &\,[\underline{p}_1\underline{p}_1(2n^*+1)]^t\\
    = &\,2\left[\big\lfloor \frac{b_1}{2} \big\rfloor^{a_1}
\big\lfloor \frac{b_2}{2} \big\rfloor^{a_2} \cdots \big\lfloor \frac{b_r}{2} \big\rfloor^{a_r}\right]+\left[1^{\sum_{i=1}^r a_i}\right]\\
=&\,\left[b_1\left(\prod_{i=2}^r (b_i-1)^{a_i}\right)\right]+\left[1^{\sum_{i=1}^r a_i}\right]\\
=&\,\left[(b_1+1)\left(\prod_{i=2}^r (b_i)^{a_i}\right)\right],
\end{align*}
which is exactly $[b_1^{a_1} \cdots b_r^{a_r}]^+$.
Therefore, \eqref{criterion2} holds.

Next, we assume that $b_i$ are all even, then we have 
$n^*= 0$. Hence, similarly as above,
\MATHBLOCKDOLLARDOLLAR[[\underline{p}_1\underline{p}_1(2n^*+1)]]{^t=2\left[\big\lfloor \frac{b_1}{2} \big\rfloor^{a_1}
\big\lfloor \frac{b_2}{2} \big\rfloor^{a_2} \cdots \big\lfloor \frac{b_r}{2} \big\rfloor^{a_r}\right]+[1]=[b_1^{a_1} \cdots b_r^{a_r}]+[1],}
which is exactly $[b_1^{a_1} \cdots b_r^{a_r}]^+$. 
Therefore, \eqref{criterion2} also holds.

Proof of Part (3), $\mathrm{G}_n=\SO_{2n}^{\alpha}$.
Without of loss of generality, we assume further that $b_1>b_2>\cdots>b_r$. 
First, we assume that all $b_i$ are even. Since $\psi$ is an Arthur parameter for even orthogonal groups, $[b_1^{a_1} \cdots b_r^{a_r}]$ is an orthogonal partition of $2n$. Hence, all $a_i$'s are also even, and $n^*=0$. In this case, $(2n^*-1)1$ will be omitted. Then
\MATHBLOCKDOLLARDOLLAR[[\underline{p}_1\underline{p}_1(2n^*-1)1]]{=[\underline{p}_1\underline{p}_1]= a_1^{2\big\lfloor \frac{b_1}{2} \big\rfloor} + a_2^{2\big\lfloor \frac{b_2}{2} \big\rfloor} + \cdots + a_r^{2\big\lfloor \frac{b_r}{2} \big\rfloor} = a_1^{b_1} + \cdots + a_r^{b_r},}
which is a special even orthogonal partition. Therefore, \MATHBLOCKDOLLARDOLLAR[[\underline{p}_1\underline{p}_1(2n^*-1)1]]{^{\SO_{2n}}=[\underline{p}_1\underline{p}_1(2n^*-1)1].}

On one hand, 
\MATHBLOCKDOLLARDOLLAR[[\underline{p}_1\underline{p}_1(2n^*-1)1]]{^t=2\left[\big\lfloor \frac{b_1}{2} \big\rfloor^{a_1}
\big\lfloor \frac{b_2}{2} \big\rfloor^{a_2} \cdots \big\lfloor \frac{b_r}{2} \big\rfloor^{a_r}\right]=[b_1^{a_1} \cdots b_r^{a_r}].}
On the other hand, 
\begin{align*}
    &\,(\eta_{\frak{o}_{2n}, \frak{o}_{2n}}([b_1^{a_1} \cdots b_r^{a_r}]))^t\\
    =&\, (([b_1^{a_1} \cdots b_r^{a_r}]^t)_{\SO_{2n}})^t\\
    =&\, ([b_1^{a_1} \cdots b_r^{a_r}]^{+-})_{\Sp_{2n}}\\
    =&\,[b_1^{a_1} \cdots b_r^{a_r}].
\end{align*}
Here we used the fact that by \cite[Lemma 3.3]{Ach03}, given a partition $\ul{p}$ of $2n$, 
if it is an orthogonal partition or its transpose is a symplectic partition, then 
$(\ul{p}^t)_{\SO_{2n}}=((\ul{p}^{+-})_{\Sp_{2n}})^t$. 
Therefore, we have shown that 
\eqref{criterion3} holds.

Next, we assume that all $b_i$ are odd. Since 
$[b_1^{a_1} \cdots b_r^{a_r}]$ is an orthogonal partition of $2n$, $\sum_{i=1}^r a_i$ should be even. Hence, $2n^*=\sum_{i=1}^r a_i$.
Then, we have
\begin{align*}
    &\,[\underline{p}_1\underline{p}_1(2n^*-1)1]^t\\
    = &\,2\left[\big\lfloor \frac{b_1}{2} \big\rfloor^{a_1}
\big\lfloor \frac{b_2}{2} \big\rfloor^{a_2} \cdots \big\lfloor \frac{b_r}{2} \big\rfloor^{a_r}\right]+\left[1^{(\sum_{i=1}^r a_i)-1}\right] + [1]\\
=&\,\left[\prod_{i=1}^r (b_i-1)^{a_i})\right]+\left[1^{(\sum_{i=1}^r a_i)-1}\right] + [1]\\
=&\,\left[(b_1+1)b_1^{a_1-1}\left(\prod_{i=2}^{r-1} (b_i)^{a_i}\right) b_r^{a_r-1}(b_r-1)\right]\\
=&\,[b_1^{a_1} \cdots b_r^{a_r}]^{+-}.
\end{align*}
Hence, 
\begin{align*}
    &\,[\underline{p}_1\underline{p}_1(2n^*-1)1] \\
    =&\, ([b_1^{a_1} \cdots b_r^{a_r}]^{+-})^t\\
    =&\, ([b_1^{a_1} \cdots b_r^{a_r}]^t)^{(t+t)(t-t)}\\
    =&\,([b_1^{a_1} \cdots b_r^{a_r}]^t)_{+-}.
\end{align*}

Rewrite $[b_1^{a_1} \cdots b_r^{a_r}]$ as $[(\prod_{i=1}^s n_i^2)(\prod_{j=1}^l(2m_j+1))]$, where $l$ is even, and $2m_l+1>\cdots>2m_1+1$. We can see that  
$[b_1^{a_1} \cdots b_r^{a_r}]^t$
is a partition of the following form
$$\left[p_l^1 \cdots p_l^{2m_1+1} \left(\prod_{j=1}^{l-1} p_{j}^1\cdots p_{j}^{2m_{l+1-j}-2m_{l-j}}\right)p_0^1 \cdots p_0^{m_0}\right],$$
where $p_l^{i}$ with $1\leq i \leq 2m_1+1$, $p_{2j}^i$ with $1\leq i \leq 2m_{l+1-2j}-2m_{l-2j}$ and $1\leq j \leq \frac{l-2}{2}$, and $p_0^k$ with $1 \leq k \leq m_0$  are all even; and $p_{2j+1}^i$ with $1\leq i \leq 2m_{l-2j}-2m_{l-2j-1}$ and $0\leq j \leq \frac{l-2}{2}$
are all odd; and finally $p_l^1 \geq \cdots \geq p_l^{2m_1+1}>p_{l-1}^1$, $p_j^1 \geq \cdots \geq p_j^{2m_{l+1-j}-2m_{l-j}} > p_{j-1}^1$, $1 \leq j \leq l-1$ with  $p_0^1 \geq \cdots \geq p_0^{m_0}$. Since $b_1 > b_2 > \cdots > b_r$, we observe that $m_0$ is even and $p_l^1$ occurs with odd multiplicity. 
Then, $[\underline{p}_1\underline{p}_1(2n^*-1)1] = ([b_1^{a_1} \cdots b_r^{a_r}]^t)_{+-}$ is the following partition
$$\left[(p_l^1-1)p_l^2 \cdots p_l^{2m_1+1} \left(\prod_{j=1}^{l-1} p_{j}^1\cdots p_{j}^{2m_{l+1-j}-2m_{l-j}}\right)p_0^1 \cdots p_0^{m_0}(1)\right].$$

Following the recipe on carrying out the $\SO_{2n}$-collapse (\cite[Lemma 6.3.8]{CM93}), we obtain that the partition $([b_1^{a_1} \cdots b_r^{a_r}]^t)_{\SO_{2n}}$
is equal to the 
following partition
\begin{align*}
    &\Biggl[(p_l^1 \cdots p_l^{2m_1})_\SO (p_l^{2m_1+1}-1)\left(\prod_{j=0}^{\frac{l-2}{2}} p_{2j+1}^1\cdots p_{2j+1}^{2m_{l-2j}-2m_{l-2j-1}}\right)\\
    &\left(\prod_{j=1}^{\frac{l-2}{2}}(p_{2j}^1+1)(p_{2j}^2\cdots p_{2j}^{2m_{l+1-2j}-2m_{l-2j}-1})_\SO(p_{2j}^{2m_{l+1-2j}-2m_{l-2j}}-1)\right)\\
    &(p_0^1+1) (p_0^2 \cdots p_0^{m_0})_\SO \Biggr],
\end{align*}
and the partition $[\underline{p}_1\underline{p}_1(2n^*-1)1]^{\SO_{2n}} = (([b_1^{a_1} \cdots b_r^{a_r}]^t)_{+-})^{\SO_{2n}}$ is equal to the following partition
\begin{align*}
    &\Biggl[((p_l^1-1)p_l^2 \cdots p_l^{2m_1+1}p_0^1 \cdots p_0^{m_0}(1))^\SO\\
    &\left(\prod_{j=0}^{\frac{l-2}{2}} p_{2j+1}^1\cdots p_{2j+1}^{2m_{l-2j}-2m_{l-2j-1}}\right)\\
    &\left(\prod_{j=1}^{\frac{l-2}{2}}(p_{2j}^1+1)(p_{2j}^2\cdots p_{2j}^{2m_{l+1-2j}-2m_{l-2j}-1})_\SO(p_{2j}^{2m_{l+1-2j}-2m_{l-2j}}-1)\right)\Biggr].
\end{align*}
Since $m_0$ is even and $p_l^1$ occurs with odd multiplicity, 
it is easy to see that as partitions,
\begin{align*}
    &\,((p_l^1-1)p_l^2 \cdots p_l^{2m_1+1}p_0^1 \cdots p_0^{m_0}(1))^\SO \\
    = &\,(p_l^1 \cdots p_l^{2m_1})_\SO (p_l^{2m_1+1}-1)(p_0^1+1) (p_0^2 \cdots p_0^{m_0})_\SO.
\end{align*}
Therefore, we have shown that 
\eqref{criterion3} holds. 

This completes the proof of Theorem \ref{mainthm1}.
\end{proof}

\begin{rmk}\label{mixedpartities}
When $b_i$ are of mixed parities, conditions \eqref{oriequ}, \eqref{oriequ2}, and \eqref{criterion3} may not always hold.  We give some examples as follows. 

When $\mathrm{G}_n=\Sp_{2n}$, if $\ul{p}(\psi)=[3^32^2]$, then 
\MATHBLOCKDOLLARDOLLAR[[\underline{p}_1\underline{p}_1(2n^*)]]{^{\mathrm{G}_n}=\eta_{\hat{\frak{g}}_n, \frak{g}_{n}}([3^32^2])=[5^22];}
if $\ul{p}(\psi)=[2^21^3]$, then 
\MATHBLOCKDOLLARDOLLAR[[\underline{p}_1\underline{p}_1(2n^*)]]{^{\mathrm{G}_n}=[2^3]<\eta_{\hat{\frak{g}}_n, \frak{g}_{n}}([2^21^3])=[42].}

When $\mathrm{G}_n=\SO_{2n+1}$, if
$\ul{p}(\psi)=[3^22^3]$, then
\MATHBLOCKDOLLARDOLLAR[[\underline{p}_1\underline{p}_1(2n^*+1)]]{^{\mathrm{G}_n}=\eta_{\hat{\frak{g}}_n, \frak{g}_{n}}([3^22^3])=[5^23];}
if $\ul{p}(\psi)=[2^31^2]$, then
\MATHBLOCKDOLLARDOLLAR[[\underline{p}_1\underline{p}_1(2n^*+1)]]{^{\mathrm{G}_n}=[3^3]<\eta_{\hat{\frak{g}}_n, \frak{g}_{n}}([2^31^2])=[531].}

When $\mathrm{G}_n=\SO_{2n}^{\alpha}$, if $\ul{p}(\psi)=[3^32^21]$, then 
\MATHBLOCKDOLLARDOLLAR[[\underline{p}_1\underline{p}_1(2n^*-1)1]]{^{\mathrm{G}_n}=\eta_{\hat{\frak{g}}_n, \frak{g}_{n}}([3^32^2])=[5^231];}
if $\ul{p}(\psi)=[32^21^3]$, then 
\MATHBLOCKDOLLARDOLLAR[[\underline{p}_1\underline{p}_1(2n^*-1)1]]{^{\mathrm{G}_n}=[3^31]<\eta_{\hat{\frak{g}}_n, \frak{g}_{n}}([2^21^3])=[631].}
\end{rmk}

In general, we have the following inequality of partitions which is consistent with Conjecture \ref{cubmfclocal}.

\begin{prop}\label{propinequality}
Let $\mathrm{G}_n=\Sp_{2n}, \SO_{2n+1}$, and let $\ul{p}=[\underline{p}_1\underline{p}_1(2n^*)]^{\mathrm{G}_n}$, or $[\underline{p}_1\underline{p}_1(2n^*+1)]^{\mathrm{G}_n}$, then
$$\ul{p}\leq \eta_{\hat{\frak{g}}_n, \frak{g}_{n}}([b_1^{a_1} \cdots b_r^{a_r}]).$$ 
\end{prop}

\begin{proof}
Recall that 
$n^{*}=\big\lfloor\frac{\sum_{b_i \text{ odd }} a_i}{2}\big\rfloor,$
and 
$$\underline{p}_1=\left[\big\lfloor \frac{b_1}{2} \big\rfloor^{a_1}
\big\lfloor \frac{b_2}{2} \big\rfloor^{a_2} \cdots \big\lfloor \frac{b_r}{2} \big\rfloor^{a_r}\right]^t = a_1^{\big\lfloor \frac{b_1}{2} \big\rfloor} + a_2^{\big\lfloor \frac{b_2}{2} \big\rfloor} + \cdots + a_r^{\big\lfloor \frac{b_r}{2} \big\rfloor},$$
with $b_1 \geq b_2 \geq \cdots \geq b_r$. 

When $\mathrm{G}_n=\Sp_{2n}$, we have that $2n^{*}=(\sum_{b_i \text{ odd }} a_i)-1$, 
\MATHBLOCKDOLLARDOLLAR[[\underline{p}_1\underline{p}_1(2n^*)]]{^{\Sp_{2n}}=
(([\underline{p}_1\underline{p}_1(2n^*)]^t)_{\Sp_{2n}})^t,}
and 
$$\eta_{\frak{so}_{2n+1}, \frak{sp}_{2n}}([b_1^{a_1} \cdots b_r^{a_r}])=(([b_1^{a_1} \cdots b_r^{a_r}]^-)_{\Sp_{2n}})^t.$$
Hence, it suffices to show that 
$$([\underline{p}_1\underline{p}_1(2n^*)]^t)_{\Sp_{2n}} \geq ([b_1^{a_1} \cdots b_r^{a_r}]^-)_{\Sp_{2n}},$$
and 
suffices to show that 
\MATHBLOCKDOLLARDOLLAR[[\underline{p}_1\underline{p}_1(2n^*)]]{^t \geq [b_1^{a_1} \cdots b_r^{a_r}]^-.}
This is clear since 
\MATHBLOCKDOLLARDOLLAR[[\underline{p}_1\underline{p}_1(2n^*)]]{^t= 2\left[\big\lfloor \frac{b_1}{2} \big\rfloor^{a_1}
\big\lfloor \frac{b_2}{2} \big\rfloor^{a_2} \cdots \big\lfloor \frac{b_r}{2} \big\rfloor^{a_r}\right]+\left[1^{(\sum_{b_i \text{ odd }} a_i)-1}\right],}
and 
\MATHBLOCKDOLLARDOLLAR[[b_1^{a_1} \cdots b_r^{a_r}]]{^-=[b_1^{a_1} \cdots b_r^{a_r-1}(b_r-1)].}

When $\mathrm{G}_n=\SO_{2n+1}$, we have that $2n^{*}=\sum_{b_i \text{ odd }} a_i$,
\MATHBLOCKDOLLARDOLLAR[[\underline{p}_1\underline{p}_1(2n^*+1)]]{^{\SO_{2n+1}}=
(([\underline{p}_1\underline{p}_1(2n^*+1)]^t)_{\SO_{2n+1}})^t,}
and 
$$\eta_{\frak{sp}_{2n},\frak{so}_{2n+1}}([b_1^{a_1} \cdots b_r^{a_r}])=(([b_1^{a_1} \cdots b_r^{a_r}]^+)_{\SO_{2n+1}})^t.$$
Hence, it suffices to show that 
\begin{equation}\label{sec5equ100}
([\underline{p}_1\underline{p}_1(2n^*+1)]^t)_{\SO_{2n+1}} \geq ([b_1^{a_1} \cdots b_r^{a_r}]^+)_{\SO_{2n+1}}.
\end{equation}
Note that 
\MATHBLOCKDOLLARDOLLAR[[\underline{p}_1\underline{p}_1(2n^*+1)]]{^t= 2\left[\big\lfloor \frac{b_1}{2} \big\rfloor^{a_1}
\big\lfloor \frac{b_2}{2} \big\rfloor^{a_2} \cdots \big\lfloor \frac{b_r}{2} \big\rfloor^{a_r}\right]+\left[1^{(\sum_{b_i \text{ odd }} a_i)+1}\right],}
and 
\MATHBLOCKDOLLARDOLLAR[[b_1^{a_1} \cdots b_r^{a_r}]]{^+=[(b_1+1)b_1^{a_1-1} \cdots b_r^{a_r}].}
When $b_1$ is even, it is easy to see that 
\MATHBLOCKDOLLARDOLLAR[[\underline{p}_1\underline{p}_1(2n^*+1)]]{^t \geq [b_1^{a_1} \cdots b_r^{a_r}]^+.}
Hence, \eqref{sec5equ100} holds. 
When $b_1$ is odd, by the recipe \cite[Lemma 6.3.8]{CM93} of taking the $\SO_{2n+1}$-collapse, we have that 
\begin{align*}
    &\,(([b_1^{a_1} \cdots b_r^{a_r}]^+)_{\SO_{2n+1}}\\
    =&\,([(b_1+1)b_1^{a_1-1} \cdots b_r^{a_r}])_{\SO_{2n+1}}\\
    =&\,[b_1^{a_1} \cdots b_{j-1}^{a_{j-1}} (b_j+1)b_j^{a_j-1} \cdots b_r^{a_r}]_{\SO_{2n+1}},
\end{align*}
where $b_j$ is the largest even part. Then it is easy to see that 
\MATHBLOCKDOLLARDOLLAR[[\underline{p}_1\underline{p}_1(2n^*+1)]]{^t \geq [b_1^{a_1} \cdots b_{j-1}^{a_{j-1}} (b_j+1)b_j^{a_j-1} \cdots b_r^{a_r}].}
Hence, \eqref{sec5equ100} also holds. 
\end{proof}

\begin{rmk}
When $\mathrm{G}_n=\SO_{2n}^{\alpha}$, we also expect that 
\MATHBLOCKDOLLARDOLLAR[[\underline{p}_1\underline{p}_1(2n^*-1)1]]{^{\SO_{2n}}\leq
\eta_{\frak{so}_{2n}, \frak{so}_{2n}}([b_1^{a_1} \cdots b_r^{a_r}])=([b_1^{a_1} \cdots b_r^{a_r}]^t)_{\SO_{2n}}.}
Due to the complexity of this case, we don't discuss it here. It is an interesting question to completely describe the conditions on $a_i, b_i$ such that the equalities \eqref{criterion}, \eqref{criterion2}, and  \eqref{criterion3} hold. 
\end{rmk}


\section{Proof of Theorems \ref{mainintro}}\label{vanandmain}

In this section, we prove Theorems \ref{mainintro}. We follow closely the notation in Section 4. 
First we recall the identity
\eqref{sec6equ2all}. 
\begin{align}\label{sec6equ2_2sp2n}
\begin{split}
    &\,  \sum_{\CO \in \CN_{\frak{g}_{n}}}
c_{\CO}(\wt{\Pi}_{\psi})\hat{\mu}_\CO(f) \\
= &\,  \left(\sum {_{\CO \in \CN_{\frak{g_1'}^{\theta}}}
} c_{\CO}(\wt{\pi}_{\psi^1})\hat{\mu}_\CO({{}\wt{f}^1_{\theta_{\widehat{\mathrm{G}}_1'}}})\right) \left(\sum {_{\CO \in \CN_{\frak{g_2'}^{\theta}}}
} c_{\CO}(\wt{\pi}_{\psi^2})\hat{\mu}_\CO({{}\wt{f}^2_{\theta_{\widehat{\mathrm{G}}_2'}}})\right),
\end{split}
\end{align}
where $f \in \wt{\CH}(G_n)$, transferring to $f'=f'^1 \otimes f'^2 \in \wt{\CS}(G')$, $f'^i \in \wt{\CS}(G'_i)$, $i=1,2$. Let 
$\wt{f}^1 \in 
\wt{\CH}(N_1)$,
and $\wt{f}^2 \in \wt{\CH}(N_2)$, transferring to $f'^1$ and $f'^2$, via surjective maps in \cite[Corollary 2.1.2]{Art13}
$$\iota_{G'_i}: \wt{\CH}(N_i) \rightarrow \wt{\CS}(G'_i), i=1,2,$$
respectively. Here, $\psi^1=\sum_{i\in I}\psi_i$, $\psi^2 = \sum_{j \in J}\psi_j$, with $I=\{i|b_i \text{ odd}\}$, $J=\{i|b_i \text{ even}\}$.
$\theta_{\widehat{\mathrm{G}}_i'}$ is defined in \eqref{thetaGn}, and $\wt{f}_{\theta_{\widehat{\mathrm{G}}_i'}}^i$ is the Harish-Chandra descent of $\widetilde{f}^i$ (see for example \cite[Section 3.1]{Kon02}), $i=1,2$.
$c_{\CO}(\wt{\Pi}_{\psi})=\sum_{\pi \in \wt\Pi_{\psi}} 
c_{\CO}(\pi)$.

Next, we recall several lemmas from \cite{Sha90} and \cite{Kon02} as follows. Let $\Delta$ be the transfer factor defined by Langlands-Shelstad (\cite{LS87}), or by Kottwitz-Shelstad (\cite{KS99}). 

\begin{lem}{\cite[Lemma 9.8]{Sha90}}\label{liorbits}
Let $\CO_0$ be a nilpotent orbit in $\CN_{\frak{g}_n}$. Then there exists a function $f \in \mathcal{H}(G_n)$ with arbitrarily small support such that 
$$\hat{\mu}_{\CO_0}(f)=1, \text{ and } \hat{\mu}_{\CO}(f)=0,$$
for any other $\CO \in \CN_{\frak{g}_n}$.
\end{lem}

We may assume that $f \in \wt{\mathcal{H}}(G_n)$, $f'=f'^1\otimes f'^2 \in \wt{\CS}(G')$, $\wt{f}^1$, $\wt{f}^2$, $\wt{f}^1_{\theta_{\widehat{\mathrm{G}}_1'}}$, and $\wt{f}^2_{\theta_{\widehat{\mathrm{G}}_2'}}$ all have sufficiently small support since they are all related by either transfer or descent.
For $f \in \wt{\mathcal{H}}(G_n)$ and $t \in F^{\times}$ with $|t|$ small, we define $f_t$ by $f(exp(t^{-1} (\cdot)))$.
Similarly, we can define $f'^i_t$,
$\wt{f}^i_t$ and $\wt{f}^i_{\theta_{\widehat{\mathrm{G}}_i'}, t}$, $i=1,2$. 

\begin{lem}{\cite[Lemma 9.7]{Sha90}}\label{transfer}
Suppose $f \in \wt{\mathcal{H}}(G_n)$ has sufficiently small support near $1$. Then for $t \in F^{\times}$ with $|t|$ small, $f_{t^2}$ and $$|t|^{\dim(\frak{g}_n)-\dim(\frak{g'})}(f^{G'})_{t^2}$$ have $\Delta$-matching orbital integrals.
\end{lem}

\begin{lem}{\cite[Lemma 8.5]{Kon02}}\label{transfer2}
Suppose $\wt{f} \in \wt{\mathcal{H}}(N)$ has sufficiently small support near $1$ and transfers to $f^{G_n}$ for some $f \in \wt{\mathcal{H}}(G_n)$. Then for $t \in F^{\times}$ with $|t|$ small, ${f}^{G_n}_{t^2}$ and $|t|^{\dim(\frak{g}_n)-\dim(\frak{g}_n^{\theta})}({{}\wt{f}_{\theta_{\widehat{\mathrm{G}}_n}}})_{t^2}$ have $\Delta$-twisted matching orbital integrals.
\end{lem}

The above statement in Lemma \ref{transfer2} is directly implied by the same proof of \cite[Lemma 8.5]{Kon02}, while \cite[Lemma 8.5]{Kon02} considered the setting of $\dim(\frak{g}_n)=\dim(\frak{g}_n^{\theta})$.






Now, we are ready to prove Theorem \ref{mainintro}. First, we prove \ref{mainintro}(1).

\begin{thm}\label{mainintropart1}
Let $\psi$ be a local Arthur parameter as in \eqref{lap}. 
Then, Conjecture \ref{wfsbitorsor} is true if and only if for any 
$\pi\in\wt{\Pi}_{\psi}$,
$\ul{p}\in\frak{p}^m(\pi)$, $$\dim_{{\frak{g}}_n} (\ul{p}) \leq \dim_{{\frak{g}}_n} (\eta_{{\hat{\frak{g}}_n,\frak{g}_n}}(\ul{p}(\psi))).$$
\end{thm}

\begin{proof}
First, we assume that 
Conjecture \ref{wfsbitorsor} is true.
Assume that there there exist
$\pi \in \wt{\Pi}_{\psi}$ and   
$\ul{p} \in \frak{p}^m(\pi)$, such that 
$\dim_{{\frak{g}}_n} (\ul{p}) > \dim_{{\frak{g}}_n} (\eta_{{\hat{\frak{g}}_n,\frak{g}_n}}(\ul{p}(\psi)))$. 
Without loss of generality, we assume further that
$$\ul{p}\in max\{\cup_{\pi\in\wt{\Pi}_{\psi}} \frak{p}^m(\pi)\}.$$ 

In the identity \eqref{sec6equ2_2sp2n}, we have that for any  nilpotent orbit $\CO$, 
$$c_{\CO}(\wt{\Pi}_{\psi})=\sum_{\pi \in \wt\Pi_{\psi}} c_{\CO}(\pi).$$
On the other hand, since $\ul{p}\in max\{\cup_{\pi\in\wt{\Pi}_{\psi}} \frak{p}^m(\pi)\}$, by \cite{MW87},
for any nilpotent orbit $\CO_{\ul{p}}$ corresponding to $\ul{p}$, 
$c_{\CO_{\ul{p}}}(\pi) \in \mathbb{Z}_{\geq 0}$ and there exist $\pi \in \wt\Pi_{\psi}$ and a rational nilpotent orbit $\CO_0 \in \CN_{\frak{g}_{n}}$ corresponding to $\ul{p}$ such that $c_{\CO_0}(\pi) \in \mathbb{Z}_{> 0}$. Hence, 
$c_{\CO_0}(\wt{\Pi}_{\psi}) >0$.
By Lemma \ref{liorbits}, there exists a function $f_0 \in \mathcal{H}(G_n)$ with arbitrarily small support such that 
$$\hat{\mu}_{\CO_0}(f_0)=1, \text{ and } \hat{\mu}_{\CO}(f_0)=0,$$
for any other $\CO \in \CN_{\frak{g}_{n}}$.
Let $f=f_0$ when $\mathrm{G}_n=\Sp_{2n}, \SO_{2n+1}$, and let $f=R(c_{\RG_n})(f_0+f_0^c)$ when 
$\mathrm{G}_n=\SO_{2n}^{\alpha}$, where $c$ is the non-trivial element in the outer automorphism group. Then $f \in \wt{\CH}(G_n)$. When 
$\mathrm{G}_n=\SO_{2n}^{\alpha}$, if $\CO_0\neq \CO_0^c$, i.e., $\CO_0$ is a very even nilpotent orbit, then we have that 
$$\hat{\mu}_{\CO_0}(f)=\hat{\mu}_{\CO_0^c}(f)=1, \text{ and } \hat{\mu}_{\CO}(f)=0,$$
for any other $\CO \in \CN_{\frak{g}_{n}}$. Note that $\dim(\CO_0)=\dim(\CO_0^c)$, and $c_{\CO_0^c}(\wt{\Pi}_{\psi}) \in \mathbb{Z}_{\geq 0}$. 

Now, without loss of generality, we can assume that $f$, $R(c_{\RG_1'})f'^1\otimes R(c_{\RG_2'})f'^2$,  {$\wt{f}^1$, $\wt{f}^2$} , $\wt{f}^1_{\theta_{\widehat{\mathrm{G}}_1'}}$, and $\wt{f}^2_{\theta_{\widehat{\mathrm{G}}_2'}}$ all have small supports near 1. Take $t \in F^{\times}$ with $|t|$ small. By Lemma \ref{transfer} and Lemma \ref{transfer2}, 
we have 
\begin{align}\label{sec6equ2_3sp2n}
\begin{split}
    &\,  \sum_{\CO \in \CN_{\frak{g}_{n}}}
c_{\CO}(\wt{\Pi}_{\psi})\hat{\mu}_\CO((f)_{t^2}) \\
= &\, |t|^{\dim(\frak{g}_n)-\dim(\frak{g_1'}^{\theta})-\dim(\frak{g'_2}^{\theta})} \left(\sum {_{\CO \in \CN_{\frak{g_1'}^{\theta}}}
} c_{\CO}(\wt{\pi}_{\psi^1})\hat{\mu}_\CO((\wt{f}^1_{\theta_{\widehat{\mathrm{G}}_1'}})_{t^2})\right) \\
&\, \cdot \left(\sum {_{\CO \in \CN_{\frak{g_2'}^{\theta}}}
} c_{\CO}(\wt{\pi}_{\psi^2})\hat{\mu}_\CO((\wt{f}^2_{\theta_{\widehat{\mathrm{G}}_2'}})_{t^2})\right).
\end{split}
\end{align}
On the other hand,
\begin{align*} 
\hat{\mu}_\CO((f)_{t^2})&=|t|^{2\dim(\frak{g_n})-\dim(\CO)}\hat{\mu}_\CO(f),\\
\hat{\mu}_\CO((\wt{f}^1_{\theta_{\widehat{\mathrm{G}}_1'}})_{t^2})&=|t| {^{2\dim(\frak{g_1'}^{\theta})-\dim(\CO)}} \hat{\mu}_\CO(\wt{f}^1_{\theta_{\widehat{\mathrm{G}}_1'}}),\\
\hat{\mu}_\CO((\wt{f}^2_{\theta_{\widehat{\mathrm{G}}_2'}})_{t^2})&=|t| {^{2\dim(\frak{g_2'}^{\theta})-\dim(\CO)}} \hat{\mu}_\CO(\wt{f}^2_{\theta_{\widehat{\mathrm{G}}_2'}}).
\end{align*}
        
Therefore, we obtain that 
\begin{align}\label{sec6equ2_4sp2n}
\begin{split}
    &\,  
    |t|^{\dim(\frak{g}_n)-\dim(\CO_0)}
c_{\CO_0}(\wt{\Pi}_{\psi})\hat{\mu}_{\CO_0}(f) \\
= &\, \left(\sum {_{\CO \in \CN_{\frak{g_1'}^{\theta}}}
} |t| {^{\dim(\frak{g_1'}^{\theta})-\dim(\CO)}
} c_{\CO}(\wt{\pi}_{\psi^1})\hat{\mu}_\CO(\wt{f}^1_{\theta_{\widehat{\mathrm{G}}_1'}})\right)\\
&\,\cdot \left(\sum {_{\CO \in \CN_{\frak{g_2'}^{\theta}}}
} |t| {^{\dim(\frak{g_2'}^{\theta})-\dim(\CO)}
} c_{\CO}(\wt{\pi}_{\psi^2})\hat{\mu}_\CO(\wt{f}^2_{\theta_{\widehat{\mathrm{G}}_2'}})\right).
\end{split}
\end{align}

Then one can see that, there exist nilpotent orbits  {$\CO_1 \in \CN_{\frak{g_1'}^{\theta}}$, $\CO_2 \in \CN_{\frak{g_2'}^{\theta}}$ } such that 
$c_{\CO_1}(\wt{\pi}_{\psi^1}) \hat{\mu}_{\CO_1}({{}\wt{f}^1_{\theta_{\widehat{\mathrm{G}}_1'}}})$ and $c_{\CO_2}(\wt{\pi}_{\psi^2}) \hat{\mu}_{\CO_2}({{}\wt{f}^2_{\theta_{\widehat{\mathrm{G}}_2'}}})$ are all nonzero and
\begin{align*}
\begin{split}
 \dim(\frak{g}_{n}) - \dim_{\frak{g}_{n}}(\CO_0) 
    = \dim( {\frak{g_1'}^{\theta}} ) + \dim( {\frak{g_2'}^{\theta}} ) -\dim {_{\frak{g_1'}^{\theta}}} (\CO_1)-\dim {_{\frak{g_2'}^{\theta}}} (\CO_2).
    \end{split}
\end{align*}

By Proposition \ref{prop:keylemma-v3-all}, 
we can see that 
\begin{align*}
\dim {_{\frak{g_1'}^{\theta}}} (\CO_1)+\dim {_{\frak{g_2'}^{\theta}}} (\CO_2)
> \dim {_{\frak{g_1'}^{\theta}}} ( (\ul{p}(\psi {^1} ){^t} ){_{{\RG_1'}^{\theta}}} )+\dim {_{\frak{g_2'}^{\theta}}} ((\ul{p}(\psi^2)^t) {_{{\RG_2'}^{\theta}}} ).
\end{align*}
Hence, we have either 
$$\dim {_{\frak{g_1'}^{\theta}}} (\CO_1)> \dim {_{\frak{g_1'}^{\theta}}} ( (\ul{p}(\psi {^1} ){^t} ){_{{\RG_1'}^{\theta}}} ),$$
or, 
$$\dim {_{\frak{g_2'}^{\theta}}} (\CO_2)> \dim {_{\frak{g_2'}^{\theta}}} ((\ul{p}(\psi^2)^t) {_{{\RG_2'}^{\theta}}} ),$$
contradicting Conjecture \ref{wfsbitorsor}.

Next, we assume that for any 
$\pi\in\wt{\Pi}_{\psi}$,
$\ul{p}\in\frak{p}^m(\pi)$, $$\dim_{{\frak{g}}_n} (\ul{p}) \leq \dim_{{\frak{g}}_n} (\eta_{{\hat{\frak{g}}_n,\frak{g}_n}}(\ul{p}(\psi))).$$ Assume that Conjecture \ref{wfsbitorsor} does not hold, i.e., there exists $\ul{p} \in \frak{p}^m(\wt{\pi}_{\psi})$, such that 
 $$\dim_{\frak{g}_n^{\theta}}(\ul{p}) >  \dim_{\frak{g}_n^{\theta}}((\ul{p}(\psi)^t)_{\RG_n^{\theta}}).$$

Applying \cite[Theorem 2.2.1, Part (b)]{Art13} with $x=s=1$, we have the following distribution identity
\begin{equation}\label{distribution identity2}
    \sum_{\pi \in \wt{\Pi}_{\psi}} \langle s_{\psi}, \pi \rangle f_{G_n}(\pi) = f(\psi)=\tr (\wt{\pi}_{\psi}(\wt{f})),
\end{equation}
where $f \in \wt{\CH}(G_n)$, $\wt{f} \in
\wt{\CH}(N)$ such that $\wt{f}$ transfers to $f^{G_n}$. We take the character expansions of both sides of \eqref{distribution identity2}.  The left hand side  {is expanded at the identity, and } the right hand side  {is expanded at the twisted points } $\theta_{\widehat{\mathrm{G}}_n}$ (see \eqref{thetaGn}). We have the following equality: 
\begin{align}\label{sec6otherdirection1}
\begin{split}
 \sum_{\pi \in \wt\Pi_{\psi}} \langle s_{\psi}, \pi \rangle \sum_{\CO \in \CN_{\frak{g}_{n}}} 
c_{\CO}(\pi)\hat{\mu}_\CO(f) 
=  \sum {_{\CO \in \CN_{\mathfrak{g}_{n}^{\theta}}}
} c_{\CO}(\wt{\pi}_{\psi})\hat{\mu}_\CO({{}\wt{f}_{\theta_{\widehat{\mathrm{G}}_n}}}),
\end{split}
\end{align}
where ${{}\wt{f}_{\theta_{\widehat{\mathrm{G}}_n}}}$ is the Harish-Chandra descent of $\widetilde{f}$ (see for example \cite[Section 3.1]{Kon02}). 
Let $c_{\CO}(\wt{\Pi}_{\psi})=\sum_{\pi \in \wt\Pi_{\psi}} \langle s_{\psi}, \pi \rangle
c_{\CO}(\pi)$, then \eqref{sec6otherdirection1} becomes
\begin{align}\label{sec6otherdirection2}
\begin{split}
\sum_{\CO \in \CN_{\frak{g}_{n}}} 
c_{\CO}(\wt{\Pi}_{\psi})\hat{\mu}_\CO(f) 
= \sum {_{\CO \in \CN_{\mathfrak{g}_{n}^{\theta}}}
} c_{\CO}(\wt{\pi}_{\psi})\hat{\mu}_\CO({{}\wt{f}_{\theta_{\widehat{\mathrm{G}}_n}}}).
\end{split}
\end{align}

Since $\ul{p} \in \frak{p}^m(\wt{\pi}_{\psi})$, there exists a rational nilpotent orbit $\CO_0 \in \CN_{\frak{g}_{n}^{\theta}}$ corresponding to $\ul{p}$ such that $c_{\CO_0}(\wt{\pi}) \neq 0$.
Since both the transfer map and the Harish-Chandra descent map are surjective at 1 (see \cite[Corollary 2.1.2]{Art13} and \cite[Proposition 3.9]{Var17}), there exist $f$ and $\wt{f}$ such that $\hat{\mu}_{\CO_0}({{}\wt{f}_{\theta_{\widehat{\mathrm{G}}_n}}}) =1$ and 
$\hat{\mu}_\CO({{}\wt{f}_{\theta_{\widehat{\mathrm{G}}_n}}}) =0$ for any other $\CO \in \CN_{\mathfrak{g}_{n}^{\theta}}$. Therefore, \eqref{sec6otherdirection2} becomes 
$$\sum_{\CO \in \CN_{\frak{g}_{n}}} 
c_{\CO}(\wt{\Pi}_{\psi})\hat{\mu}_\CO(f) = c_{\CO_0}(\wt{\pi}_{\psi})\hat{\mu}_{\CO_0}({{}\wt{f}_{\theta_{\widehat{\mathrm{G}}_n}}}).$$

Now, without loss of generality, we can assume that $f$, $\wt{f}$,  $\wt{f}_{\theta_{\widehat{\mathrm{G}}_n}}$ all have small supports near 1. Take $t \in F^{\times}$ with $|t|$ small. By Lemma \ref{transfer} and Lemma \ref{transfer2}, 
we have 
\begin{align*}
\begin{split}
    \sum_{\CO \in \CN_{\frak{g}_{n}}}
c_{\CO}(\wt{\Pi}_{\psi})\hat{\mu}_\CO((f)_{t^2}) 
= |t|^{\dim(\frak{g}_n)-\dim(\frak{g}_n^{\theta})} c_{\CO_0}(\wt{\pi}_{\psi})\hat{\mu}_{\CO_0}(({{}\wt{f}_{\theta_{\widehat{\mathrm{G}}_n}}})_{t^2}).
\end{split}
\end{align*}
On the other hand, 
\begin{align*}
\hat{\mu}_\CO((f)_{t^2})&=|t|^{2\dim(\frak{g_n})-\dim(\CO)}\hat{\mu}_\CO(f),\\
\hat{\mu}_{\CO_0}(({{}\wt{f}_{\theta_{\widehat{\mathrm{G}}_n}}})_{t^2}) & = |t|^{2\dim(\mathfrak{g}_{n}^{\theta})-\dim(\CO_0)}\hat{\mu}_{\CO_0}({{}\wt{f}_{\theta_{\widehat{\mathrm{G}}_n}}}).
\end{align*}

Therefore, there exists $\CO_1 \in \CN_{\frak{g}_{n}}$ such that $c_{\CO_1}(\wt{\Pi}_{\psi}) \neq 0$ and 
$$2\dim(\frak{g_n})-\dim(\CO_1)= \dim(\frak{g}_n)-\dim(\frak{g}_n^{\theta}) + 2\dim(\mathfrak{g}_{n}^{\theta})-\dim(\CO_0).$$
Hence, when $\RG_n=\SO_{2n+1}, \Sp_{2n}$, $\dim(\CO_1)= \dim(\CO_0) >\dim_{\frak{g}_n^{\theta}}((\ul{p}(\psi)^t)_{\RG_n^{\theta}})$ which is equal to $$\dim_{\frak{g}_{n}}(\eta_{\hat{\frak{g}}_n, \frak{g}_{n}}([b_1^{a_1} \cdots b_r^{a_r}]))$$ by Proposition \ref{prop identity}, contradicting our assumption. When $\RG_n=\SO_{2n}$, $\dim(\CO_1)= \dim(\frak{g_n}) - \dim(\mathfrak{g}_{n}^{\theta}) + \dim(\CO_0) >\dim(\frak{g_n}) - \dim(\mathfrak{g}_{n}^{\theta}) + \dim_{\frak{g}_n^{\theta}}((\ul{p}(\psi)^t)_{\RG_n^{\theta}})=\dim_{\frak{g}_{n}}(\eta_{\hat{\frak{g}}_n, \frak{g}_{n}}([b_1^{a_1} \cdots b_r^{a_r}]))$, by Lemma \ref{lem:onefactor-soeven-codim}, also contradicting our assumption.

This completes the proof of Theorem \ref{mainintropart1}.
\end{proof}

Next, we prove Theorem \ref{mainintro}(3-a).

\begin{thm}\label{mainintropart2}
Let $\psi$ be a local Arthur parameter as in \eqref{lap}. Assume that Conjecture \ref{wfsbitorsor} holds. Then Conjecture \ref{shaconj2} is true.
\end{thm}

\begin{proof}
Assume that $\psi = \bigoplus_{i=1}^r \phi_i \otimes S_{m_i} \otimes S_{n_i}$ is non-tempered, i.e., there exists $n_i >1$. As before, we assume that $\phi_i$ has dimension $k_i$, let $a_i=k_im_i$ and $b_i=n_i$. Then $\ul{p}_{\psi}=[b_1^{a_1} \cdots b_r^{a_r}]$ and by Theorem \ref{wfslinear}, $\frak{p}^m({\pi}_{\psi})=\{[b_1^{a_1} \cdots b_r^{a_r}]^t\}.$ Recall that $I=\{i|b_i \text{ odd}\}$, $J=\{i|b_i \text{ even}\}$, and 
$\psi^1=\sum_{i\in I}\psi_i$, $\psi^2 = \sum_{j \in J}\psi_j$.
Since $\psi$ is non-tempered, one can easily see that the two partitions
 {$(\ul{p}(\psi^1)^t)_{{\RG_1'}^{\theta}}$ and 
$(\ul{p}(\psi_2)^t)_{{\RG_2'}^{\theta}}$ } cannot be both parametrizing regular nilpotent orbits. 

Assume that there exists a generic representation $\pi_0 \in \widetilde{\Pi}_{\psi}$. 
Then there exists a regular nilpotent orbit $\CO_0 \in \CN_{\frak{g}_{n}}$ such that the coefficient $c_{\CO_0}(\pi_0)$ in the character expansion of $\pi_0$ is nonzero, indeed it is $1$ by the result in \cite{MW87} and the uniqueness of Whittaker models. Since for any $\pi \in \widetilde{\Pi}_{\psi}$, $c_{\CO_0}(\pi)$ is either one or zero and $s=1$,  
we can see that 
$\sum_{\pi \in \widetilde{\Pi}_{\psi}} c_{\CO_0}(\pi) \neq 0$. 

Applying similar arguments as in the proof of Theorem 
\ref{mainintropart1}, one can see that there exist nilpotent orbits  {$\CO_1 \in \CN_{\frak{g_1'}^{\theta}}$ and $\CO_2 \in \CN_{\frak{g_2'}^{\theta}}$ } such that both $c_{\CO_1}(\wt{\pi}_{\psi^1}) \hat{\mu}_{\CO_1}({{}\wt{f}^1_{\theta_{\widehat{\mathrm{G}}_1'}}})$ and  $c_{\CO_2}(\wt{\pi}_{\psi^2}) \hat{\mu}_{\CO_2}({{}\wt{f}^2_{\theta_{\widehat{\mathrm{G}}_2'}}})$ are nonzero, and 
\begin{align*}
\begin{split}
    &\, \dim(\frak{g}_{n}) - \dim_{\frak{g}_{n}}(\CO_0) \\
    = &\, \dim( {\frak{g_1'}^{\theta}} ) + \dim( {\frak{g_2'}^{\theta}} ) -\dim {_{\frak{g_1'}^{\theta}}} (\CO_1)-\dim {_{\frak{g_2'}^{\theta}}} (\CO_2).
    \end{split}
\end{align*}
Note that the regular orbit $\CO_0$ has maximal dimension of $\dim(\CO_0)= \dim(\frak{g}_{n}) -n$. Hence, $$\dim {_{\frak{g_1'}^{\theta}}} (\CO_1)+\dim {_{\frak{g_2'}^{\theta}}} (\CO_2) =\dim( {\frak{g_1'}^{\theta}} ) + \dim( {\frak{g_2'}^{\theta}} )-n_1-n_2,$$
since $n=n_1+n_2$. Therefore, both $\CO_1$ and $\CO_2$ are 
regular orbits, which is a contradiction by Conjecture \ref{wfsbitorsor}, and the discussion above. 

This completes the proof of Theorem \ref{mainintropart2}.
\end{proof}

At last, we prove Theorem \ref{mainintro}(3-b).

\begin{thm}\label{mainintropart3}
Let $\psi$ be a local Arthur parameter as in \eqref{lap}, with $\ul{p}(\psi) = [b_1^{a_1} b_2^{a_2} \cdots b_r^{a_r}]$ and $b_1 \geq b_2 \geq \cdots \geq b_r$. 
Let $$\underline{p}_1=\left[\big\lfloor \frac{b_1}{2} \big\rfloor^{a_1}
\big\lfloor \frac{b_2}{2} \big\rfloor^{a_2} \cdots \big\lfloor \frac{b_r}{2} \big\rfloor^{a_r}\right]^t,$$
and $n^{*}=\big\lfloor\frac{\sum_{b_i \text{ odd }} a_i}{2}\big\rfloor$. 
Assume that Conjecture \ref{wfsbitorsor} is true. 
    Then Conjecture \ref{cubmfclocal}(3) holds for the following cases.
    \begin{enumerate}
        \item When $\mathrm{G}_n=\Sp_{2n}$, and
        \begin{equation}\label{criterion_intro1_2}
    ([\underline{p}_1\underline{p}_1(2n^*)]^t)_{\Sp_{2n}}=([b_1^{a_1} \cdots b_r^{a_r}]^-)_{\Sp_{2n}}.
\end{equation}
In particular, if 
\begin{enumerate}
    \item[(i)] $a_r=b_r=1$ and $b_i$ are all even for $1 \leq i \leq r-1$,
    \item[(ii)] or, $b_i$ are all odd,
\end{enumerate}
then \eqref{criterion_intro1_2} holds and thus Conjecture \ref{cubmfclocal}(3) is valid. 
        \item When $\mathrm{G}_n=\SO_{2n+1}$, and
        \begin{equation}\label{criterion_intro2_2}
    ([\underline{p}_1\underline{p}_1(2n^*+1)]^t)_{\SO_{2n+1}}=([b_1^{a_1} \cdots b_r^{a_r}]^+)_{\SO_{2n+1}}.
\end{equation}
        In particular, if 
        \begin{enumerate}
            \item[(i)]  $b_1$ is even and $a_1=1$, and $b_i$ are all odd for $2 \leq i \leq r$,
            \item[(ii)] or, $b_i$ are all even,
        \end{enumerate}
       then \eqref{criterion_intro2_2} holds and thus Conjecture \ref{cubmfclocal}(3) is valid. 
        \item When    $\mathrm{G}_n=\SO_{2n}^{\alpha}$, and
         \begin{equation}\label{criterion_intro3_2}
    [\underline{p}_1\underline{p}_1(2n^*-1)1]^{\SO_{2n}}=([b_1^{a_1} \cdots b_r^{a_r}]^t)_{\SO_{2n}}.
 \end{equation}
        In particular, if all $b_i$ are of the same parity, then \eqref{criterion_intro3_2} holds and thus Conjecture \ref{cubmfclocal}(3) is valid.
    \end{enumerate}
\end{thm}

\begin{proof}
By Lemma \ref{raise} and Theorem \ref{mainthm1}, when \eqref{criterion_intro1_2} -- \eqref{criterion_intro3_2} hold, we have constructed an irreducible admissible representation $\sigma \in \wt{\Pi}_{\psi}$ such that $$\eta_{{\hat{\frak{g}}_n,\frak{g}_n}}(\ul{p}(\psi))\in \frak{p}(\sigma).$$
To prove Jiang's Conjecture \ref{cubmfclocal}(3), it remains to show that for any $\ul{p} > \eta_{{\hat{\frak{g}}_n,\frak{g}_n}}(\ul{p}(\psi))$, $\ul{p}\notin \frak{p}^m(\sigma)$; this is implied by Theorem \ref{mainintropart1} under the validity of Conjecture \ref{wfsbitorsor}.
\end{proof}


In following special cases of Theorem \ref{mainintropart3}, we only need to assume Conjecture \ref{shaconj2}. 

\begin{thm}
Let $\psi$ be a local Arthur parameter of $G_n$ as in \eqref{lap}. Assume that Conjecture \ref{shaconj2} is true. 
    Then Conjecture \ref{cubmfclocal} holds for the following cases.
    \begin{enumerate}
        \item When $\mathrm{G}_n=\Sp_{2n}$,
       $a_1=1$, $b_1=3$, and $b_i=1$ for $2 \leq i \leq r$. 
        \item When $\mathrm{G}_n=\SO_{2n+1}$,
        $a_1=1$, $b_1=2$, and $b_i=1$ for $2 \leq i \leq r$. 
        \item When    $\mathrm{G}_n=\SO_{2n}^{\alpha}$, 
       $a_1=1$, $b_1=3$, and $b_i=1$ for $2 \leq i \leq r$.  
    \end{enumerate}
\end{thm}

\begin{proof}
First, note that with the above assumptions in Cases (1-3), 
$\eta_{{\hat{\frak{g}}_n,\frak{g}_n}}(\ul{p}(\psi))$ will be exactly the partition parametrizing the subregular nilpotent orbits of $\frak{g}_n$. 
On the other hand, by Lemma \ref{raise} and Theorem \ref{mainthm1}, we have constructed an irreducible admissible representation $\sigma \in \wt{\Pi}_{\psi}$ such that $\eta_{{\hat{\frak{g}}_n,\frak{g}_n}}(\ul{p}(\psi))\in \frak{p}(\sigma)$. Hence, it suffices to show that there is no generic element in the local 
Arthur packet corresponding to $\psi$. Since $\psi$ is a non-tempered local Arthur parameter, this is clear from Conjecture \ref{shaconj2}.
\end{proof}

\appendix

\section{On central characters of representations in local Arthur packets, by Alexander Hazeltine, Baiying Liu, Chi-Heng Lo, and Freydoon Shahidi}\label{appendix}

In this appendix, we show that all representations in local Arthur packets have the same central characters, which has its own interests. 

Let $F$ be a non-Archimedean local field of characteristic zero and $V$ a finite dimensional vector space. Let
\[ \langle -,- \rangle: V\times V \to F\]
be a non-degenerate $\epsilon$-symmetric form ($\epsilon \in \{\pm 1\} \subset F^{\times}$):
\begin{align*}
    \langle \alpha v+\beta w, u \rangle&= \alpha \langle v,u\rangle +  \beta \langle w, y \rangle,\\
    \langle v,u \rangle &=\epsilon \langle u,v \rangle.
\end{align*}
We consider $\RG(V) \subset \GL(V)$ be the algebraic subgroup of elements $T$ in $\GL(V)$, which preserve the form $\langle -, - \rangle $:
\[ \langle Tv, Tu\rangle= \langle v,u \rangle. \]
Then $\RG(V)= \RO(V)$ when $\epsilon=1$ and $\RG(V)= \Sp(V)$ when $\epsilon=-1$. We take $\SO(V)$ be the connected component of $\RO(V)$ consisting of elements $T$ with determinant $+1$.

Let $\RG$ be a connected classical group over $F$, and denote $K$ be the splitting field of $\RG(V)$. We consider the $L$-group of $\RG(V)$ given by
\[ {}^{L}\RG:= \widehat{\RG}(\BC) \rtimes \Gal(K/F).\]
The cases we consider are given explicitly in the following table.
\begin{center}
    \begin{tabular}{ |c|c|c|c|}
\hline
     $\RG(V)$& $\widehat{\RG}(\BC)$ & $K$ & ${}^L \RG$  \\
     \hline
     $\Sp(V),$& $\SO_{2n+1}(\BC)$& $F$& $\SO_{2n+1}(\BC)$\\
     $\dim(V)=2n$&&&\\
     \hline
     $\SO(V),$& $\Sp_{2n}(\BC)$& $F$& $\Sp_{2n}(\BC)$\\
     $\dim(V)=2n+1$&&&\\
     \hline
     $\SO(V),$& $\SO_{2n}(\BC)$& $F(\sqrt{\disc(V)})$& $\RO_{2n}(\BC)$ ($\disc(V) \not\in F^{\times2}$) \\
     $\dim(V)=2n$&&&$\SO_{2n}(\BC)$ ($\disc(V) \in F^{\times2}$)\\
     \hline
\end{tabular}
\end{center}
An $L$-parameter $\varphi$ of $\RG(V)$ is an admissible homomorphism from $W_F \times \SL_2(\BC)$ to $ {}^{L} \RG$. The local Langlands correspondence associates to each $\varphi$ a set of isomorphism class of irreducible representation $\Pi_{\varphi}$. 

It is expected that all representations in $\Pi_{\varphi}$ share the same central character $\omega_{\varphi}$. Indeed, in the cases we consider this is known by \cite{GR10, Xu16}. Indeed, the central character is defined by \cite[\S8]{GR10}. Xu verified the representations in a tempered $L$-packet share this central character (\cite[Proposition 6.27]{Xu16}). In particular, this is true for simple $L$-parameters. By \cite[Proposition 2.4]{Xu16} and the Langlands classification, for an $L$-parameter $\varphi$ of $G$, it follows that all representations in $\Pi_\varphi$ share the same central character $\omega_\varphi.$ 

When $\dim(V)=2n$ ($n \geq 2$), the center of $G(V)$ is $Z(k)=\{\pm 1\}$. The following proposition gives $\omega_{\varphi}(-1)$ in these cases. This follows from \cite[\S8]{GR10} (see also \cite[\S 10]{GGP12}).

\begin{prop}\cite[\S 10]{GGP12}\label{prop GGP12}
Fix a nontrivial additive character $\psi$ of $F$ and assume $\dim(V)=2n$.  
\begin{enumerate}
    \item When $\RG(V)=\Sp(V)$, we have $\omega_\varphi (-1)= \epsilon(\varphi, \psi)$.
    \item When $\RG(V)=\SO(V)$, we have $\omega_\varphi (-1)= \epsilon(\varphi, \psi)/ \epsilon(\det \varphi, \psi)$.
\end{enumerate}
\end{prop}
We remark that in both cases, $\omega_{\varphi}(-1)$ is independent of the choice of $\psi$ (see \cite[Proposition 5.1(1)]{GGP12}). We recall the definition of the local root number $\epsilon(\varphi,\psi)$ for a representation $\varphi$ of the Weil-Deligne group $W_F\times \SL_2(\BC)$.

\begin{defn}\cite[\S 5]{GGP12}
For a representation $\varphi$ of $W_F \times \SL_2(\BC)$, we decompose
\[ \varphi= \sum_{i \geq 1} \rho_i \otimes S_{i}.  \]
Then we define
\[ \epsilon(\varphi, \psi):= \prod_{i \geq 1} \epsilon_L(\rho_i,\psi)^{i} \det(-Fr|_{\rho_i^{I}})^{n-1}, \]
where $\epsilon_L(\rho_i, \psi)$ is the local constant in \cite[(3.6.1)]{Tat79}, and $Fr$ is the Frobenius element in $W_F$.
\end{defn}

For each representation $\varphi$ of $W_F \times \SL_2(\BC)$, we may associate a representation $\lambda_{\varphi}$ of $W_F$ by
\[ \lambda_\varphi (w)= \varphi \left( w, \begin{pmatrix} |w|^{1/2} & \\ & |w|^{-1/2} \end{pmatrix}  \right). \]
When $\varphi$ is an $L$-parameter of a group $\RG(V)$, then $\lambda_{\varphi}$ is an \emph{infinitesimal parameter}. We compare $\epsilon(\varphi, \psi)$ and $\epsilon(\lambda_{\varphi}\otimes S_1, \psi)$ when $\varphi$ is orthogonal in the following lemma.

\begin{lem}\label{lem epsilon}
Suppose $\varphi$ is an orthogonal representation of $W_F \times \SL_2(\BC)$. Then we have 
\[ \epsilon(\varphi,\psi)= \epsilon(\lambda_{\varphi}\otimes S_1, \psi). \]
\end{lem}

\begin{proof}
 If $\rho \otimes S_{n}$ is an irreducible subrepresentation of $\varphi$, then either $\rho\otimes S_n$ is orthogonal and self-dual, or $\rho^{\vee} \otimes S_n$ is also a subrepresentation of $\varphi$. Therefore, we may decompose $\varphi$ as follows
 \[ \varphi= \bigoplus_{i \in I_{bp}} (\rho_i  \otimes S_{n_i} \oplus \rho_i^{\vee} \otimes S_{n_i}) \bigoplus_{i \in I_{gp}} \rho_i \otimes S_{n_i}, \]
 where $\rho_i \otimes S_{n_i}$ are all self-dual and orthogonal for all $i \in I_{gp}$.  As a consequence, it suffices to consider the following cases.
 \begin{enumerate}
     \item [(a)] $\varphi= \rho\otimes S_n$.
     \item [(b)] $\varphi=\rho  \otimes S_{n} \oplus \rho^{\vee} \otimes S_{n}$.
 \end{enumerate}
 For Case (a), we have by definition that
\[ \epsilon(\rho \otimes S_{n}, \psi)= \epsilon_L(\rho,\psi)^{n} \det( -Fr |_{\rho^{I}} )^{n-1}=\epsilon(\rho,\psi)^{n}. \]
The last equality follows from the fact that $n-1$ is odd if and only if $\rho$ is symplectic, and hence $ \det(-Fr|_{\rho^I})=1$.

On the other hand, we have
\[ \lambda_{\rho \otimes S_n}= \rho|\cdot|^{\frac{n-1}{2}}\oplus \rho|\cdot|^{\frac{n-3}{2}} \oplus \cdots \oplus \rho |\cdot|^{\frac{1-n}{2}}, \]
and hence
\begin{align*}
    \epsilon( \lambda_{\rho \otimes S_n}\otimes S_1, \psi)= \epsilon(\rho,\psi)^{n} 
\end{align*}
based on the general fact that 
\begin{align}\label{eq epsilon_L}
     \epsilon_L(\rho |\cdot|^{s},\psi) \epsilon_L(\rho|\cdot|^{-s},\psi)= \epsilon_L(\rho,\psi)^2
\end{align}
for general representation $\rho$ and $s \in \BC$ (see \cite[(3.6.5)]{Tat79}). This proves Case (a). Case (b) follows from the same computation using \eqref{eq epsilon_L} and
\[ \det(-Fr|_{\rho^{I} \oplus (\rho^{\vee})^{I}})=1. \]
This completes the proof of the lemma.
\end{proof}

\begin{prop}\label{prop central characters}
Let $\RG$ be the group $\Sp(V)$ or $\SO(V)$, where $\dim(V)$ is even. Suppose $\varphi_1$ and $\varphi_2$ are $L$-parameters $G$ such that $\lambda_{\varphi_1}=\lambda_{\varphi_2}$. Then we have $\omega_{\varphi_1}= \omega_{\varphi_2}$. In particular, all representations in any given local Arthur packet of $\RG$ share the same central character. 
\end{prop}
\begin{proof}
In either case, $\varphi_1, \varphi_2$ are orthogonal representations of $W_F \times \SL_2(\BC)$. Moreover, when $\RG=\SO(V)$, we have $\det (\varphi_1)= \det(\varphi_2)$, which is determined by $\disc(V)$ (See \cite[Theorem 8.1 (i)]{GGP12}). Therefore, Proposition \ref{prop GGP12} implies that $\omega_{\varphi_1}=\omega_{\varphi_2}$ if and only if $\epsilon(\varphi_1,\psi)= \epsilon(\varphi_2, \psi)$, which is shown by Lemma \ref{lem epsilon}.

Since M{\oe}glin proved that all representations in a fixed local Arthur packet share the same infinitesimal character (see \cite[Proposition 4.1]{Moe09b}), the second part follows. This completes the proof of the proposition.
\end{proof}


\begin{thebibliography}{dihuajiang}
\bibitem[Ach03]{Ach03}
P. Achar,
An order-reversing duality map for conjugacy classes in Lusztig's canonical quotient.
{\it Transform. Groups} \textbf{8} (2003), no. 2, 107--145.
begin 

\bibitem[AC26]{AC26}
{H. Atobe and D. Ciubotaru,
Endoscopic transfer and the wavefront upper bound conjecture.
Preprint. 2026.
} 

\bibitem[Art13]{Art13}
J. Arthur,
{\it The endoscopic classification of representations: Orthogonal and Symplectic groups.}
{\it Colloquium Publication}
Vol. \textbf{61}, 2013,
American Mathematical Society.

\bibitem[Ato20]{Ato20} H. Atobe, Construction of local A-packets. {\em J. Reine Angew. Math.}. \textbf{790} pp. 1-51 (2022).

\bibitem[Ato23]{Ato23} H. Atobe, The set of local A-packets containing a given representation. {\em J. Reine Angew. Math.}. \textbf{804} pp. 263-286 (2023).

\bibitem[AM20]{AM20} H. Atobe and A. M{\'i}nguez, The explicit Zelevinsky-Aubert duality. {\em Compos. Math.}. \textbf{159}, 380-418 (2023).


\bibitem[BV85]{BV85}
D. Barbasch and D. Vogan,
 Unipotent representations of complex semisimple groups.
{\it Ann. of Math.} (2) \textbf{121} (1985), no. 1, 41--110.

\bibitem[BZ77]{BZ77}
I.N. Bernstein and A.V.  Zelevinski, 
 Induced representations of reductive p-adic groups I.
 {\it Ann. Scient. Ec. Norm. Sup.}, 10(4):441--472, 1977.



\bibitem[CFK18]{CFK18}
Y. Cai, S. Friedberg and E. Kaplan,
 Doubling constructions: local and global theory, with an application to global functoriality for non-generic cuspidal representations. 
Preprint.2018. arXiv:1802.02637.


\bibitem[CMBO21]{CMBO21} D. Ciubotaru, L. Mason-Brown, and E. Okada,
Wavefront Sets of Unipotent Representations of Reductive p-adic Groups I. American Journal of Mathematics (to appear).

\bibitem[CMBO24]{CMBO24} D. Ciubotaru, L. Mason-Brown, and E. Okada, Some unipotent Arthur
packets for reductive p-adic groups. {\em Int. Math. Res. Not. IMRN.} IMRN 2024, no. 9, 7502-7525.

\bibitem[CMBO25]{CMBO25} D. Ciubotaru, L. Mason-Brown, and E. Okada,
Wavefront Sets of Unipotent Representations of Reductive p-adic Groups II. {\em J. Reine Angew. Math.} {\bf 823} (2025), 191--253.

\bibitem[Clo87]{Clo87}
L. Clozel, 
 Characters of nonconnected, reductive p-adic groups.
{\it Canad. J. Math.} 39 (1987), no. 1, 149--167.

\bibitem[CKPSS04] {CKPSS04}
J. Cogdell, H. Kim, I. Piatetski-Shapiro and F. Shahidi,
 Functoriality for the classical groups.
{\it Publ. Math. Inst. Hautes Etudes Aci.} No. 99 (2004),
163--233.

\bibitem[CPSS11] {CPSS11}
J. Cogdell, I. Piatetski-Shapiro and F. Shahidi,
 {\it Functoriality for the quasisplit classical groups}. {\it On certain L-functions}, 117--140, Clay Math. Proc., 13, Amer. Math. Soc., Providence, RI, 2011.

\bibitem[CM93]{CM93}
D. Collingwood and W. McGovern,
{\it Nilpotent orbits in semisimple Lie algebras.}
Van Nostrand Reinhold Mathematics Series. Van Nostrand Reinhold Co., New York, 1993. xiv+186 pp.


\bibitem[GGP12]{GGP12} Gan, W., Gross, B. \& Prasad, D. Symplectic local root numbers, central critical L values, and restriction problems in the representation theory of classical groups. {\em Astérisque}. pp. 1-109 (2012), Sur les conjectures de Gross et Prasad.

\bibitem[GGP20]{GGP20} W. T. Gan, B. H. Gross and D. Prasad, Branching laws for classical groups: The non-tempered case, Compos. Math. 156(11) (2020), 2298–2367.





%







\bibitem[GRS11]{GRS11}
D. Ginzburg, S. Rallis and D. Soudry,
{\it The descent map from automorphic representations of {${\rm GL}(n)$} to classical groups.} World Scientific, Singapore, 2011. v+339 pp.

\bibitem[GGS17]{GGS17}
R. Gomez, D. Gourevitch and S. Sahi,
 Generalized and degenerate Whittaker models.
{\it Compositio Math.} \textbf{153} (2017) 223–-256.

\bibitem[GGS21]{GGS21}
R. Gomez, D. Gourevitch and S. Sahi,
 Whittaker supports for representations of reductive groups. 
{\it Annales de l'Institut Fourier}, Volume 71 (2021) no. 1, pp. 239--286.

\bibitem[GR10]{GR10}Gross, B. \& Reeder, M. Arithmetic invariants of discrete Langlands parameters. {\em Duke Math. J.} \textbf{154}, 431--508 (2010).

\bibitem[HC78]{HC78}
Harish-Chandra, 
{\it Admissible invariant distributions on reductive p-adic groups.} Lie theories and their applications (Proc. Ann. Sem. Canad. Math. Congr., Queen's Univ., Kingston, Ont., 1977), pp. 281--347. Queen's Papers in Pure Appl. Math., No. 48, Queen's Univ., Kingston, Ont., 1978.

\bibitem[HLL22]{HLL22}
A. Hazeltine, B. Liu, and C. Lo,
On the intersection of local Arthur packets for classical groups.
Preprint. 2022. arXiv:2201.10539.

\bibitem[HLLS24]{HLLS24}
A. Hazeltine, B. Liu, C. Lo, and F. Shahidi,
On the upper bound of wavefront sets of representations of $p$-adic groups 
Submitted. 2024.  

\bibitem[HLLZ22]{HLLZ22} A. Hazeltine, B. Liu, C. Lo, and Q. Zhang,
The closure ordering conjecture on local $L$-parameters in local Arthur packets of classical groups. (2022), Preprint.


\bibitem[JL14]{JL14}
C. Jantzen and B. Liu,
 The generic dual of p-adic split $\mathrm{SO}_{2n}$ and local Langlands parameters. 
{\it Israel J. Math.} 204 (2014), no. 1, 199--260.

\bibitem[JL22]{JL22}
C. Jantzen and B. Liu,
 The generic dual of p-adic groups and Local Langlands parameters.
Preprint. 2022.  

\bibitem[Jia14]{Jia14}
D. Jiang,
{\it Automorphic Integral transforms for classical groups I: endoscopy correspondences}.
{\it Automorphic Forms: L-functions and related geometry: assessing the legacy of I.I. Piatetski-Shapiro}, 179--242,
{\it Comtemp. Math.} \textbf{614}, 2014, AMS.




\bibitem[JLS16]{JLS16}
D. Jiang, B. Liu and G. Savin,
Raising nilpotent orbits in wave-front sets.
{\it Representation Theory} 20 (2016), 419--450.


\bibitem[JNQ10]{JNQ10}
D. Jiang, C. Nien and Y. Qin,
Symplectic supercuspidal representations of $GL(2n)$ over $p$-adic fields. 
{\it Pacific J. Math.} 245 (2010), no. 2, 273--313. 

\bibitem[JS04]{JS04}
D. Jiang and D. Soudry,
{\it Generic representations and local Langlands reciprocity law for p-adic $\SO_{2n+1}$.} {\it Contributions to automorphic forms, geometry, and number theory}, 457--519, Johns Hopkins Univ. Press, Baltimore, MD, 2004.

\bibitem[JS12]{JS12}
D. Jiang and D. Soudry,
Appendix: On the local descent from GL(n) to classical groups [appendix to MR2931222]. {\it Amer. J. Math.} 134 (2012), no. 3, 767--772.





\bibitem[KK04]{KK04}
H. H. Kim and M. Krishnamurthy, 
{\it Base change lift for odd unitary groups}, Functional analysis VIII, Various Publ. Ser. (Aarhus), vol. 47, Aarhus Univ., Aarhus, 2004, pp. 116--125.

\bibitem[KK05]{KK05}
H. H. Kim and M. Krishnamurthy, 
 Stable base change lift from unitary groups to GLn,
{\it IMRP Int. Math. Res. Pap.} 1 (2005), 1--52.

\bibitem[Kon02]{Kon02}
T. Konno, 
 Twisted endoscopy and the generic packet conjecture.
{\it Israel J. Math.} 129 (2002), 253--289. 

\bibitem[KS99]{KS99}
R. E. Kottwitz and D. Shelstad, 
{\it Foundations of twisted endoscopy.}
Ast{\'e}risque No. 255 (1999), vi+190 pp.








\bibitem[LS87]{LS87}
R. Langlands and D. Shelstad,
On the definition of transfer factors.
{\it Math. Ann.} 278 (1987), 219-271. 

\bibitem[Liu11]{Liu11}
B. Liu,
Genericity of representations of p-adic $Sp_{2n}$ and local Langlands parameters.
{\it Canad. J. Math.} \textbf{63} (2011), 1107--1136.


\bibitem[Moe96]{Moe96}
C. M{\oe}glin, 
Front d\'onde de\'s repr\'esentations des groupes classiques p-adiques (French,
with French summary), {\it Amer. J. Math.} 118 (1996), no. 6, 1313--1346.






\bibitem[Moe06a]{Moe06a}
{C. M{\oe}glin}, 
Paquets d'Arthur pour les groupes classiques; point de vue combinatoire.
arXiv:math/0610189v1. 

\bibitem[Moe06b]{Moe06b}
{C. M{\oe}glin}, 
Sur certains paquets d'Arthur et involution d'Aubert-Schneider-Stuhler g{\'e}n{\'e}ralis{\'e}e. 
\emph{Represent. Theory}  {\bf10}, (2006), 86--129.

\bibitem[Moe09]{Moe09}
{C. M{\oe}glin}, 
\emph{Paquets d'Arthur discrets pour un groupe classique $p$-adique.} 
{\it Automorphic forms and L-functions II. Local aspects}, 179--257, 
\emph{Contemp. Math.}, {\bf489}, \emph{Israel Math. Conf. Proc.}, 
{\it Amer. Math. Soc., Providence, RI}, 2009.


\bibitem[M{\oe}09b]{Moe09b}
C. M{\oe}glin, Comparaison des paramètres de Langlands et des exposants à l'intérieur d'un paquet d'Arthur. {\em J. Lie Theory}. \textbf{19}, 797-840 (2009).



\bibitem[Moe10]{Moe10}
{C. M{\oe}glin}, 
Holomorphie des op{\'e}rateurs d'entrelacement normalis{\'e}s {\`a} l'aide des param{\`e}tres d'Arthur. 
\emph{Canad. J. Math.} {\bf62} (2010), no.~6, 1340--1386.

\bibitem[Moe11]{Moe11}
{C. M{\oe}glin}, 
\emph{Multiplicit{\'e} $1$ dans les paquets d'Arthur aux places $p$-adiques.} 
\emph{On certain $L$-functions}, 333--374, 
\emph{Clay Math.~Proc.,} {\bf13}, {\it Amer.~Math.~Soc., Providence, RI,} 2011.

\bibitem[MW87]{MW87}
C. M{\oe}glin and J.-P. Waldspurger,
 Modèles de Whittaker dégénérés pour des groupes p-adiques.
{\it Math. Z.} \textbf{196} (1987), no. 3, 427--452.



\bibitem[Mok15]{Mok15}
C. Mok,
{\it Endoscopic classification of representations of quasi-split unitary groups.} 
Mem. Amer. Math. Soc. 235 (2015), no. 1108, vi+248 pp. 







\bibitem[Oka21]{Oka21}
E. T. Okada,
The wavefront set of spherical Arthur
representations. Preprint. 2021. arXiv:2107.10591. 




\bibitem[Sha90]{Sha90}
F. Shahidi,
A proof of Langlands' conjecture on Plancherel measures; complementary series for p-adic groups.
{\it Ann. of Math.} (2) 132 (1990), no. 2, 273--330.


\bibitem[Sha11]{Sha11}
F. Shahidi, 
Arthur packets and the Ramanujan conjecture. {\it Kyoto J. Math.} 51 (2011), no. 1, 1–23.


\bibitem[ST15]{ST15}
D. Soudry and Y. Tanay,
On local descent for unitary groups.
{\it J. Number Theory} 146 (2015), 557--626.



\bibitem[Tat79] {Tat79} 
J. Tate,
{\it Number theoretic background}, Automorphic
forms, representations and $L$-functions, Proc. Sympos. Pure Math.,
33, Part 2, pp. 3--26, Amer. Math. Soc., Providence, R. I., 1979.

\bibitem[Tsa24]{Tsa24} Cheng-Chiang Tsai. Geometric wave-front set may not be a singleton. J. Amer. Math. Soc.   37 (2024), no. 1, 281--304.

\bibitem[Var14]{Var14}
S. Varma,
On a result of Moeglin and Waldspurger in residual characteristic 2, 
{\it Math. Z.} 277, no. 3--4,
1027--1048 (2014).

\bibitem[Var17]{Var17}
S. Varma,
On descent and the generic packet conjecture. 
{\it Forum Math.} 29 (2017), no. 1, 111--155. 

\bibitem[Wal01]{Wal01}
J.-L. Waldspurger,
{\it Int\'egrales orbitales nilpotentes et endoscopie pour les groupes classiques non ramifi\'es}. Ast\'erisque \textbf{269}, 2001.

\bibitem[Wal18]{Wal18}
J.-L. Waldspurger,
Repr\'esentations de r\'eduction unipotente pour $SO_{2n+1}$, III: exemples de fronts d\'onde. 
{\it Algebra Number Theory} 12.5 (2018),
pp. 1107--1171.

\bibitem[Wal20]{Wal20}
J.-L. Waldspurger, 
Fronts d\'onde des repr\'esentations temp\'er\'ees et de
r\'eduction unipotente pour SO(2n + 1). {\it Tunis. J. Math.}, 2(1):43--95, 2020.

\bibitem[Xu16]{Xu16} B. Xu,
On a lifting problem of L-packets. 
{\it Compos. Math.} 152 (2016), no. 9, 1800–1850.

\bibitem[Xu17]{Xu17} B. Xu, On Mœglin's parametrization of Arthur packets for p-adic quasisplit Sp(N) and SO(N). {\em Canad. J. Math.}. \textbf{69}, 890-960 (2017).




\end{thebibliography}
\end{document}